\documentclass[10pt]{amsart}

\usepackage{graphicx}
\usepackage{amsmath,amssymb,amsfonts}
\usepackage{amsthm}
\usepackage{appendix}
\usepackage{xcolor}
\usepackage{physics}
\usepackage{enumerate}
\usepackage{bbm}
\usepackage[a4paper,margin=3cm]{geometry}
\usepackage[colorlinks=true,citecolor=magenta,linkcolor=blue]{hyperref}
\usepackage{fancyhdr}
\numberwithin{equation}{section}
\theoremstyle{plain}
\newtheorem{theorem}{Theorem}[section]

\newtheorem{proposition}[theorem]{Proposition}
\newtheorem{lemma}[theorem]{Lemma}

\newtheorem{corollary}[theorem]{Corollary}

\theoremstyle{definition}
\newtheorem{defn}[theorem]{Definition}

\theoremstyle{remark}
\newtheorem{remark}[theorem]{Remark}

\DeclareMathOperator{\diam}{\mathrm{diam}}
\DeclareMathOperator{\supp}{\mathrm{supp}}
\DeclareMathOperator{\vari}{\mathrm{var}}

\newcommand{\N}{\mathbb{N}}
\newcommand{\R}{\mathbb{R}}
\newcommand{\Z}{\mathbb{Z}}

\newcommand{\opera}{\mathcal{L}}
\newcommand{\ind}{\mathbbm{1}}
\newcommand{\prob}{\mathbb{P}}
\newcommand{\probmeas}{\mathcal{P}}
\newcommand{\F}{\mathcal{F}}
\newcommand{\B}{\mathcal{B}}
\newcommand{\parti}{\mathcal{Z}}
\newcommand{\supnorm}[1]{\Vert #1 \Vert_{\infty}}
\newcommand{\bvnorm}[1]{\Vert #1 \Vert_{BV}}

\newcommand{\normopera}{\tilde{\opera}}
\newcommand{\fullnormopera}{\hat{\opera}}
\newcommand{\pres}[1]{\mathcal{E}P(#1)}
\newcommand{\M}{\mathcal{M}}

\begin{document}
\title{Invariant and conditionally invariant measures for random open interval maps with countably many branches}
\author{Cunyi Nan}
\date{\today}
\address{Graduate School of Mathematics, Nagoya University,
Furocho, Chikusa-ku, Nagoya, 464-8602, Japan}
\email{nan.cunyi.x5@s.mail.nagoya-u.ac.jp}
\begin{abstract}
    In this paper, building on previous work, we extend the thermodynamic formalism for random open dynamical systems generated by piecewise monotone interval maps with countably many branches. Under summable and contracting assumptions on the potential, we establish the Ruelle-Perron-Frobenius type theorem for the associated random open operator and prove exponential decay of correlations. In addition, we investigate the escape rate for the hole and conditionally invariant measure for the open system.
\end{abstract}
\maketitle
\tableofcontents
\section{Introduction}
The study of open dynamical systems starts by introducing a \emph{hole} into a designated closed system; the state points are allowed to escape through this specified region under iterates of maps, which models loss of mass or energy in physics. G. Pianigiani and J. A. Yorke started the study of expanding maps with holes in \cite{pianigiani}; they showed that under suitable conditions there exists a \emph{conditionally invariant probability measure} for this system, which is regarded as the initial work on open systems. The subsequent studies of hyperbolic systems with holes are abundant, for instance, \cite{collet2} on the limit Cantor sets in expanding systems; \cite{collet1} on subshifts of finite type; \cite{chernov1,chernov2} on Anosov diffeomorphisms; and \cite{liverani4} on Lasota-Yorke maps. We refer to studies of one-dimensional systems without Markov partitions in \cite{chernov3,demers2}. C. Bose \cite{bose} revisited Lasota-Yorke maps with holes by Ulam's method. Regarding symbolic dynamics, constructing cylindrical holes in topological Markov shifts with countable alphabet is carried out in papers \cite{demers3,tanaka} and in the book \cite{poll} for graph directed Markov systems. The conditionally invariant probability measures describe the escape of mass or particles from the system when conditioned on holes, while the \emph{escape rate} describes the speed at which orbits of points first enter an area of the phase space by quantifying the size of survivor sets; see \cite{demers} for relations with pressures and Lyapunov exponents.

Random dynamical systems are characterized by a base system $(\Omega,\F,\prob,\theta)$ where $(\Omega,\F,\prob)$ is a complete probability space and $\theta:\Omega\to\Omega$ is an invertible ergodic $\prob$-preserving map. The map between fiber spaces depends on the random state $\omega\in\Omega$; it replaces the iterates of a fixed map in a deterministic dynamical system with random composition. The potential function also becomes random, leading to a family of random transfer operators acting on fiber spaces. Each operator describes the evolution of observables along one realization of random dynamics. For a systematic description, we suggest \cite{gundlach,kifer}.

We say a map $T:I\to I$ on a compact interval $I$ is a \emph{piecewise monotone interval map} with finitely or countably many branches if there exists a finite or countable partition of $I$ into intervals such that the map is continuous, strictly monotone, and maps intervals to intervals for every element in the partition. The invariant measures and decomposition of transfer operators for this system are carried out in \cite{hofbauer2,baladi}. Using an effective method of \emph{Birkhoff cone technique} from \cite{birkhoff1}, C. Liverani et al. \cite{liverani3} obtained density functions, invariant measures, and decay of correlations for a covering system under the assumption of contracting potential. The key step is to define the functional and the related Birkhoff cones. One common way to construct the desired density functions is to show that the Birkhoff cone is strongly contracted after sufficiently many iterates of operators such that the Hilbert metric of any two elements in the cone tends to zero exponentially fast. The previously defined functional is identified with the invariant measure by Riesz representation theorem. However, such procedures are not directly applicable to random dynamical systems. J. Buzzi's two papers \cite{buzzi2,buzzi1} pointed out the cones on some fibers may not enjoy strict contraction under iterates of operators; such random states are called bad fibers. To overcome this difficulty, the bad fibers with their orbits are grouped together in the \emph{blocks} to control the upper bound of coefficient in the Lasota-Yorke inequality. Once the total length of bad blocks is small enough, the contraction and invariance of cones follow.

In the style of J. Buzzi, J. Atnip et al. \cite{atnip1} considered random dynamical systems of piecewise monotone interval maps with countably many branches, combined with a random covering condition. It is natural to consider introducing holes into these random systems; the fiberwise holes are determined by the canonical projection of the hole subset onto the fiber spaces. J. Atnip et al. \cite{atnip2} dealt with non-transitive random dynamical systems with holes under a strongly contracting potential condition. In \cite{atnip3}, the theory of thermodynamic formalism for piecewise monotone random interval maps with finite branches and holes is developed. 

We state what is different and new in this paper compared with the finite-branch case. First, the summability hypothesis on the potential is essential to control the tail distributions from countably many branches and to ensure that the sum of the weight function is convergent, which is not mentioned in \cite{atnip3}. Under this condition, the associated random Ruelle-Perron-Frobenius operator is well defined on the Banach space of bounded variation functions. We define the random functional related to the random open operator and identify it with a Borel probability measure on the fiber space in the following sections. Second, specific partitions of the interval are chosen to control the variation of the weight function and its iterates, in order to ensure the coefficients in the Lasota-Yorke type inequality do not blow up to infinity. Due to the existence of holes, we divide the coarsest finite partition into good and bad subcollections according to the values of the random functional on the indicator function of each element. Third, we complete the second modified Lasota-Yorke inequality in the sense of J. Buzzi; see Proposition \ref{prop6.13}. The second term on the right-hand side related to the defined random functional is first provided in the finite-branch case \cite{atnip3} but without proof; we complete the details in the countable-branch case. Some auxiliary lemmas previously stated in the finite-branch setting are adapted to the countable framework and proved in full detail. Finally, although it is subtle, we relax the definitions of good fibers in \cite{buzzi2,atnip3} to a more flexible pair of parameters, which leads to changes in the Birkhoff cone parameters and the corresponding propositions. As a consequence, we obtain a Ruelle-Perron-Frobenius type theorem for the random open system, including the existence and uniqueness of a strictly positive random density function and the construction of invariant measures. At the end of the paper, we calculate the escape rate and a conditionally invariant measure characterizing the random open system.
\subsection{Statement of the main result}
We record the main results of the paper in this subsection. We begin with the setting of piecewise monotone interval maps. Let $I\subseteq\R$ be a compact interval with its Borel $\sigma$-algebra $\B$. Let $(\Omega,\F,\prob)$ be a complete probability space, and let $\theta:\Omega\to\Omega$ be an invertible ergodic map preserving the probability measure $\prob$; we say the quadruple $(\Omega,\F,\prob,\theta)$ is a \emph{base system}. For each $\omega\in\Omega$, let $T_{\omega}:I\to I$ be a \emph{piecewise monotone random interval map} with countably many branches. We remove the singular set $S_\omega$ of map $T_\omega$ from the whole interval to make it a fiber space $X_\omega$, and then we obtain the family of fiber spaces $\{X_\omega\}_{\omega\in\Omega}$. The global phase space is denoted by $X:=\bigcup_{\omega\in\Omega}\{\omega\}\times X_\omega$, and we study the dynamics of $T_{\omega}|_{X_\omega}:X_{\omega}\to X_{\theta\omega}$ at each state. We define the bounded functions and bounded variation functions on $I$ by $B(I)$ and $BV(I)$, respectively. Given a function $f:\Omega\times I\to\R$, we say $f\in BV_{\Omega}(I)$ if its $\omega$-section $f_{\omega}(\cdot):=f(\omega,\cdot):I\to\R$ has bounded variation for every $\omega\in\Omega$.

Let $H\subseteq X$ be a hole, a measurable subset with respect to the product $\sigma$-algebra $\F\otimes\B$ on the global phase space $\Omega\times I$. Its family of projections onto $I$ is denoted by $\{H_\omega\}_{\omega\in\Omega}$. Let $\varphi_{c}:\Omega\times I\to\R\cup\{-\infty\}$ be a summable contracting potential for the open system, and let $\opera_{\omega}:BV(X_\omega)\to BV(X_{\theta\omega})$ be the associated open operator defined by $\opera_{\omega}f:=\opera_{\omega,c}(f\cdot\ind_{\omega})$ where $\opera_{\omega,c}$ is the closed random Ruelle-Perron-Frobenius operator and $\ind_\omega$ is the indicator function of $J_\omega:=X_{\omega}\setminus H_{\omega}$; see Section \ref{sec2} for more details. For every $n\geq0$, define $K_{\omega,n}:=\{x\in J_{\omega}:T_{\omega}^{i}(x)\in J_{\theta^{i}\omega},0\leq i\leq n\}$ to be the $(\omega,n)$-\emph{survivor set}. The global $(\omega,\infty)$-survivor set is denoted by $K_{\infty}$.

The first main theorem focuses on the density function and the unique invariant measure supported in survivor sets with respect to the family of open operators $\{\opera_{\omega}\}_{\omega\in\Omega}$.
Let $\probmeas_{\Omega}(\Omega\times I)$ denote the set of all random probability measures on $\Omega\times I$; the disintegration of Borel probability measures can thereafter be used to construct the unique random probability measure for the open system. The proof is provided in Section \ref{sec8}.
\begin{theorem}\label{main1}
    (1) There exists a unique random probability measure $\nu\in\probmeas_{\Omega}(\Omega\times I)$ with $\supp(\nu)\subseteq K_{\infty}$ and disintegration $\{\nu_\omega\}_{\omega\in\Omega}$ such that \[\nu_{\theta\omega}(\opera_\omega f)=\lambda_{\omega}\nu_{\omega}(f)\] for all $f\in BV(I)$, where $\lambda_\omega:=\nu_{\theta\omega}(\opera_{\omega}\ind_\omega)$ is a positive number and $\log\lambda_\omega\in L^1(\prob)$ for $\prob$-a.e. $\omega\in\Omega$.

    (2) There exists a unique positive function $q\in BV_{\Omega}(I)$ such that $\nu(q)=1$ and \[\opera_{\omega}q_{\omega}=\lambda_{\omega}q_{\theta\omega}\] for $\prob$-a.e. $\omega\in\Omega$.

    (3) The measure $\mu:=q\nu$ is $T$-invariant and ergodic random probability measure absolutely continuous with respect to $\nu$ and $\supp(\mu)\subseteq K_{\infty}$.
\end{theorem}
We have statistical descriptions of the open system; the second main theorem below reveals the limit of the normalized open operator, as well as an exponential rate of decay of correlations for observables in this random open system. The proof is in the end part of Section \ref{sec7}.
\begin{theorem}\label{main2}
    For all $f\in BV_{\Omega}(I)$ with $\supnorm{f_\omega},\log\vari(f_\omega),\log\inf|f_\omega|\in L^1(\prob)$ for every $\omega\in\Omega$, there exists a measurable finite function $\hat{C}_{f}:\Omega\to(0,\infty)$ and $\iota\in(0,1)$ such that for $\prob$-a.e. $\omega\in\Omega$ and all $n\geq1$, we have
    \[\supnorm{(\lambda_{\omega}^{n})^{-1}\opera_{\omega}^{n}f_{\omega}-\nu_{\omega}(f_{\omega})q_{\theta^{n}\omega}}\leq \hat{C}_{f}(\omega)\bvnorm{f_\omega}\iota^n.\]

    There exists a measurable finite function $D_{f}:\Omega\to(0,\infty)$ and $\kappa\in(\iota,1)$ such that for $\prob$-a.e. $\omega\in\Omega$, for all $n\geq1$, all $f\in BV_{\Omega}(I)$ and all $g\in L^1(\mu)$, we have
    \[|\mu_{\omega}\left(f_{\omega}(g_{\theta^{n}\omega}\circ T_{\omega}^{n})\right)-\mu_{\theta^{n}\omega}(g_{\theta^{n+p}\omega})\mu_{\omega}(f_{\omega})|\leq D_f(\omega)\bvnorm{f_{\omega}}\Vert g_{\theta^{n}\omega}\Vert_{L^1(\mu_{\theta^{n}\omega})}\kappa^n.\]
\end{theorem}
In the third main theorem, we investigate escape rate and random conditionally invariant probability measure for random open systems. Let $\varphi:=\varphi_{c}|_{J}$ be the restriction of summable contracting potential $\varphi_c$ to $J:=\bigcup_{\omega\in\Omega}\{\omega\}\times J_{\omega}$ and let $\pres{\cdot}$ denote the expected pressure of potential; see Section \ref{sec9.1}. Recall $\nu_c$ is the random probability measure for the closed system, we define the \emph{escape rate} for the hole $H$ by \[R_{\nu_{\omega,c}}(H):=-\lim_{n\to\infty}\frac{1}{n}\log\nu_{\omega,c}(K_{\omega,n})\] in terms of measure of survivor set if the supremum limit and the infimum limit coincide. We calculate it by differences of two expected pressures. We can find the unique random conditionally invariant probability measure that is absolutely continuous with respect to $\nu_c$. We provide the proof in Section \ref{sec9}.
\begin{theorem}\label{main3}
    Given a hole $H\subseteq X$, for $\prob$-a.e. $\omega\in\Omega$ we have \[R_{\nu_{\omega,c}}(H)=\pres{\varphi_c}-\pres{\varphi}.\]

    The random probability measure $\tau:=(\ind_{J}q)\nu_c$ is the unique random conditionally invariant probability measure absolutely continuous with respect to $\nu_c$ and $\supp(\tau)\subseteq J$. 
\end{theorem}
The results obtained in this paper extend the works of J. Atnip et al. to the setting of piecewise monotone random interval maps with countably many branches and holes. Our approach combines techniques developed by C. Liverani for interval maps with countably many branches \cite{liverani3}, by C. Liverani and V. Maume-Deschamps for Lasota-Yorke maps with holes \cite{liverani4}, and by J. Buzzi for the decomposition into good and bad fibers in random Lasota-Yorke maps \cite{buzzi1,buzzi2}. In this sense, the present work can be viewed as an extension of \cite{atnip1} to open systems, carried out within the framework used in \cite{liverani4} and further adapted to the random setting considered in \cite{atnip3}, with the proofs following these general schemes and necessary modifications.
\subsection{Plan of the paper}
The paper is organized as followings. We start with preliminaries and settings of piecewise monotone random interval map with countable branches in Section \ref{sec2.1}, while the function spaces, potentials and transfer operators are defined in Section \ref{sec2.2}. We introduce holes into such system and construct random functionals in Section \ref{sec2.3}. To obtain the contraction of Birkhoff cones, we prove two Lasota-Yorke type inequalities, divide good and bad fibers, and find their finite diameters in Section \ref{sec4} and Section \ref{sec5}. The density functions and invariant measures are constructed in Section \ref{sec6}, the exponential rate of decay of correlations is obtained in Section \ref{sec7}. The invariant measures previously constructed can be raised to random probability measures in \ref{sec8}. We calculate and compare expected pressures in the closed and open systems, we end this paper with quantifying escape rate and random conditionally invariant probability measure in Section \ref{sec9}.
\section{Piecewise monotone random interval maps with holes}\label{sec2}
\subsection{Settings}\label{sec2.1}
Let $(\Omega,\F,\prob)$ be a complete probability space, and let $\theta:\Omega\to\Omega$ be an invertible ergodic map preserving probability measure $\prob$. Let $I\subseteq\R$ be a compact interval; we denote its Borel $\sigma$-algebra by $\B$. For each $\omega\in\Omega$ consider a surjective map $T_\omega:I\to I$ with a countable partition $\parti_{\omega}^*=\{Z_{\omega,i}\}_{i=1}^{\infty}$ of $I$ such that each $Z_{\omega,i}$ is an interval and $I\setminus\bigcup_{i=1}^{\infty}Z_{\omega,i}$ has zero Lebesgue measure. For each $i\geq1$, $T_{\omega}|_{Z_{\omega,i}}$ can be extended to a homeomorphism $T_{\omega,i}:=T_{\omega}|_{Z_{\omega,i}}$ from the interval $Z_{\omega,i}$ onto its image which is an interval. We say the family of interval maps $\{T_{\omega}\}_{\omega\in\Omega}$ is \emph{piecewise monotone random interval map with countably many branches} provided $T_{\omega,i}$ is continuous and strictly monotone on each $Z_{\omega,i}\in\parti_{\omega}^*$. For $n\geq1$, we write the $n$-fold composition of $T_{\omega}$ as $T_{\omega}^{n}:=T_{\theta^{n-1}\omega}\circ\cdots\circ T_{\omega}$ and its inverse as $T_{\omega}^{-n}:=(T_{\omega}^{n})^{-1}$.

Let $D_\omega$ be the set of discontinuous points of $T_\omega$, that is, the collection of all endpoints of intervals $\{Z_{\omega,i}\}_{i=1}^{\infty}$. The singular set $S_\omega$ of $T_\omega$ is defined by $$S_\omega:=\bigcup_{k=0}^{\infty}T_{\omega}^{-k}(D_{\theta^{k}\omega})=D_{\omega}\cup T_{\omega}^{-1}D_{\theta\omega}\cup\cdots\cup T_{\omega}^{-n}D_{\theta^{n}\omega}\cup\cdots,$$ it is backward $T_\omega$-invariant and countable. Indeed, $S_{\omega}$ is countable since $D_{\theta^{k}\omega}$ is countable for every $k\geq0$. To show it is backward $T_\omega$-invariant, we observe that $$T_{\omega}^{-1}S_{\theta\omega}=T_{\omega}^{-1}\left(D_{\theta\omega}\cup T_{\theta\omega}^{-1}D_{\theta^2 \omega}\cup\cdots\right)\subseteq T_{\omega}^{-1}D_{\theta\omega}\cup T_{\omega}^{-2}D_{\theta^{2}\omega}\cup\cdots\subseteq S_\omega.$$ For each $\omega\in\Omega$, we assume the partition $\parti_{\omega}^{*}$ is a \emph{generator}, which means that $S_\omega$ is dense in $I$; see \cite{hofbauer}. 

Let $X_{\omega}:=I\setminus S_{\omega}$ be the subspace that inherits the subspace topology of $I$. By an interval $Z\subseteq X_\omega$ we mean $Z=\tilde{Z}\cap X_\omega$ for some intervals $\tilde{Z}\subseteq I$. We have $T_{\omega}(X_\omega)\subseteq X_{\theta\omega}$, which allows us to consider the dynamics of $T_{\omega}|_{X_\omega}:X_\omega\to X_{\theta\omega}$.

The skew-product map $T:\Omega\times I\to\Omega\times I$ is given by $$T(\omega,x):=(\theta\omega,T_{\omega}x),$$ then $T^{n}(\omega,x)=(\theta^{n}\omega,T_{\omega}^{n}x)$. We assume the map $T$ is measurable with respect to the product $\sigma$-algebra $\F\otimes\B$. Let $X:=\bigcup_{\omega\in\Omega}\{\omega\}\times X_\omega\subseteq\Omega\times I$ be the global phase space of the random dynamical system.

For each $\omega\in\Omega$, let $\parti_{\omega}^{1}:=\{Z\cap X_\omega:Z\in\parti_\omega^{*}\}$ be the collection of nonempty open intervals in $I$, and let $$\parti_{\omega}^{n}:=\bigvee_{i=0}^{n-1}T_{\omega}^{-i}(\parti_{\theta^{i}\omega}^{1})=\parti_{\omega}^{1}\vee T_{\omega}^{-1}\parti_{\theta\omega}^{1}\vee\cdots\vee T_{\omega}^{-(n-1)}\parti_{\theta^{n-1}\omega}^{1}$$ be given inductively for $n\geq2$, each element of $\parti_{\omega}^{n}$ would be $Z_{0}\cap T_{\omega}^{-1}Z_{1}\cap\cdots\cap T_{\omega}^{-(n-1)}Z_{n-1}$ for $Z_{i}\in\parti_{\theta^{i}\omega}^{1}$. The partition $\parti_{\omega}^{n}$ is a monotonicity partition for $T^n$ where the map $T^n$ is continuous and strictly monotone on every element. It consists of countably many non-degenerate intervals and it is obviously a countable partition of $X_\omega$, which we assume it to be generating, that is, $\bigvee_{n=1}^{\infty}\parti_{\omega}^{n}=\B$.
\subsection{Transfer operator and potential}\label{sec2.2}
The set of all bounded real-valued functions on $I$ is denoted by $B(I)$. For $f\in B(I)$, given a subset $A\subseteq I$, the variation of $f$ on subset $A$ is defined by
\[\vari_{A}(f):=\sup\left\{\sum_{i=0}^{k-1}|f(x_{i+1})-f(x_i)|:x_{0}<\cdots< x_{k},x_{i}\in A\right\},\] where the supremum is taken over all finite points $\{x_i\}_{i=0}^{k}$ in the subset $A$. When $A=I$, we omit the subscript and write it as $\vari(f)$. The set of all bounded variation functions is denoted by $BV(I):=\{f\in B(I):\vari(f)<\infty\}$. It is clear that when $f\in BV(I)$, then $f\in BV(X_\omega)$ for every $\omega\in\Omega$ since $\vari_{X_\omega}(f)\leq\vari(f)<\infty$. Let $\supnorm{\cdot}$ be the supremum norm on $B(I)$, and let $\bvnorm{\cdot}:=\supnorm{\cdot}+\vari(\cdot)$ be the $BV$-norm on $BV(I)$.

Given a function $f:\Omega\times I\to\R$, we denote $f_{\omega}(\cdot):=f(\omega,\cdot):I\to\R$ for each $\omega\in\Omega$.
\begin{defn}
    A function $f:\Omega\times I\to\R$ is \emph{random bounded} if (1) $f_\omega\in B(I)$ for each $\omega\in\Omega$; (2) for every $x\in I$ the map $\omega\mapsto f_{\omega}(x)$ is measurable; (3) the function $\omega\mapsto\supnorm{f_\omega}$ is measurable. The collection of all random bounded functions on $\Omega\times I$ is denoted by $B_{\Omega}(I)$. A function $f\in B_{\Omega}(I)$ is said to be of \emph{random bounded variation} if $f_\omega\in BV(I)$ for each $\omega\in\Omega$ and we let $BV_{\Omega}(I)$ be the collection of all random bounded variation functions.
\end{defn}
Consider a potential $\varphi_{c}:\Omega\times I\to\R\cup\{-\infty\}$ with fibers $\{\varphi_{\omega,c}\}_{\omega\in\Omega}$ on $I$, we let $g_{\omega,c}^{1}:=e^{\varphi_{\omega,c}}$ be the weight function for each $\omega\in\Omega$. We say such potential $\varphi_c$ is \emph{summable} if (1) $\inf g_{\omega,c}^{1}|_{Z}>0$ for all $Z\in\parti_{\omega}^{1}$; (2) $g_{\omega,c}^{1}\in BV(I)$; (3) $S_{\omega}^{1}:=\sum_{Z\in\parti_{\omega}^*}\sup_{Z}g_{\omega,c}^{1}<\infty$. The condition (3) allows us to define the \emph{random Ruelle-Perron-Frobenius operator} $\opera_{\omega,c}:=\opera_{\varphi,\omega,c}:B(X_\omega)\to B(X_{\theta\omega})$ associated with the summable potential $\varphi_c$ acting on bounded functions by $$\opera_{\omega,c}f(x):=\sum_{y\in T_{\omega}^{-1}x}f(y)g_{\omega,c}^{1}(y)$$ for each $f\in B(X_\omega)$ and $x\in X_{\theta\omega}$, it is well defined since the sum is convergent. In \cite[Lemma 2.4]{atnip1}, we observe that $g_{\omega,c}^{n}:=\exp(S_{n}\varphi_{\omega,c})$ is of bounded variation where $$S_{n}\varphi_{\omega,c}=\sum_{i=0}^{n-1}\varphi_{\theta^{i}\omega,c}\circ T_{\omega}^{i}$$ is the Birkhoff sum with respect to $T_\omega$, and then $S_{\omega}^{n}$ is finite. Hence the iterates $\opera_{\omega,c}^{n}:B(X_\omega)\to B(X_{\theta^{n}\omega})$ are also well defined for each $n\geq1$ and are given by $$\opera_{\omega,c}^{n}f(x)=\sum_{y\in  T_{\omega}^{-n}x}f(y)g_{\omega,c}^{n}(y).$$ 

Given a measurable space $Y$, let $\M(Y)$ denote the set of all Borel probability measures on $Y$. For each $\omega$, let $\B_\omega$ denote the Borel $\sigma$-algebra of $X_\omega$, the restriction of $\B$ to the set $X_\omega$. Let $\M_{\prob}(\Omega\times I)$ be the set of all Borel probability measures $\mu$ on $\Omega\times I$ with marginal $\prob$, that is, $\mu\circ\pi_{\Omega}^{-1}=\prob$ where $\pi_{\Omega}:\Omega\times I\to\Omega$ is the canonical projection. We take a note on a special class of measures.
\begin{defn}[{\cite[Definition 3.1]{crauel}}]\label{def2.2}
    A map $\mu:\Omega\times\B\to[0,1]$, $(\omega,B)\mapsto\mu_{\omega}(B)$ is a \emph{random probability measure} on $\Omega\times I$ if

    (1) For every $B\in\B$, the map $\omega\mapsto\mu_{\omega}(B)$ is measurable;

    (2) for $\prob$-a.e. $\omega\in\Omega$, the map $B\mapsto\mu_{\omega}(B)$ is a Borel probability measure.
\end{defn}
Let $\probmeas_{\Omega}(\Omega\times I)$ denote the set of all random probability measures on $\Omega\times I$, a random probability measure $\mu\in\probmeas_{\Omega}(\Omega\times I)$ has a disintegration of $\{\mu_{\omega}\}_{\omega\in\Omega}$ on each fiber $\omega\in\Omega$. We say $\mu$ is supported in $\Omega\times I$ if $\supp(\mu)\subseteq\Omega\times I$ and consequently $\supp(\mu_\omega)\subseteq I$ for $\prob$-a.e. $\omega\in\Omega$. Let $A\subseteq\Omega\times I$ be a measurable subset with $A_{\omega}\subseteq I$ for every $\omega\in\Omega$, we denote its measure by $\mu(A)$, while the measure of $A_\omega$ is $\mu_{\omega}(A_\omega)$.

Let $\ind_{X_\omega}$ be the constantly one function on $X_\omega$, we assume

(1) $\log S_{\omega}^{1}\in L^1(\prob)$;

(2) $\log\inf\opera_{\omega,c}\ind_{X_\omega}\in L^1(\prob)$.

Therefore, we obtain the following log-integrable quantities.
\begin{lemma}[{\cite[Lemma 2.13]{atnip1}}]
    For all $n\geq1$, we have $$\log\inf\opera_{\omega,c}^{n}\ind_{X_\omega},\log\supnorm{\opera_{\omega,c}^{n}\ind_{X_\omega}},\log S_{\omega}^{n},\log\supnorm{g_{\omega,c}^{n}}\in L^1(\prob).$$
\end{lemma}
We require the potential to satisfy the summable and contracting conditions simultaneously.
\begin{defn}[{\cite[Definition 2.15]{atnip1}}]\label{def6.4}
    A summable potential $\varphi_{c}:\Omega\times I\to\R\cup\{-\infty\}$ is said to be \emph{contracting} if for $\prob$-a.e. $\omega\in\Omega$, we have $$\lim_{n\to\infty}\frac{1}{n}\log\supnorm{g_{\omega,c}^n}<\lim_{n\to\infty}\frac{1}{n}\log\inf\opera_{\omega,c}^{n}\ind_{X_\omega}$$ and there exists $N\geq1$ such that $$-\infty<\int_{\Omega}\log\supnorm{g_{\omega,c}^N}\,\dd\prob(\omega)-\int_{\Omega}\log\inf\opera_{\omega,c}^{N}\ind_{X_\omega}\,\dd\prob(\omega)<0.$$
\end{defn}
We cite a version of the Ruelle-Perron-Frobenius theorem in piecewise monotone random interval map with countably many branches.
\begin{theorem}[{\cite[Theorem 2.19]{atnip1}}]
    Let $\varphi_{c}:\Omega\times I\to\R\cup\{-\infty\}$ be a summable contracting potential, and let $\opera_{\omega,c}:BV(X_\omega)\to BV(X_{\theta\omega})$ be the associated random Ruelle-Perron-Frobenius operator, then the following holds.

    (1) There exists a unique random probability measure $\nu_{c}\in\probmeas_{\Omega}(\Omega\times I)$ whose disintegration is $\{\nu_{\omega,c}\}_{\omega\in\Omega}$ such that $$\nu_{\theta\omega,c}(\opera_{\omega,c}f)=\lambda_{\omega,c}\nu_{\omega,c}(f)$$ for each $f\in BV(I)$, here $\lambda_{\omega,c}:=\nu_{\theta\omega,c}(\opera_{\omega,c}\ind)>0$ and $\log\lambda_{\omega,c}\in L^1(\prob)$ for $\prob$-a.e. $\omega\in\Omega$.

    (2) There exists a unique positive function $q_{c}\in BV_{\Omega}(I)$ such that $\nu_{c}(q_{c})=1$ and $$\opera_{\omega,c}q_{\omega,c}=\lambda_{\omega,c}q_{\theta\omega,c}$$ for $\prob$-a.e. $\omega\in\Omega$.

    (3) The probability measure $\mu_{c}:=q_{c}\nu_{c}$ is ergodic $T$-invariant random probability measure supported in $X$.
\end{theorem}
\subsection{Holes and random functionals}\label{sec2.3}
We want to introduce holes to this closed system to make it an open one. Let $H\subseteq X$ be a measurable subset with respect to the product $\sigma$-algebra $\F\otimes\B$. Suppose the hole $H$ is nontrivial, that is $\nu_{c}(H)\nu_{c}(X\setminus H)\neq0$. For each $\omega\in\Omega$, the subsets $H_\omega\subseteq X_\omega$ are uniquely determined by $$\{\omega\}\times H_\omega=H\cap(\{\omega\}\cap X_\omega).$$ Define $J_\omega:=X_{\omega}\setminus H_{\omega}$ and denote $\ind_\omega:=\ind_{J_\omega}$, then we define $$J:=\bigcup_{\omega\in\Omega}\{\omega\}\times J_\omega\subseteq X \subseteq\Omega\times I$$ to be the complement of the hole $H$ in $X$. For each $n\geq0$, define $$K_{\omega,n}:=\{x\in J_\omega:T_{\omega}^{i}(x)\in J_{\theta^{i}\omega}:0\leq i\leq n\}=\bigcap_{i=0}^{n}T_{\omega}^{-i}(J_{\theta^{i}\omega})$$ to be the set of points in $J_\omega$ which survive for at least $n$ iterates, meaning they do not enter into the hole after time $n$, we call such sets \emph{$(\omega,n)$-survivor sets}. It is obvious that $(\omega,0)$-survivor set is indeed $J_\omega$. It is natural to define $K_{\omega,\infty}:=\bigcap_{n\geq0}K_{\omega,n}$ since the $(\omega,n)$-survivor sets are descending. We define \[K_{\infty}:=\bigcup_{\omega\in\Omega}\{\omega\}\times K_{\omega,\infty}\] and we also assume it to be nonempty.

The \emph{random open operator} $\opera_{\omega}:B(X_\omega)\to B(X_{\theta\omega})$ associated with a summable potential $\varphi_c:\Omega\times I\to\R\cup\{-\infty\}$ is defined by $$\opera_{\omega}f:=\opera_{\omega,c}(f\ind_{\omega})$$ for each $f\in B(X_\omega)$, more precisely, \[\opera_{\omega}f(x)=\sum_{y\in T_{\omega}^{-1}x}f(y)\ind_{\omega}(y)g_{\omega,c}^{1}(y)\] for all $x\in X_{\theta\omega}$. Since $g_{\omega,c}^{n}$ is of bounded variation and $S_{\omega}^{n}$ is finite, the iterate of open transfer operator is well defined and is given by $$\opera_{\omega}^{n}f:=\opera_{\omega,c}^{n}(f\ind_{K_{\omega,n-1}}).$$ For ease of calculation, we denote $g_{\omega}^{n}:=g_{\omega,c}^{n}|_{K_{\omega,n-1}}$ and hence $g_{\omega}^{n}=\exp(S_{n}\varphi_{\omega})$ where $\varphi_{\omega}:=\varphi_{\omega,c}|_{J_\omega}$ for each $\omega\in\Omega$.

Let $S_{\omega,n}:=\{x\in X_{\omega}:\opera_{\theta^{-n}\omega}^{n}\ind_{\theta^{-n}\omega}(x)\neq0\}$ be the support of $\opera_{\theta^{-n}\omega}^{n}\ind_{\theta^{-n}\omega}$, then it is descending as $n\to\infty$ and we assume the limit set $S_{\omega,\infty}$ is nonempty. For $i\in\Z$ and $j\in\N$, we let $$S(i,j):=S_{\theta^{i}\omega,j}=\text{the support of }\opera_{\theta^{-j+i}\omega}^{j}\ind_{\theta^{-j+i}\omega}.$$ For each $n\geq1$, we have $$\inf g_{\omega,c}^{n}\leq\underset{K_{\omega,n-1}}{\inf}\,g_{\omega}^{n}\leq\supnorm{g_{\omega}^{n}}\leq\supnorm{g_{\omega,c}^{n}}$$ and as a consequence $$\underset{S(n,\infty)}{\inf}\,\opera_{\omega}^{n}\ind_{\omega}\leq\supnorm{\opera_{\omega}^{n}\ind_\omega}\leq\supnorm{\opera_{\omega,c}^{n}\ind_{X_\omega}}.$$ Therefore, we have $\log\supnorm{\opera_{\omega}^{n}\ind_\omega}\in L^1(\prob)$ by $\log\supnorm{\opera_{\omega,c}^{n}\ind_{X_\omega}}\in L^1(\prob)$.

Now we adapt from the construction of functionals from \cite{liverani3,liverani4} in deterministic systems and \cite{atnip3} in random systems. For every $\omega\in\Omega$, we define the random functional $F_\omega:BV(X_\omega)\to\R$ by $$F_{\omega}(f):=\lim_{n\to\infty}\underset{x\in S(n,n)}{\inf}\,\frac{\opera_{\omega}^{n}f(x)}{\opera_{\omega}^{n}\ind_\omega(x)}$$ for every $f\in BV(X_\omega)$. Since $\inf f\leq F_{\omega}(f)\leq\supnorm{f}$ and the sequence is monotone increasing, the limit exists and hence $F_\omega$ is well defined. We then have $F_{\omega}(\ind_\omega)=1$ and $F_{\omega}(\ind_{X_\omega})=1$ as well. By linking this functional with the variation, we obtain a useful inequality that
\[f(x)\leq\vari(f)+F_{\omega}(f)\] for all $f\in BV(X_\omega)$ and $x\in X_\omega$. Take the supremum norm on $X_\omega$, we have $\supnorm{f}\leq\vari(f)+F_{\omega}(f)$.

We set $\rho_\omega:=F_{\theta\omega}(\opera_\omega\ind_\omega)$, it immediately implies $$\underset{J_\omega}{\inf}\,g_{\omega}\leq\rho_\omega\leq\supnorm{\opera_{\omega}\ind_\omega}$$ and $\log\rho_\omega\in L^1(\prob)$ due to the result of $\log\supnorm{\opera_{\omega}\ind_{\omega}}\in L^1(\prob)$. Furthermore, by Birkhoff's ergodic theorem, we have $$\lim_{n\to\infty}\frac{1}{n}\log\rho_{\omega}^{n}=\int_{\Omega}\log\rho_{\omega}\,\dd\prob(\omega),$$ where $\rho_{\omega}^{n}:=\prod_{i=0}^{n-1}\rho_{\theta^{i}\omega}$. In the next lemma we have an easy observation of $F_{\theta^{n}\omega}$.
\begin{lemma}\label{lem6.6}
    For all $f\in BV(X_\omega)$ and all $n\geq1$, we have $$F_{\theta^{n}\omega}(\opera_{\omega}^{n}\ind_{\omega})F_{\omega}(f)\leq F_{\theta^{n}\omega}(\opera_{\omega}^{n}f),$$ and in particular, $$\rho_{\omega}^{n} F_{\omega}(f)\leq F_{\theta^{n}\omega}(\opera_{\omega}^{n}f),\quad\rho_{\omega}^{n}\leq F_{\theta^{n}\omega}(\opera_{\omega}^{n}\ind_\omega).$$
\end{lemma}
\begin{proof}
    The proof is partially referred to \cite[Lemma 3.6]{liverani4} and \cite[Lemma 1.5.4]{atnip3}. For $n,m\geq1$ and $x\in S(n+m,n)$,
    \begin{equation*}
        \begin{aligned}
            \frac{\opera_{\theta^{m}\omega}^{n}(\opera_{\omega}^{m}f)(x)}{\opera_{\theta^{m}\omega}^{n}\ind_{\theta^{m}\omega}(x)}&=\frac{\opera_{\theta^{n}\omega}^{m}(\opera_{\omega}^{n}f)(x)}{\opera_{\theta^{m}\omega}^{n}\ind_{\theta^{m}\omega}(x)} \\
            &=\frac{\opera_{\theta^{n}\omega}^{m}\left(\ind_{S(n,n)}\cdot\frac{\opera_{\omega}^{n}f}{\opera_{\omega}^{n}\ind_{\omega}}\cdot\opera_{\omega}^{n}\ind_\omega\right)(x)}{\opera_{\theta^{m}\omega}^{n}\ind_{\theta^{m}\omega}(x)} \\
            &\geq\frac{\opera_{\theta^{m}\omega}^{n}(\opera_{\omega}^{m}\ind_\omega)(x)}{\opera_{\theta^{m}\omega}^{n}\ind_{\theta^{m}\omega}(x)}\cdot\underset{x\in S(n,n)}{\inf}\,\frac{\opera_{\omega}^{n}f(x)}{\opera_{\omega}^{n}\ind_{\omega}(x)}.
        \end{aligned}
    \end{equation*}
    Take the infimum over $x\in S(n+m,n)$ and let $n\to\infty$, we have $$F_{\theta^{m}\omega}(\opera_{\omega}^{m}f)\geq F_{\theta^{m}\omega}(\opera_{\omega}^{m}\ind_\omega) F_{\omega}(f),$$ changing the index $m$ to $n$ completes the first inequality. By the first inequality, we have $$F_{\theta^{n}\omega}(\opera_{\theta^{n-1}\omega}f)\geq F_{\theta^{n}\omega}(\opera_{\theta^{n-1}\omega}\ind_{\theta^{n-1}\omega})F_{\theta^{n-1}\omega}(f)=\rho_{\theta^{n-1}\omega}F_{\theta^{n-1}\omega}(f)$$ for every $n\geq1$ and $f\in BV(X_{\theta^{n-1}\omega})$. Recall that $\rho_{\omega}^{n}:=\prod_{i=0}^{n-1}\rho_{\theta^{i}\omega}$, iterating such inequality gives us the second inequality. The third one is obtained by putting $f=\ind_\omega$ and $F_{\omega}(\ind_\omega)=1$.
\end{proof}

We define some partitions related to the random functional $F_\omega$. Recall the countable partition $\parti_{\omega}^{n}$ is a monotonicity partition of the map $T_{\omega}^{n}$, meaning $T_{\omega}^{n}$ is continuous and strictly monotone on each element. For each $\omega\in\Omega$ and $n\geq1$, given $a\geq0$ and $b\geq1$, let $P_{\omega,n}(a,b)$ be the collection of all finite partitions of $I$ such that for each finite partition $\mathcal{U}=\{U_i\}\in P_{\omega,n}(a,b)$, $$\vari_{U_i}(g_{\omega}^{n})\leq a\supnorm{g_{\omega}^{n}}$$ and $$\sum_{\substack{Z\in\parti_{\omega}^{n} \\ Z\cap U_{i}\neq\varnothing}}\underset{Z\cap U_i}{\sup}\,g_{\omega}^{n}\leq b\supnorm{g_{\omega}^{n}}.$$ hold for finite indexes. From \cite[Lemma 2.10, Remark 2.11]{atnip1}, such partition exists. Given such a finite partition $\mathcal{U}$, choose the coarsest finite partition $\tilde{\mathcal{U}}$ among those finer than $\parti_{\omega}^{n}$ and $\mathcal{U}$ such that elements are either disjoint from $K_{\omega,n-1}$ or contained in $K_{\omega,n-1}$. Let $$\parti_{\omega,g}^{n}:=\{U\in\tilde{\mathcal{U}}:U\subseteq K_{\omega,n-1},F_{\omega}(\ind_U)>0\},\quad\parti_{\omega,b}^{n}:=\{U\in\tilde{\mathcal{U}}:U\subseteq K_{\omega,n-1},F_{\omega}(\ind_U)=0\}$$ be the good and bad subcollections of partition $\tilde{\mathcal{U}}$ contained in $K_{\omega,n-1}$ respectively. For all $n\geq1$, we assume there exists $\delta_{\omega,n}>0$ with $\log\delta_{\omega,n}\in L^1(\prob)$ such that $$\underset{U\in\parti_{\omega,g}^{n}}{\inf}\,F_{\omega}(\ind_U)\geq 2\delta_{\omega,n}.$$ Finally, we say two elements in $\parti_{\omega,g}^{n}\cup\parti_{\omega,b}^{n}$ are \emph{contiguous} if either they share a boundary point or they are separated by a connected component of $\bigcup_{i=0}^{n-1}T_{\omega}^{-i}(H_{\theta^{i}\omega})$. Let $\eta_{\omega}^{n}$ denote the maximum number of contiguous elements in $\parti_{\omega,b}^{n}$ and assume $\log\eta_{\omega}^{n}\in L^1(\prob)$ for each $n\geq1$.

Compare with Definition \ref{def6.4}, we say the summable potential $\varphi_{c}:\Omega\times I\to\R\cup\{-\infty\}$ is \emph{contracting} for the open system if $$\lim_{n\to\infty}\frac{1}{n}\log\supnorm{g_{\omega}^{n}}+\limsup_{n\to\infty}\frac{1}{n}\log\eta_{\omega}^{n}<\lim_{n\to\infty}\frac{1}{n}\log\underset{S(n,\infty)}{\inf}\,\opera_{\omega}^{n}\ind_\omega.$$
\begin{remark}
    Since $\inf_{S(1,\infty)}\opera_{\omega}\ind_{\omega}\leq\rho_{\omega}$, we see the contracting condition also implies $$\lim_{n\to\infty}\frac{1}{n}\log\supnorm{g_{\omega}^{n}}+\limsup_{n\to\infty}\frac{1}{n}\log\eta_{\omega}^{n}<\lim_{n\to\infty}\frac{1}{n}\log\rho_{\omega}^{n}=\int_{\Omega}\log\rho_{\omega}\,\dd\prob(\omega).$$ It is clear that $$\lim_{n\to\infty}\frac{1}{n}\log\supnorm{g_{\omega}^{n}}<\lim_{n\to\infty}\frac{1}{n}\log\rho_{\omega}^{n}.$$
\end{remark}

\section{Lasota-Yorke type inequality}\label{sec3}
We start this section with a useful estimation of the variations of iterated functions.
\begin{proposition}\label{prop6.11}
    For all $\omega\in\Omega$, all $f\in BV(X_\omega)$ and all $n\geq1$, there exist two measurable positive constants $A_{\omega,n}$ and $B_{\omega,n}$ such that $$\vari(\opera_{\omega}^{n}f)\leq A_{\omega,n}\vari(f)+B_{\omega,n}F_{\omega}(|f|).$$
\end{proposition}
\begin{proof}
    The proof is based on \cite{liverani3,liverani4} in the deterministic case and \cite{atnip1} in random weighted covering case, on \cite[Lemma 1.5.1]{atnip3} as well.
    
    For intervals $U\in\tilde{\mathcal{U}}\setminus(\parti_{\omega,g}^{n}\cup\parti_{\omega,b}^{n})$, the value of $\opera_{\omega}^{n}\ind_{U}$ turns out to be zero since $U\cap K_{\omega,n-1}=\varnothing$. Therefore, we only need to consider estimation on intervals in $\parti_{\omega,g}^{n}\cup\parti_{\omega,b}^{n}$. Rewrite $$\opera_{\omega}^{n}f=\sum_{U\in\parti_{\omega,g}^{n}\cup\parti_{\omega,b}^{n}}(\ind_{U}fg_{\omega}^{n})\circ T_{\omega,U}^{-n},$$ where $T_{\omega,U}^{-n}:T_{\omega}^{n}(J_\omega)\to U$ is the inverse branch sending $T_{\omega}^{n}(x)$ to $x\in U$. Since $\ind_{U}\circ T_{\omega,U}^{-n}=\ind_{T_{\omega}^{n}(U)}$, the above expression is in fact $$\opera_{\omega}^{n}f=\sum_{U\in\parti_{\omega,g}^{n}\cup\parti_{\omega,b}^{n}}\ind_{T_{\omega}^{n}(U)}((fg_{\omega}^{n})\circ T_{\omega,U}^{-n})$$ and its variation is changed to $$\vari(\opera_{\omega}^{n}f)\leq\sum_{U\in\parti_{\omega,g}^{n}\cup\parti_{\omega,b}^{n}}\vari(\ind_{T_{\omega}^{n}(U)}((fg_{\omega}^{n})\circ T_{\omega,U}^{-n})).$$ Since for $U\in\parti_{\omega,g}^{n}\cup\parti_{\omega,b}^{n}$, by $$\vari(\ind_{T_{\omega}^{n}(U)}((fg_{\omega}^{n})\circ T_{\omega,U}^{-n}))\leq \vari_{U}(fg_{\omega}^{n})+2\,\underset{U}{\sup}\,|fg_{\omega}^{n}|,$$ we then have
    \begin{equation}\label{ls}
        \begin{aligned}
            \vari(\opera_{\omega}^{n}f)&\leq\sum_{U\in\parti_{\omega,g}^{n}\cup\parti_{\omega,b}^{n}}\vari_{U}(fg_{\omega}^{n})+2\,\underset{U}{\sup}\,|fg_{\omega}^{n}| \\
            &\leq \sum_{U\in\parti_{\omega,g}^{n}\cup\parti_{\omega,b}^{n}}\vari_{U}(fg_{\omega}^{n})+2\sum_{\substack{Z\in\parti_{\omega}^{n} \\ Z\cap U\neq\varnothing}}\underset{Z\cap U}{\sup}\,|fg_{\omega}^{n}| \\
            &\leq \sum_{U\in\parti_{\omega,g}^{n}\cup\parti_{\omega,b}^{n}}\vari_{U}(f)\supnorm{g_{\omega}^{n}}+\vari_{U}(g_{\omega}^{n})\supnorm{f\ind_{U}}+2\supnorm{f\ind_U}\sum_{\substack{Z\in\parti_{\omega}^{n} \\ Z\cap U\neq\varnothing}}\underset{Z\cap U}{\sup}\,g_{\omega}^{n} \\
            &\leq \sum_{U\in\parti_{\omega,g}^{n}\cup\parti_{\omega,b}^{n}}(a+2b+1)\supnorm{g_{\omega}^{n}}\vari_{U}(f)+(a+2b)\supnorm{g_{\omega}^{n}}\underset{U}{\inf}\,|f| \\
            &= (a+2b+1)\supnorm{g_{\omega}^{n}}\vari(f) \\
            &\qquad+(a+2b)\supnorm{g_{\omega}^{n}}\left(\sum_{U\in\parti_{\omega,g}^{n}}\underset{U}{\inf}\,|f|+\sum_{U\in\parti_{\omega,b}^{n}}\underset{U}{\inf}\,|f|\right)
        \end{aligned}
    \end{equation} We are interested in the last two sums. For the sum on the good part $\parti_{\omega,g}^{n}$, since we assumed there exists $\delta_{\omega,n}>0$ such that $$\underset{U\in\parti_{\omega,g}^{n}}{\inf}\,F_{\omega}(\ind_U)\geq 2\delta_{\omega,n}>0.$$ So there exists an integer $n_0(\omega,n)\geq1$ such that for all $x\in S(n_0(\omega,n),n_0(\omega,n))$, we have $$\underset{U\in\parti_{\omega,g}^{n}}{\inf}\,\frac{\opera_{\omega}^{n_0(\omega,n)}\ind_U(x)}{\opera_{\omega}^{n_0(\omega,n)}\ind_\omega(x)}\geq\delta_{\omega,n}.$$ Hence for such $x$ and $U$ in the good part, by $$\opera_{\omega}^{n_0(\omega,n)}(|f|\ind_U)(x)\geq\opera_{\omega}^{n_0(\omega,n)}\ind_U(x)\underset{U}{\inf}\,|f|\geq\delta_{\omega,n}\opera_{\omega}^{n_0(\omega,n)}\ind_\omega(x)\,\underset{U}{\inf}\,|f|$$ we obtain $$\sum_{U\in\parti_{\omega,g}^{n}}\underset{U}{\inf}\,|f|\leq\delta_{\omega,n}^{-1}\sum_{U\in\parti_{\omega,g}^{n}}\frac{\opera_{\omega}^{n_0(\omega,n)}(|f|\ind_U)(x)}{\opera_{\omega}^{n_0(\omega,n)}\ind_\omega(x)}\leq\delta_{\omega,n}^{-1}\frac{\opera_{\omega}^{n_0(\omega,n)}|f|(x)}{\opera_{\omega}^{n_0(\omega,n)}\ind_\omega(x)}.$$ Regarding the sum on the bad part $\parti_{\omega,b}^{n}$, we obtain that $$\sum_{U\in\parti_{\omega,b}^{n}}\underset{U}{\inf}\,|f|\leq 2\eta_{\omega}^{n}\left(\vari(f)+\sum_{U\in\parti_{\omega,g}^{n}}\underset{U}{\inf}\,|f|\right)$$ from \cite[Lemma 2.5]{liverani4} and \cite[Lemma 1.5.1]{atnip3}. Put the two sums into inequality (\ref{ls}), we now have
    \begin{equation*}
        \begin{aligned}
            \vari(\opera_{\omega}^{n}f)&\leq (a+2b+1)\supnorm{g_{\omega}^{n}}\vari(f) \\
            &\qquad+(a+2b)\supnorm{g_{\omega}^{n}}\left(2\eta_{\omega}^{n}\vari(f)+(1+2\xi_{\omega}^{n})\sum_{U\in\parti_{\omega,g}^{n}}\underset{U}{\inf}\,|f|\right) \\
            &\leq (a+2b+1+(2a+4b)\eta_{\omega}^{n})\supnorm{g_{\omega}^{n}}\vari(f) \\
            &\qquad+(a+2b)(1+2\eta_{\omega}^{n})\supnorm{g_{\omega}^{n}}\delta_{\omega,n}^{-1}\frac{\opera_{\omega}^{n_0(\omega,n)}|f|(x)}{\opera_{\omega}^{n_0(\omega,n)}\ind_\omega(x)} \\
            &\leq (a+2b+1+(2a+4b)\eta_{\omega}^{n})\supnorm{g_{\omega}^{n}}\vari(f) \\
            &\qquad+(a+2b)(1+2\eta_{\omega}^{n})\supnorm{g_{\omega}^{n}}\delta_{\omega,n}^{-1}F_{\omega}(|f|),
        \end{aligned}
    \end{equation*}
    where the last inequality is obtained by taking the infimum over $x\in S(n_0(\omega,n),n_0(\omega,n))$, such sequence is increasing, which allows us to replace it with $F_{\omega}(|f|)$ as $n\to\infty$. Therefore, the positive measurable constants are given by $$A_{\omega,n}:=(a+2b+1+(2a+4b)\eta_{\omega}^{n})\supnorm{g_{\omega}^{n}}$$ and $$B_{\omega,n}:=(a+2b)(1+2\eta_{\omega}^{n})\supnorm{g_{\omega}^{n}}\delta_{\omega,n}^{-1}.$$
\end{proof}

Since $\rho_{\omega}:=F_{\theta\omega}(\opera_{\omega}\ind_{\omega})$ is positive, we can set $\normopera_{\omega}:=\rho_{\omega}^{-1}\opera_{\omega}$, then $\normopera_{\omega}^{n}=(\rho_{\omega}^{n})^{-1}\opera_{\omega}^{n}$ for every $n\geq1$. Lemma \ref{lem6.6} immediately implies that $F_{\omega}(f)\leq F_{\theta^{n}\omega}(\normopera_{\omega}^{n}f)$ for all $\omega\in\Omega$ and $n\geq1$, which is useful in these estimations afterwards. As a consequence of Proposition \ref{prop6.11}, for all $f\in BV(X_\omega)$ and all $n\geq1$ we have $$\vari(\normopera_{\omega}^{n}f)\leq C_{\omega,n}\vari(f)+D_{\omega,n}F_{\omega}(|f|)$$ where $C_{\omega,n}=(\rho_{\omega}^{n})^{-1}A_{\omega,n}$ and $D_{\omega,n}=(\rho_{\omega}^{n})^{-1}B_{\omega,n}$. Since $$\log\eta_{\omega}^{n},\log\supnorm{g_{\omega}^{n}},\log\delta_{\omega,n},\log\rho_{\omega}^{n}\in L^1(\prob),$$ it follows that $\log C_{\omega,n},\log D_{\omega,n}\in L^1(\prob)$. We define $N_{c}\geq1$ to be the smallest integer $n\geq1$ such that $$-\infty<\int_{\Omega}\log C_{\omega,n}\,\dd\prob(\omega)<0.$$
\begin{lemma}\label{lem6.12}
    For all $f\in BV(X_\omega)$ and all $1\leq n\leq N_{c}$, there exists a measurable function $L_{\omega}\in(0,\infty)$ with $L_{\omega}^{n}:=\prod_{i=0}^{n-1}L_{\theta^i \omega}\geq 6^n$ and $\log L_{\omega}\in L^1(\prob)$ such that $$\bvnorm{\normopera_{\omega}^{n}f}\leq L_{\omega}^{n}(\vari(f)+F_{\theta^{n}\omega}(\normopera_{\omega}^{n}|f|)).$$
\end{lemma}
\begin{proof}
    We have
    \begin{equation*}
        \begin{aligned}
            \bvnorm{\normopera_{\omega}^{n}f}&=\vari(\normopera_{\omega}^{n}f)+\supnorm{\normopera_{\omega}^{n}f}\leq 2\vari(\normopera_{\omega}^{n}f)+F_{\theta^{n}\omega}(\normopera_{\omega}^{n}|f|) \\
            &\leq 2(C_{\omega,n}\vari(f)+D_{\omega,n}F_{\omega}(|f|))+F_{\theta^{n}\omega}(\normopera_{\omega}^{n}|f|) \\
            &\leq 2C_{\omega,n}\vari(f)+(2D_{\omega,n}+1)F_{\theta^{n}\omega}(\normopera_{\omega}^{n}|f|).
        \end{aligned}
    \end{equation*} Set $\tilde{L}_{\omega,n}:=\max\{6,2C_{\omega,n},2D_{\omega,n}+1\}$, then letting $L_{\omega}:=\max\{\tilde{L}_{\omega,i}:1\leq i\leq N_{c}\}$ and $L_{\omega}^{n}:=\prod_{i=0}^{n-1}L_{\theta^i \omega}\geq 6^n$ finishes the lemma.
\end{proof}
We establish another adequate inequality in the style of J. Buzzi \cite{buzzi2} and J. Atnip et al. \cite{atnip1} for the following discussions, rather than the Lasota-Yorke type inequality in Proposition \ref{prop6.11}. For such $N_{c}\geq1$ that $-\infty<\int_{\Omega}\log C_{\omega,N_{c}}\,\dd\prob(\omega)<0$, define the negative random Birkhoff average $$\xi:=-\frac{1}{N_{c}}\int_{\Omega}\log C_{\omega,N_{c}}\,\dd\prob(\omega)>0.$$

According to Lemma \ref{lem6.12}, since $\log L_{\omega}$ is integrable, we similarly define its random Birkhoff average of $L_\omega$ at iterate $N_{c}$ as $\zeta:=\frac{1}{N_{c}}\int_{\Omega}\log L_{\omega}^{N_{c}}\,\dd\prob(\omega)>0$. Birkhoff's ergodic theorem also implies that $$\zeta=\lim_{n\to\infty}\frac{1}{nN_{c}}\sum_{k=0}^{n-1}\log L_{\theta^{kN_{c}}\omega}^{N_{c}}.$$
\begin{proposition}\label{prop6.13}
    For every $\epsilon>0$, there exists a finite measurable function $C_{\epsilon}(\omega)>0$ such that for $\prob$-a.e. $\omega\in\Omega$,  all $n\geq1$ and all $f\in BV(X_{\theta^{-n}\omega})$ we have $$\vari(\normopera_{\theta^{-n}\omega}^{n}f)\leq C_{\epsilon}(\omega)e^{-(\xi-\epsilon)n}\vari(f)+C_{\epsilon}(\omega)F_{\omega}(\normopera_{\theta^{-n}\omega}^{n}|f|).$$
\end{proposition}
\begin{proof}
    The proof is referred to \cite[Proposition 4.9]{atnip1} but we need to be careful with the second term, so there are some modifications related to random functional $F_\omega$.

    For a function $g:\Omega\to\R$, we can write $$\sum_{i=0}^{nN_{c}-1}g\circ\theta^{i}=\sum_{j=0}^{N_{c}-1}\sum_{i=0}^{n-1}g\circ\theta^{iN_{c}+j}.$$ Hence, there exists an integer $0\leq r_\omega\leq N_{c}-1$ such that $$\lim_{n\to\infty}\frac{1}{n}\sum_{i=0}^{n-1}\log C_{\theta^{-r_{\omega}-iN_{c}}\omega,N_{c}}\leq\xi.$$ For each $n\geq1$, we can write it as $n=s_{\omega}N_{c}+d_{\omega}+r_{\omega}$ where $s_{\omega}\geq0$ and $0\leq d_{\omega}\leq N_{c}-1$. For simplicity, we drop the symbol $\omega$ in the following proof. For $f\in BV(X_{\theta^{-n}\omega})$, we write $\normopera_{\theta^{-n}\omega}^{n}f$ as $\normopera_{\theta^{-r}\omega}^{r}\circ\normopera_{\theta^{-n+d}\omega}^{sN_{c}}\circ\normopera_{\theta^{-n}\omega}^{d}f$. For $0\leq r\leq N_{c}-1$, recall that $F_{\theta^{-i}\omega}(|f|)\leq F_{\omega}(\normopera_{\theta^{-i}\omega}^{i}|f|)$ for $i\geq0$ and $f\in BV(X_{\theta^{-i}\omega})$, then
    \begin{equation*}
        \begin{aligned}
            \vari(\normopera_{\theta^{-r}\omega}^{r}f)&=\vari(\normopera_{\theta^{-1}\omega}(\normopera_{\theta^{-r}\omega}^{r-1}f)) \\
            &\leq L_{\theta^{-1}\omega}(\vari(\normopera_{\theta^{-r}\omega}^{r-1}f)+F_{\omega}(\normopera_{\theta^{-r}\omega}^{r}|f|)) \\
            &\leq L_{\theta^{-1}\omega}L_{\theta^{-2}\omega}(\vari(\normopera_{\theta^{-r}\omega}^{r-2}f)+F_{\theta^{-1}\omega}(\normopera_{\theta^{-r}\omega}^{r-1}|f|))+L_{\theta^{-1}\omega}F_{\omega}(\normopera_{\theta^{-r}\omega}^{r}|f|) \\
            &\leq L_{\theta^{-1}\omega}L_{\theta^{-2}\omega}\vari(\normopera_{\theta^{-r}\omega}^{r-2}f)+(L_{\theta^{-1}\omega}L_{\theta^{-2}\omega}+L_{\theta^{-1}\omega})F_{\omega}(\normopera_{\theta^{-r}\omega}^{r}|f|) \\
            &\leq \cdots \\
            &\leq \vari(f)\prod_{i=1}^{r}L_{\theta^{-i}\omega}+F_{\omega}(\normopera_{\theta^{-r}\omega}^{r}|f|)\sum_{i=1}^{r}\prod_{k=1}^{i}L_{\theta^{-k}\omega}.
        \end{aligned}
    \end{equation*} Therefore, we have
    $$\vari(\normopera_{\theta^{-r}\omega}^{r}f)\leq C_{1}(\vari(f)+F_{\omega}(\normopera_{\theta^{-r}\omega}^{r}|f|))$$ for $f\in BV(X_{\theta^{-r}\omega})$, where $C_1:=C_{1}(\omega)=N_{c}\prod_{i=1}^{N_{c}}L_{\theta^{-i}\omega}$.

    Since $\log L_{\omega}\in L^1(\prob)$, by Birkhoff's ergodic theorem, we have $$\lim_{|k|\to\infty}\frac{1}{|k|}\log L_{\theta^{k}\omega}=0$$ for all $k\in\Z$. So for each $\delta>0$ there exists a measurable positive constant $C_\delta:=C_{\delta}(\omega)>0$ such that $L_{\theta^{k}\omega}\leq C_\delta e^{\delta|k|}$. It allows us to make the estimation below, let $f\in BV(X_{\theta^{-n}\omega})$,
    \begin{equation*}
        \begin{aligned}
            \vari(\normopera_{\theta^{-n}\omega}^{d}f)&\leq \vari(f)\prod_{i=n-d+1}^{n}L_{\theta^{-i}\omega}+F_{\omega}(\normopera_{\theta^{-n}\omega}^{n}|f|)\sum_{i=n-d+1}^{n}\prod_{k=n-d+1}^{i}L_{\theta^{-k}\omega} \\
            &\leq C_{\delta}^{d}e^{\delta\left(\sum_{i=n-d+1}^{n}i\right)}\vari(f)+F_{\omega}(\normopera_{\theta^{-n}\omega}^{n}|f|)\sum_{i=n-d+1}^{n}C_{\delta}^{i-n+d}e^{\delta\left(\sum_{k=n-d+1}^{i}k\right)} \\
            &\leq dC_{\delta}^{N_{c}}e^{nd\delta}(\vari(f)+F_{\omega}(\normopera_{\theta^{-n}\omega}^{n}|f|)) \\
            &\leq C_2 e^{\frac{n\epsilon}{2}}(\vari(f)+F_{\omega}(\normopera_{\theta^{-n}\omega}^{n}|f|))
        \end{aligned}
    \end{equation*} for chosen $0<\delta<\frac{\epsilon}{2N_{c}}$ and $C_2:=C_{2}(\omega)=N_{c}C_{\delta}^{N_{c}}$. Let $\tau:=\theta^{-n+d}\omega$ and $f\in BV(X_\tau)$, then we have
    \begin{equation*}
        \begin{aligned}
            \vari(\normopera_{\tau}^{sN_{c}}f)&=\vari(\normopera_{\theta^{(s-1)N_{c}}\tau}^{N_{c}}(\normopera_{\tau}^{(s-1)N_{c}}f)) \\
            &\leq C_{\theta^{(s-1)N_{c}}\tau,N_{c}}\vari(\normopera_{\tau}^{(s-1)N_{c}}f)+D_{\theta^{(s-1)N_{c}}\tau,N_{c}}F_{\tau}(|f|) \\
            &\leq \prod_{i=1}^{s}C_{\theta^{(s-i)N_{c}}\omega,N_{c}}\vari(f)+\sum_{i=1}^{s}\prod_{k=1}^{i}C_{\theta^{(s-k)N_{c}}\tau,N_{c}}D_{\theta^{(s-k)N_{c}}\tau,N_{c}}F_{\tau}(|f|).
        \end{aligned}
    \end{equation*} There exists a measurable function $C_{3}:=C_{3}(\omega)>0$ such that for every $i\geq1$, we have $$\prod_{k=1}^{i}C_{\theta^{(s-k)N_{c}}\tau,N_{c}}\leq C_{3}e^{-\left(\xi-\frac{\epsilon}{2}\right)iN_{c}}.$$ Since $\log D_{\omega,N_{c}}\in L^1(\prob)$, we can also find another measurable function $C_4:=C_{4}(\omega)\geq C_3$ such that for every $i\geq1$ we have $D_{\theta^{(s-i)N_{c}}\tau,N_{c}}\leq C_{4}e^{\frac{\epsilon}{2}}iN_{c}$. Therefore, by putting these inequalities into the above, we have
    \begin{equation*}
        \begin{aligned}
            \vari(\normopera_{\tau}^{sN_{c}}f)&\leq \prod_{i=1}^{s}C_{\theta^{(s-i)N_{c}}\omega,N_{c}}\vari(f)+\sum_{i=1}^{s}\prod_{k=1}^{i}C_{\theta^{(s-k)N_{c}}\tau,N_{c}}D_{\theta^{(s-k)N_{c}}\tau,N_{c}}F_{\tau}(|f|) \\
            &\leq C_{4}e^{-\left(\xi-\frac{\epsilon}{2}\right)sN_{c}}\vari(f)+\sum_{i=1}^{s}C_{4}e^{-\left(\xi-\frac{\epsilon}{2}\right)(s-i)N_{c}}C_{4}e^{\frac{\epsilon}{2}(s-i)N_{c}}F_{\tau}(|f|) \\
            &\leq C_{4}e^{-\left(\xi-\frac{\epsilon}{2}\right)sN_{c}}\vari(f)+\frac{C_{4}^{2}}{1-e^{-(\xi-\epsilon)N_{c}}}F_{\tau}(|f|) \\
            &\leq C_{5}e^{-\left(\xi-\frac{\epsilon}{2}\right)sN_{c}}\vari(f)+C_{5}F_{\tau}(|f|)
        \end{aligned}
    \end{equation*} where $C_{5}:=C_{5}(\omega)=\max\{C_{4},\frac{C_{4}^{2}}{1-e^{-(\xi-\epsilon)N_{c}}}\}$ is also a positive measurable function. Notice that $sN_{c}=n-d-r$ and $0\leq r,d\leq N_{c}-1$, we rewrite the power as $$-\left(\xi-\frac{\epsilon}{2}\right)sN_{c}+\frac{n\epsilon}{2}=-(\xi-\epsilon)n+\left(\xi-\frac{\epsilon}{2}\right)(d+r)\leq-(\xi-\epsilon)n+(2\xi-\epsilon)N_{c}.$$ Summarizing all estimations above together, we obtain
    \begin{equation*}
        \begin{aligned}
            \vari(\normopera_{\theta^{-n}\omega}^{n}f)&=\vari(\normopera_{\theta^{-r}\omega}^{r}\circ\normopera_{\theta^{-n+d}\omega}^{sN_{c}}\circ\normopera_{\theta^{-n}\omega}^{d}f) \\
            &\leq C_{1}\left(\vari\left(\normopera_{\theta^{-n+d}\omega}^{sN_{c}}\circ\normopera_{\theta^{-n}\omega}^{d}f\right)+F_{\omega}(\normopera_{\theta^{-n}\omega}^{n}|f|)\right) \\
            &\leq C_{1}\left(C_{5}e^{-\left(\xi-\frac{\epsilon}{2}\right)sN_{c}}\vari(\normopera_{\theta^{-n}\omega}^{d}f)+C_{5}F_{\theta^{-n+d}\omega}(\normopera_{\theta^{-n}\omega}^{d}|f|)+F_{\omega}(\normopera_{\theta^{-n}\omega}^{n}|f|)\right) \\
            &\leq C_{1}\left(C_{5}e^{-\left(\xi-\frac{\epsilon}{2}\right)sN_{c}}C_{2}e^{\frac{n\epsilon}{2}}(\vari(f)+F_{\omega}(\normopera_{\theta^{-n}\omega}^{n}|f|))+(1+C_5)F_{\omega}(\normopera_{\theta^{-n}\omega}^{n}|f|)\right) \\
            &\leq C_{1}C_{2}C_{5}e^{(2\xi-\epsilon)N_{c}}e^{-(\xi-\epsilon)n}\vari(f) \\
            &\qquad +C_{1}(C_{2}C_{5}e^{(2\xi-\epsilon)N_{c}}e^{-(\xi-\epsilon)n}+(1+C_5))F_{\omega}(\normopera_{\theta^{-n}\omega}^{n}|f|).
        \end{aligned}
    \end{equation*} Taking the measurable finite function $C_{\epsilon}(\omega)\geq C_{1}C_{2}C_{5}(1+C_5)e^{(2\xi-\epsilon)N_{c}}$ large enough finishes the proof.
\end{proof}
\section{Birkhoff cones and cone invariance}\label{sec4}
\subsection{Birkhoff cones}
We introduce an alternative and powerful method of proving the Ruelle-Perron-Frobenius theorem, called the \emph{Birkhoff cone}. It is different from the usual procedure of establishing a weak$^*$ limit of invariant measures in the case of distance expanding map on compact space; for details, see \cite[Chapter 13]{UrbanskiRoyMunday2}. This technique originated from lattice theory by G. Birkhoff \cite{birkhoff1}, it allows us to obtain the invariant measure by the constructive fixed-point theorem and thereafter the exponential rate for decay of correlations effectively. C. Liverani gives a systematic introduction to it; see \cite[Section 1]{liverani1}, \cite[Section 1]{liverani2} and \cite{Naud} as well.

Given a topological vector space $V$, a subset $C\subseteq V$ is called a \emph{convex cone} if it satisfies the following: (1) $C\cap -C=\varnothing$; (2) $\alpha C=C$ for all $\alpha>0$; (3) $C$ is convex; (4) for all $f,g\in C$ and all $\alpha_n\in\R$ with $\lim_{n\to\infty}\alpha_n=\alpha$, if $g-\alpha_{n}f\in C$ for each $n\geq1$, then $g-\alpha f\in C\cup\{0\}$. The relation $\preceq$ is defined on $V$ by condition $$f\preceq g\Longleftrightarrow g-f\in C\cup\{0\}$$ and it is a partial order. Define a distance $\Theta_{C}(f,g)$ on $C$ by $$\Theta_{C}(f,g):=\log\frac{\beta(f,g)}{\alpha(f,g)}$$ where $\alpha(f,g):=\sup\{a>0:af\preceq g\}$ and $\beta(f,g):=\inf\{b>0:g\preceq bf\}$. Note that $\Theta_{C}$ is a \emph{pseudo-metric} as two elements in the cone may have a distance of infinity. Furthermore, it is a projective metric, since any two proportional elements have zero distance according to definition. We call such a metric the \emph{Hilbert metric} on the convex cone $C$ and denote it by $\Theta_{C}$. We copy two useful inequalities from \cite{liverani1,liverani3}.
\begin{theorem}[Birkhoff's inequality, {\cite[Theorem 1.1]{liverani1}}]
    Let $V_1,V_2$ be two vector spaces and $C_1\subseteq V_1,C_2\subseteq V_2$ be two convex cones. Given a positive linear operator $L:V_1\to V_2$ such that $L(C_1)\subseteq C_2$, let $\Theta_{i}$ denote the Hilbert metric on each cone $C_i$ and $\Delta:=\diam_{C_2}(L(C_1))=\sup_{f,g\in C_1}\Theta_{2}(Lf,Lg)$, then $$\Theta_2(Lf,Lg)\leq\tanh\left(\frac{\Delta}{4}\right)\Theta_1(f,g),\ \forall\ f,g\in C_1.$$
\end{theorem}
\begin{lemma}[{\cite[Lemma 2.2]{liverani3}}]\label{lem4.3}
    Let $\Vert\cdot\Vert$ be a norm on $V$ such that for all $f,g\in V$ if $-f\preceq g\preceq f$ then $\Vert g\Vert\leq\Vert f\Vert$. Let $\rho:C\to(0,\infty)$ be a homogeneous and order-preserving function; that is, for all $f\in C$ and all $c>0$, $\rho(cf)=c\rho(f)$ and for all $f,g\in C$, if $f\preceq g$ then $\rho(f)\leq\rho(g)$. Given $f,g\in C$ for which $\rho(f)=\rho(g)$, we have $$\Vert f-g\Vert\leq(e^{\Theta_{C}(f,g)}-1)\min\{\Vert f\Vert,\Vert g\Vert\}.$$
\end{lemma}
\begin{remark}
    Clearly in Lemma \ref{lem4.3}, choosing $\rho(\cdot)=\Vert\cdot\Vert$ also satisfies this inequality since the norm on $V$ is homogeneous and order-preserving, so if $\Vert f\Vert=\Vert g\Vert$ then the inequality is reduced to $$\Vert f-g\Vert\leq(e^{\Theta_{C}(f,g)}-1)\Vert f\Vert.$$
\end{remark}
Fix a fiber $\omega\in\Omega$, given a positive cone parameter $a>0$, we construct the cones $$\Lambda_{\omega,a}:=\{f\in BV(X_\omega):f\not\equiv0,f\geq0,\,\vari(f)\leq aF_{\omega}(f)\}$$ and $$\Lambda_{\omega,+}:=\{f\in BV(X_\omega):f\not\equiv0,\,f\geq0\}.$$ We immediately obtain the following lemma by Proposition \ref{prop6.11}.
\begin{lemma}
    For every $\omega\in\Omega$, the open transfer operator $\opera_{\omega}$ is a weak contraction on $\Lambda_{\omega,+}$, that is, $\opera_{\omega}\Lambda_{\omega,+}\subseteq\Lambda_{\theta\omega,+}$. Hence, for all $n\geq1$, we have $\opera_{\omega}^{n}\Lambda_{\omega,+}\subseteq\Lambda_{\theta^{n}\omega,+}$.
\end{lemma}
We have an upper bound of the Hilbert metric of two specific elements in the cone.
\begin{lemma}[{\cite[Lemma 5.8]{atnip1}}]\label{lem6.15}
    Let $c\in(0,1)$, for all $f\in\Lambda_{\omega,ca}$, we have
    $$\Theta_{\Lambda_{\omega,a}}(\ind,f)\leq\log\frac{\sup f+c F_{\omega}(f)}{\min\{\inf_{X_\omega}f,(1-c)F_{\omega}(f)\}}.$$
\end{lemma}
\begin{proof}
    We provide the detailed proof since that of the referred one is omitted. Let $f\in\Lambda_{\omega,ca}$, then $\vari(f)\leq caF_{\omega}(f)$. Suppose $A\ind$ is a constant function with positive constant $A$, the relation $A\ind\preceq f$ occurs in $\Lambda_{\omega,a}$ if and only if $A\leq\inf_{X_\omega}f$ and $\vari(f-A\ind)\leq aF_{\omega}(f-A\ind)$. But $F_{\omega}(f-A\ind)=F_{\omega}(f)-A$ and $\vari(f-A\ind)=\vari(f)$, therefore, $c F_{\omega}(f)\leq F_{\omega}(f)-A$ is required, hence $A\leq\min\{\inf_{X_\omega}f,(1-c)F_{\omega}(f)\}$.

    Suppose $B\ind$ is another constant function with positive constant $B$, $f\preceq B\ind$ if and only if $\sup f\leq B$ and $\vari(B\ind-f)\leq aF_{\omega}(B\ind-f)$. By $\vari(B\ind-f)=\vari(f)$ and $F_{\omega}(B\ind-f)=B-F_{\omega}(f)\geq B-\sup f$, we have that $B\geq\sup f+c F_{\omega}(f)$ is sufficient.

    Finally, the inequality of distance is obtained by \[\Theta_{\Lambda_{\omega,a}}(\ind,f)\leq\log\frac{\inf\{B>0:f\preceq B\ind\}}{\sup\{A>0:A\ind\preceq f\}}=\log\frac{\sup f+c F_{\omega}(f)}{\min\{\inf_{X_\omega}f,(1-c)F_{\omega}(f)\}}.\]
\end{proof}
\subsection{Good and bad fibers}
Dividing good and bad fibers in the base dynamical system is undertaken in J. Buzzi's work \cite{buzzi2,buzzi1}. As the name suggests, the probability space $\Omega$ is divided into two disjoint parts $\Omega_G$ and $\Omega_B$, referring to good and bad fibers, respectively. The good one guarantees uniform contraction of Birkhoff cones under iterates of transfer operator, while the bad one may not have the desired contraction. To deal with them, we need to group them together with enough good fibers and show they do not occur too frequently.

In Proposition \ref{prop6.13}, for $\epsilon>0$, define $B:=B(\epsilon)$ to be the smallest upper bound of $C_{\epsilon}(\omega)$ on a set of $\prob$-measure at least $1-\frac{\epsilon}{8}$, that is, $$\prob(\{\omega\in\Omega:C_{\epsilon}(\omega)\leq B\})\geq 1-\frac{\epsilon}{8}.$$
\begin{defn}\label{def6.16}
    We say $\omega\in\Omega$ is a \emph{good} fiber with respect to numbers $\epsilon>0,a>0,B>0,R_{a}=q_{a}N_{c}$ (a multiple of $N_{c}$ for some $q_{a}\geq1$) if

    (1) \[Bq_{a}e^{-\frac{\xi}{2}R_{a}}\leq u<1;\]

    (2) $$\bigg|\frac{1}{R_{a}}\sum_{k=0}^{q_a-1}\log L_{\theta^{kN_{c}}\omega}^{N_{c}}-\zeta\bigg|\leq\epsilon.$$
\end{defn}
\begin{lemma}
    Given $0<\epsilon<\min\{1,\frac{\xi}{2}\}$ and $a>0$, there exist two numbers $B$ and $R_{a}$ such that there is a set $\Omega_{G}\subseteq\Omega$ of good fibers $\omega\in\Omega$ with measure $\prob(\Omega_G)\geq 1-\frac{\epsilon}{4}$.
\end{lemma}
\begin{proof}
    We have the set $\Omega_{1}:=\{\omega\in\Omega:C_{\epsilon}(\omega)\leq B\}$ with measure at least $1-\frac{\epsilon}{8}$ according to definition. Since $\epsilon<\frac{\xi}{2}$, choose $R=qN_{c}$ as a sufficiently large multiple of $N_{c}$ such that $Bqe^{-(\xi-\epsilon)R}\leq Bqe^{-\frac{\xi}{2}R}\leq u$. Let $q'\geq q$ be an integer and define the set $$\Omega_2:=\left\{\omega\in\Omega:\bigg|\frac{1}{q'N_{c}}\sum_{k=0}^{q'-1}\log L_{\theta^{kN_{c}}\omega}^{N_{c}}-\zeta\bigg|\leq\epsilon\right\}.$$ Choose $q_a\geq q'$ such that $\prob(\Omega_2)\geq1-\frac{\epsilon}{8}$ and set $R_{a}:=q_{a}N_{c}$. Finally set $\Omega_{G}:=\theta^{-R_a}(\Omega_1)\cap\Omega_2$, then it becomes the set of all good fibers with respect to $B$ and $R_a$, and $\prob(\Omega_G)\geq1-\frac{\epsilon}{4}$.
\end{proof}
Given $B>0$, we restrict cone parameters $a\geq a_{0}=v^{-1}B$ where $v\in(0,1)$ satisfies $u+v<\frac{3}{4}$ and $(1-u)v\leq\frac{1}{2}$, and set the number $R:=q_{a_0}N_{c}$ depending on $B$ with integer $q_{a_0}\geq1$. The cones whose states are in good fibers enjoy a strong contraction for every $R$-step iterate of the open operator.
\begin{lemma}\label{lem6.18}
    For each $a\geq a_0$ and good fibers $\omega\in\Omega$ with respect to numbers $\epsilon,B,R$, we have $$\normopera_{\omega}^{R}\Lambda_{\omega,a}\subseteq\Lambda_{\theta^{R}\omega,(u+v)a}\subseteq\Lambda_{\theta^{R}\omega,a}.$$
\end{lemma}
\begin{proof}
    For $f\in\Lambda_{\omega,a}$ and good fiber $\omega\in\Omega_G$, apply Proposition \ref{prop6.13} and $\vari(f)\leq aF_{\omega}(f)$, we have
    \begin{equation*}
        \begin{aligned}
            \vari(\normopera_{\omega}^{R}f)&\leq Be^{-(\xi-\epsilon)R}\vari(f)+BF_{\theta^{R}\omega}(\normopera_{\omega}^{R}f) \\
            &\leq Bq_{a_0}e^{-\frac{\xi}{2}R}\vari(f)+BF_{\theta^{R}\omega}(\normopera_{\omega}^{R}f) \\
            &\leq u\vari(f)+BF_{\theta^{R}\omega}(\normopera_{\omega}^{R}f) \\
            &\leq auF_{\omega}(f)+avF_{\theta^{R}\omega}(\normopera_{\omega}^{R}f) \\
            &\leq a(u+v)F_{\theta^{R}\omega}(\normopera_{\omega}^{R}f)
        \end{aligned}
    \end{equation*} since we have used $F_{\omega}(f)\leq F_{\theta^{R}\omega}(\normopera_{\omega}^{R}f)$. Since $u+v<1$, we have shown the $R$-step iterate of normalized operator $\normopera_{\omega}$ is uniformly contracting on the cone $\Lambda_{\omega,a}$ given $a\geq a_0$.
\end{proof}
We now deal with the complement $\Omega_{B}:=\Omega\setminus\Omega_{G}$, the set of \emph{bad} fibers with measure smaller than $\frac{\epsilon}{4}$, we shall not identify the subscript letter with the upper bound $B$ before. Given a fixed $\epsilon>0$, for $\omega\in\Omega$, let $j(\omega)$ be the smallest integer $1\leq j\leq R-1$ such that the following conditions are satisfied:

(1) $$\lim_{n\to\infty}\frac{1}{n}\#\{0\leq k\leq n-1:\theta^{\pm kR+j}\omega\in\Omega_{G}\}>1-\epsilon;$$

(2) $$\lim_{n\to\infty}\frac{1}{n}\#\{0\leq k\leq n-1:C_{\epsilon}(\theta^{\pm kR+j}\omega)\leq B\}>1-\epsilon.$$ We see $j(\omega)$ is measurable and $$j(\theta^{j(\omega)}\omega)=0,\quad j(\theta^{R}\omega)=j(\omega).$$

\begin{defn}\label{def6.19}
    For $\omega\in\Omega$, define the \emph{coating length} $l(\omega)$ to be the smallest integer $n\geq1$ such that $$\frac{1}{n}\sum_{\substack{0\leq k\leq n-1 \\ \theta^{kR}\omega\in\Omega_B}}\log\left(\prod_{k=0}^{q_{a_0}-1}L_{\theta^{kN_{c}}(\theta^{kR}\omega)}^{N_{c}}\right)\leq\zeta R\sqrt{\epsilon}$$ if $\omega\in\Omega_B$, or infinity if such $n$ does not exist. If $\omega\in\Omega_G$, then define $l(\omega)=1$.
\end{defn}
\begin{defn}
    We say a finite sequence $\{\omega,\theta\omega,\ldots,\theta^{l(\omega)R-1}\omega\}$ of $l(\omega)R$ fibers is a \emph{good block} originating at $\omega$ if $\omega\in\Omega_G$. If such $\omega\in\Omega_B$, then the sequence is called a \emph{bad block}.
\end{defn}
\begin{remark}\label{remark6.21}
    According to \cite{buzzi2} and \cite[Proposition 7.1]{atnip1}, for $\prob$-a.e. $\omega\in\Omega$ with $j(\omega)=0$, we have $l(\omega)<\infty$. If $\omega\in\Omega_B$, then $l(\omega)\geq2$.
\end{remark}
We copy a definition from \cite[Definition 7.3]{atnip1} to specify coating intervals.
\begin{defn}
    For each $\omega\in\Omega$, the \emph{coating intervals} along the orbit starting at $\omega$ are defined to be the bad blocks of the form $$\{\theta^{a_i R}\omega,\theta^{a_i R+1}\omega,\ldots,\theta^{b_i R-1}\omega\}$$ where $a_i,b_i$ are given by $$a_i:=\min\{k\geq b_{i-1}:\theta^{kR}\omega\in\Omega_B\},\quad b_i:=a_i+l(\theta^{a_i R}\omega)$$ for $i\geq1$ and $b_0:=-1$.
\end{defn}
For convenience, we denote by $$\Gamma(\omega):=\prod_{k=0}^{q_{a_0}-1}L_{\theta^{kN_{c}}\omega}^{N_{c}}.$$ Then for $f\in BV(X_\omega)$, for $q_{a_0}$ blocks of length $N_{c}$ and all $\omega\in\Omega$, by Lemma \ref{lem6.12} we obtain
\begin{equation*}
    \begin{aligned}
        \vari(\normopera_{\omega}^{R}f)&\leq L_{\omega}^{R}(\vari(f)+F_{\theta^{R}\omega}(\normopera_{\omega}^{R}|f|))\\
        &=\prod_{k=0}^{q_{a_0}-1}L_{\theta^{kN_{c}}\omega}^{N_{c}}(\vari(f)+F_{\theta^{R}\omega}(\normopera_{\omega}^{R}|f|)) \\
        &= \Gamma(\omega)(\vari(f)+F_{\theta^{R}\omega}(\normopera_{\omega}^{R}|f|)).
    \end{aligned}
\end{equation*}
We show the bad fibers also have strong contraction if they are grouped in coating intervals.
\begin{proposition}\label{prop6.23}
    For sufficiently small $\epsilon>0$ and $\prob$-a.e. $\omega\in\Omega$ with $j(\omega)=0$, we have $$\normopera_{\omega}^{l(\omega)R}\Lambda_{\omega,\tilde{a}}\subseteq\Lambda_{\theta^{l(\omega)R}\omega,(u+v)\tilde{a}}\subseteq\Lambda_{\theta^{l(\omega)R}\omega,\tilde{a}},$$ where $\tilde{a}:=\frac{Be^{\zeta R\sqrt{\epsilon}}}{(1-u)v}>a_0$.
\end{proposition}
\begin{proof}
    Lemma \ref{lem6.18} has already shown the case of $\omega\in\Omega_G$ since $l(\omega)=1$ for good fibers $\omega$. We turn to deal with bad fibers, set $l:=l(\omega)\geq2$ by Remark \ref{remark6.21}, for $f\in\Lambda_{\omega,+}$, by proof of Proposition \ref{prop6.13} we have 
    \begin{equation}\label{ineq4.1}
        \vari(\normopera_{\omega}^{lR}f)\leq\left(\prod_{i=0}^{l-1}E_{\theta^{iR}\omega}\right)\vari(f)+\sum_{i=0}^{l-1}\left(G_{\theta^{iR}\omega}\prod_{k=i+1}^{l-1}E_{\theta^{kR}\omega}\right)F_{\theta^{lR}\omega}(\normopera_{\omega}^{lR}f)
    \end{equation} where $E_{\theta^{iR}\omega}:=Be^{-(\xi-\epsilon)R}$ for $\theta^{iR}\omega\in\Omega_G$ and $\Gamma(\theta^{iR}\omega)$ for $\theta^{iR}\omega\in\Omega_B$; $G_{\theta^{iR}\omega}:=B$ for $\theta^{iR}\omega\in\Omega_G$ and $\Gamma(\theta^{iR}\omega)$ for $\theta^{iR}\omega\in\Omega_B$; see \cite[Lemma 7.5]{atnip1} for more details.
    
    For $0\leq j\leq l-1$, we can divide $$\sum_{\substack{0\leq k\leq l-1 \\ \theta^{kR}\omega\in\Omega_B}}\log\Gamma(\theta^{kR}\omega)=\sum_{\substack{0\leq k\leq j-1 \\ \theta^{kR}\omega\in\Omega_B}}\log\Gamma(\theta^{kR}\omega)+\sum_{\substack{j\leq k\leq l-1 \\ \theta^{kR}\omega\in\Omega_B}}\log\Gamma(\theta^{kR}\omega).$$ From Definition \ref{def6.19} on $l$, it implies that $$\frac{1}{j}\sum_{\substack{0\leq k\leq j-1 \\ \theta^{kR}\omega\in\Omega_B}}\log\Gamma(\theta^{kR}\omega)>\zeta R\sqrt{\epsilon}$$ and $$\frac{1}{l-j}\sum_{\substack{j\leq k\leq l-1 \\ \theta^{kR}\omega\in\Omega_B}}\log\Gamma(\theta^{kR}\omega)\leq\zeta R\sqrt{\epsilon}.$$ For given sufficiently small $\epsilon>0$, since $L_{\omega}^{n}\geq 6^n$, we have $\log\Gamma(\omega)\geq R\log6$ and then
    \begin{equation*}
        \begin{aligned}
            \frac{1}{l-j}\#\{j\leq k\leq l-1:\theta^{kR}\omega\in\Omega_B\}&=\frac{1}{l-j}\sum_{\substack{j\leq k\leq l-1 \\ \theta^{kR}\omega\in\Omega_B}}\ind \\
            &\leq \frac{1}{(l-j)R\log 6}\sum_{\substack{j\leq k\leq l-1 \\ \theta^{kR}\omega\in\Omega_B}}\log\Gamma(\theta^{kR}\omega) \\
            &\leq \frac{(l-j)\zeta R\sqrt{\epsilon}}{(l-j)R\log6}\leq\gamma<1
        \end{aligned}
    \end{equation*} provided such $\epsilon$ small enough that $\zeta\sqrt{\epsilon}<\log6$. When $0\leq j\leq l-1$, for the first coefficient, from $\#\{j\leq k\leq l-1:\theta^{kR}\omega\in\Omega_B\}\leq\gamma(l-j)$ we deduce \[\#\{j\leq k\leq l-1:\theta^{kR}\omega\in\Omega_G\}\geq(1-\gamma)(l-j).\] For the good fibers we have already assumed $Be^{-(\xi-\epsilon)R}<1$, therefore,
    \begin{equation*}
        \begin{aligned}
            \prod_{k=j}^{l-1}E_{\theta^{kR}\omega}&=\prod_{\substack{j\leq k\leq l-1 \\ \theta^{kR}\omega\in\Omega_G}}Be^{-(\xi-\epsilon)R}\cdot\prod_{\substack{j\leq k\leq l-1 \\ \theta^{kR}\omega\in\Omega_B}}\Gamma(\theta^{kR}\omega) \\
            &\leq (Be^{-(\xi-\epsilon)R})^{(1-\gamma)(l-j)}e^{(l-j)\zeta R\sqrt{\epsilon}} \\
            &\leq B^{l-j}(e^{R(\zeta\sqrt{\epsilon}-(\xi-\epsilon)(1-\gamma))})^{l-j}.
        \end{aligned}
    \end{equation*}
    Since $B\geq1$, $\Gamma(\omega)\geq1$, we have that for $0\leq j\leq l-1$, $$E_{\theta^{jR}\omega}\leq B\Gamma(\theta^{jR}\omega)\leq B\prod_{\substack{j\leq k\leq l-1 \\ \theta^{kR}\omega\in\Omega_B}}\Gamma(\theta^{kR}\omega)\leq B(e^{\zeta R\sqrt{\epsilon}})^{l-j}.$$ By putting them into the initial inequality (\ref{ineq4.1}), we obtain
    \begin{equation*}
        \begin{aligned}
            \vari(\normopera_{\omega}^{lR}f)&\leq B^{l}e^{R(\zeta\sqrt{\epsilon}-(\xi-\epsilon)(1-\gamma))l}\vari(f)\\
            &\qquad+\sum_{j=0}^{l-1}B(e^{\zeta R\sqrt{\epsilon}})^{l-j}\left(Be^{R(\zeta\sqrt{\epsilon}-(\xi-\epsilon)(1-\gamma))}\right)^{l-j-1}F_{\theta^{lR}\omega}(\normopera_{\omega}^{lR}f) \\
            &\leq B^{l}e^{R(\zeta\sqrt{\epsilon}-(\xi-\epsilon)(1-\gamma))l}\vari(f)\\
            &\qquad+Be^{\zeta R\sqrt{\epsilon}}\sum_{j=0}^{l-1}\left(Be^{R(\zeta\sqrt{\epsilon}-(\xi-\epsilon)(1-\gamma))}\right)^{l-j-1}F_{\theta^{lR}\omega}(\normopera_{\omega}^{lR}f).
        \end{aligned}
    \end{equation*} Here we use an estimation of power in \cite[Observation 7.4]{atnip1} that for sufficiently small $\epsilon>0$, we must have \[\zeta\sqrt{\epsilon}-(\xi-\epsilon)(1-\gamma)<-\frac{\xi}{2},\] hence by $Be^{-\frac{\xi}{2}R}\leq u$, we have
    \begin{equation}\label{ineq4.2}
        \begin{aligned}
            \vari(\normopera_{\omega}^{lR}f)&\leq(Be^{-\frac{\xi}{2}R})^{l}\vari(f)+Be^{\zeta R\sqrt{\epsilon}}\sum_{j=0}^{l-1}(Be^{-\frac{\xi}{2}R})^{l-j-1}F_{\theta^{lR}\omega}(\normopera_{\omega}^{lR}f) \\
            &\leq u^{l}\vari(f)+Be^{\zeta R\sqrt{\epsilon}}F_{\theta^{lR}\omega}(\normopera_{\omega}^{lR}f)\sum_{j=0}^{l-1}u^{l-j-1} \\
            &\leq u^{l}\vari(f)+\frac{Be^{\zeta R\sqrt{\epsilon}}}{1-u}F_{\theta^{lR}\omega}(\normopera_{\omega}^{lR}f).
        \end{aligned}
    \end{equation} Set $\tilde{a}:=\frac{Be^{\zeta R\sqrt{\epsilon}}}{(1-u)v}$, for $f\in\Lambda_{\omega,\tilde{a}}$, inequality (\ref{ineq4.2}) is changed to
    \begin{equation*}
        \begin{aligned}
            \vari(\normopera_{\omega}^{lR}f)&\leq u^{l}\tilde{a}F_{\omega}(f)+v\tilde{a}F_{\theta^{lR}\omega}(\normopera_{\omega}^{lR}f) \\
            &\leq u\tilde{a}F_{\theta^{lR}\omega}(\normopera_{\omega}^{lR}f)+v\tilde{a}F_{\theta^{lR}\omega}(\normopera_{\omega}^{lR}f) \\
            &\leq (u+v)\tilde{a}F_{\theta^{lR}\omega}(\normopera_{\omega}^{lR}f).
        \end{aligned}
    \end{equation*} That is, $\normopera_{\omega}^{l(\omega)R}f\in\Lambda_{\theta^{l(\omega)R}\omega,(u+v)\tilde{a}}$ and hence we have the contraction $$\normopera_{\omega}^{l(\omega)R}\Lambda_{\omega,\tilde{a}}\subseteq\Lambda_{\theta^{l(\omega)R}\omega,(u+v)\tilde{a}}\subseteq\Lambda_{\theta^{l(\omega)R}\omega,\tilde{a}}$$ with cone parameters $\tilde{a}=\frac{Be^{\zeta R\sqrt{\epsilon}}}{(1-u)v}=\frac{e^{\zeta R\sqrt{\epsilon}}}{1-u}a_0>a_0$.
\end{proof}
According to \cite{atnip1}, for $n\geq1$, define $K_n\geq0$ and $0\leq k(n)\leq R-1$ such that $n=K_{n}R+k(n)$. Given an $\omega_0\in\Omega$, let $\omega_{i}:=\theta^{l(\omega_{i-1})R}(\omega_{i-1})$ for each $i\geq1$. Then it is possible to divide the $n$-length orbit of $\omega_0$ into $h_{\omega_0}(n)$ blocks of length $l(\omega_i)R$ for $0\leq i\leq h_{\omega_0}(n)$ and the remainder block of length $r_{\omega_0}(n)R$ for $1\leq r_{\omega_0}(n)\leq l(\omega_{h_{\omega_0}(n)+1})$ plus with the last segment of length $k(n)$, that is, $$n=\sum_{0\leq i\leq h_{\omega_0}(n)}l(\omega_i)R+r_{\omega_0}(n)R+k(n).$$ Hence $K_n=\sum_{0\leq i\leq h_{\omega_0}(n)}l(\omega_i)+r_{\omega_0}(n)$.
\begin{lemma}[{\cite[Lemma 7.6]{atnip1}}]\label{lem6.24}
    For $\prob$-a.e. $\omega_0\in\Omega$ with $j(\omega_0)=0$, there exists a measurable $N_0:\Omega\to\N$ such that for all $n\geq N_0(\omega_0)$ we have $$\sum_{\substack{0\leq i\leq h_{\omega_0}(n) \\ \omega_i\in\Omega_B}}l(\omega_i)+r_{\omega_0}(n)\leq O(\sqrt{\epsilon})n.$$
\end{lemma}
The proof is omitted here as it is the same in the case of countably many branches. It means the total length of bad blocks is small enough such that it is proportional to the $\theta$-orbit of an initial $\omega_0$ almost everywhere.
\subsection{Auxiliary lemmas}
We prepare some lemmas for the following propositions in Section \ref{sec5} and Section \ref{sec6}. We revisit the properties of $F_{\omega}$ now; the proof is referred to \cite[Lemma 3.8]{liverani4} and \cite[Lemma 1.8.1]{atnip3}.
\begin{lemma}\label{lem6.25}
    For every $f\in BV(X_\omega)$, there exists $C>1$ such that $$F_{\theta^{lR}\omega}(\normopera_{\omega}^{lR}\ind_\omega)F_{\omega}(f)\leq F_{\theta^{lR}\omega}(\normopera_{\omega}^{lR}f)\leq C F_{\theta^{lR}\omega}(\normopera_{\omega}^{lR}\ind_\omega)F_{\omega}(f)$$ for $\prob$-a.e. $\omega\in\Omega$ with $j(\omega)=0$, where $l=l(\omega)\geq1$ is the coating length with respect to $\omega\in\Omega$.
\end{lemma}
\begin{proof}
    The first inequality is already shown in Lemma \ref{lem6.6}. Let $n,m\geq1$ and $x\in S(n+m,m)$, then
    \begin{equation*}
        \begin{aligned}
            \frac{\normopera_{\omega}^{n+m}f(x)}{\normopera_{\theta^{n}\omega}^{m}\ind_{\theta^{n}\omega}(x)}&=\frac{\normopera_{\omega}^{n+m}f(x)}{\normopera_{\omega}^{n+m}\ind_\omega(x)}\cdot\frac{\normopera_{\omega}^{n+m}\ind_\omega(x)}{\normopera_{\theta^{n}\omega}^{m}\ind_{\theta^{n}\omega}(x)} \\
            &=\frac{\normopera_{\omega}^{n+m}f(x)}{\normopera_{\omega}^{n+m}\ind_\omega(x)}\cdot\frac{\normopera_{\theta^{n}\omega}^{m}(\normopera_{\omega}^{n}\ind_{\omega}\cdot\ind_{\theta^{n}\omega})(x)}{\normopera_{\theta^{n}\omega}^{m}\ind_{\theta^{n}\omega}(x)} \\
            &\leq\frac{\normopera_{\omega}^{n+m}f(x)}{\normopera_{\omega}^{n+m}\ind_\omega(x)}\supnorm{\normopera_{\omega}^{n}\ind_{\omega}}.
        \end{aligned}
    \end{equation*} Take the infimum over $x\in S(n+m,m)$ and note that \[\frac{\normopera_{\omega}^{n+m}f(x)}{\normopera_{\theta^{n}\omega}^{m}\ind_{\theta^{n}\omega}(x)}=\frac{\normopera_{\theta^{n}\omega}^{m}(\normopera_{\omega}^{n}f)(x)}{\normopera_{\theta^{n}\omega}^{m}\ind_{\theta^{n}\omega}(x)},\] letting $m\to\infty$ on both sides gives $$F_{\theta^{n}\omega}(\normopera_{\omega}^{n}f)\leq\supnorm{\normopera_{\omega}^{n}\ind_{\omega}}F_{\omega}(f).$$ By Proposition \ref{prop6.23}, let $n=lR$ here, since $\ind_\omega\in\Lambda_{\omega,\tilde{a}}$, then $\normopera_{\omega}^{lR}\ind_\omega\in\Lambda_{\theta^{lR}\omega,\tilde{a}}$ and we have
    \begin{equation*}
        \begin{aligned}
            \supnorm{\normopera_{\omega}^{lR}\ind_\omega}&\leq \vari(\normopera_{\omega}^{lR}\ind_\omega)+F_{\theta^{lR}\omega}(\normopera_{\omega}^{lR}\ind_\omega) \\
            &\leq (\tilde{a}+1)F_{\theta^{lR}\omega}(\normopera_{\omega}^{lR}\ind_\omega).
        \end{aligned}
    \end{equation*} Hence $$F_{\theta^{lR}\omega}(\normopera_{\omega}^{lR}f)\leq(\tilde{a}+1)F_{\theta^{lR}\omega}(\normopera_{\omega}^{lR}\ind_\omega)F_{\omega}(f).$$ Taking $C:=\tilde{a}+1>1$ finishes the proof.
\end{proof}
Following \cite[Lemma 3.10]{liverani4} and \cite[Lemma 1.8.2]{atnip3}, we obtain the supremum of $F_\omega(\ind_Z)$ for $Z\in\parti_{\omega}^{n}$.
\begin{lemma}\label{lem6.26}
    Given $\tilde{a}>0$, for each $\omega\in\Omega$, there exists an integer $N_{\omega,\tilde{a}}\geq1$ such that for all $n\geq N_{\omega,\tilde{a}}$, $$\underset{Z\in\parti_{\omega}^{n}}{\sup}\,F_{\omega}(\ind_Z)<\frac{1}{8\tilde{a}^3}.$$
\end{lemma}
\begin{proof}
    Choose such an integer $N_{\omega,\tilde{a}}\geq1$ such that for all $n\geq N_{\omega,\tilde{a}}$, we have $\supnorm{g_{\omega}^{n}}<\frac{1}{8\tilde{a}^3}\rho_{\omega}^{n}$. Fix such $n$, let $m\geq1$, for $Z\in\parti_{\omega}^{n}$, we have $$\opera_{\omega}^{n}\ind_{Z}\leq\supnorm{g_{\omega}^{n}}<\frac{1}{8\tilde{a}^3}\rho_{\omega}^{n}.$$ For $x\in S(n+m,n+m)$, we have
    \begin{equation*}
        \begin{aligned}
            \frac{\opera_{\omega}^{n+m}\ind_{Z}(x)}{\opera_{\omega}^{n+m}\ind_{\omega}(x)}&\leq\frac{\supnorm{\opera_{\omega}^{n}\ind_{Z}}\opera_{\theta^{n}\omega}^{m}\ind_{\theta^{n}\omega}(x)}{\opera_{\omega}^{n+m}\ind_{\omega}(x)} \\
            &< \frac{\rho_{\omega}^{n}\opera_{\theta^{n}\omega}^{m}\ind_{\theta^{n}\omega}(x)}{8\tilde{a}^{3}\opera_{\theta^{n}\omega}^{m}(\opera_{\omega}^{n}\ind_{\omega})(x)} \\
            &\leq \frac{\rho_{\omega}^{n}}{8\tilde{a}^{3}}\left(\underset{S(n+m,m)}{\inf}\,\frac{\opera_{\theta^{n}\omega}^{m}(\opera_{\omega}^{n}\ind_{\omega})}{\opera_{\theta^{n}\omega}^{m}\ind_{\theta^{n}\omega}}\right)^{-1}
        \end{aligned}
    \end{equation*} Take the infimum over $x\in S(n+m,n+m)$ and let $m\to\infty$; by Lemma \ref{lem6.6}, the last term becomes \[(F_{\theta^{n}\omega}(\opera_{\omega}^{n}\ind_{\omega}))^{-1}\leq(\rho_{\omega}^{n})^{-1}.\] We obtain $$F_{\omega}(\ind_{Z})<\frac{\rho_{\omega}^{n}}{8\tilde{a}^{3}}(\rho_{\omega}^{n})^{-1}=\frac{1}{8\tilde{a}^3}$$ for every $Z\in\parti_{\omega}^{n}$; then taking the supremum over all intervals $Z$ completes the proof.
\end{proof}
We derive a lemma from \cite{liverani3,liverani4} that the functions in the cone $\Lambda_{\omega,a}$ cannot be too small, the infimum is bounded below.
\begin{lemma}\label{lem6.27}
    For $\prob$-a.e. $\omega\in\Omega$ with $j(\omega)=0$, for all $f\in\Lambda_{\omega,\tilde{a}}$ and all $n\geq N_{\omega,\tilde{a}}$, there exists an interval $Z_{f}\in\parti_{\omega,g}^{n}$ such that $$\inf_{Z_f}f\geq\frac{1}{2}F_{\omega}(f).$$
\end{lemma}
\begin{proof}
    We prove this lemma by contradiction. Suppose for every element $Z\in\parti_{\omega,g}^{n}$ we have $\inf_{Z}f<\frac{1}{2}F_{\omega}(f)$. We let $lR=M$ for convenience, for $N_{\omega,\tilde{a}}\leq n\leq M-1$,
    \begin{equation}\label{ineq2}
        \begin{aligned}
            \normopera_{\omega}^{M}f(x)&=\sum_{Z\in\parti_{\omega,g}^{n}}\normopera_{\omega}^{M}(f\ind_Z)(x)+\sum_{Z\in\parti_{\omega,b}^{n}}\normopera_{\omega}^{M}(f\ind_Z)(x) \\
            &\leq \frac{1}{2}F_{\omega}(f)\sum_{Z\in\parti_{\omega,g}^{n}}\normopera_{\omega}^{M}\ind_{Z}(x)+\sum_{Z\in\parti_{\omega,g}^{n}}\vari_{Z}(f)\normopera_{\omega}^{M}\ind_{Z}(x)+\sum_{Z\in\parti_{\omega,b}^{n}}\supnorm{f}\normopera_{\omega}^{M}\ind_{Z}(x)
        \end{aligned}
    \end{equation} For $Z\in\parti_{\omega,b}^{n}$, we have $0\leq F_{\theta^{M}\omega}(\normopera_{\omega}^{M}\ind_{Z})\leq C F_{\theta^{M}\omega}(\opera_{\omega}^{M}\ind)F_{\omega}(\ind_{Z})=0$ by Lemma \ref{lem6.25} and $F_{\omega}(\ind_Z)=0$ from definition of bad subcollection, hence
    $$\normopera_{\omega}^{M}\ind_{Z}\leq F_{\theta^{M}\omega}(\normopera_{\omega}^{M}\ind_Z)+\vari(\normopera_{\omega}^{M}\ind_Z)\leq u^{l}\vari(\ind_Z)+v\tilde{a}F_{\theta^{M}\omega}(\normopera_{\omega}^{M}\ind_Z)\leq2u^l.$$ For $Z\in\parti_{\omega,g}^{n}$, we have
    \begin{equation*}
        \begin{aligned}
            \normopera_{\omega}^{M}\ind_Z&\leq\vari(\normopera_{\omega}^{M}\ind_Z)+F_{\theta^{M}\omega}(\normopera_{\omega}^{M}\ind_Z) \\
            &\leq (u^{l}\vari(\ind_{Z})+v\tilde{a}F_{\theta^{M}\omega}(\normopera_{\omega}^{M}\ind_{Z}))+F_{\theta^{M}\omega}(\normopera_{\omega}^{M}\ind_Z) \\
            &\leq (u^{l}\vari(\ind_{Z})+v\tilde{a}CF_{\theta^{M}\omega}(\normopera_{\omega}^{M}\ind_\omega)F_{\omega}(\ind_Z))+CF_{\theta^{M}\omega}(\normopera_{\omega}^{M}\ind_\omega)F_{\omega}(\ind_Z) \\
            &\leq 2u^{l}+(1+v\tilde{a})CF_{\theta^{M}\omega}(\normopera_{\omega}^{M}\ind_\omega)F_{\omega}(\ind_Z) \\
            &\leq \left(2u^{l}+(1+v\tilde{a})CF_{\omega}(\ind_Z)\right)F_{\theta^{M}\omega}(\normopera_{\omega}^{M}\ind_\omega)
        \end{aligned}
    \end{equation*} where the last inequality is obtained by $F_{\theta^{n}\omega}(\normopera_{\omega}^{n}\ind_\omega)\geq1$ from Lemma \ref{lem6.6}. Applying functional $F_{\theta^{M}\omega}$ to both sides of inequality (\ref{ineq2}), we have that the first term $$F_{\theta^{M}\omega}\left(\frac{1}{2}F_{\omega}(f)\sum_{Z\in\parti_{\omega,g}^{n}}\normopera_{\omega}^{M}\ind_{Z}\right)\leq\frac{1}{2}F_{\omega}(f)F_{\theta^{M}\omega}(\normopera_{\omega}^{M}\ind_\omega),$$ the second term 
    \begin{equation*}
        \begin{aligned}
            F_{\theta^{M}\omega}\left(\sum_{Z\in\parti_{\omega,g}^{n}}\vari_{Z}(f)\normopera_{\omega}^{M}\ind_{Z}\right)&\leq F_{\theta^{M}\omega}\left(\sum_{Z\in\parti_{\omega,g}^{n}}\vari_{Z}(f)\left(2u^{l}+(1+v\tilde{a})CF_{\omega}(\ind_Z)\right)F_{\theta^{M}\omega}(\normopera_{\omega}^{M}\ind_\omega)\right) \\
            &\leq \sum_{Z\in\parti_{\omega,g}^{n}}\vari_{Z}(f)\left(2u^{l}+(1+v\tilde{a})CF_{\omega}(\ind_Z)\right)F_{\theta^{M}\omega}(\normopera_{\omega}^{M}\ind_\omega)
        \end{aligned}
    \end{equation*}
    since the term inside the second big bracket is just number, and the third term $$F_{\theta^{M}\omega}\left(\sum_{Z\in\parti_{\omega,b}^{n}}\supnorm{f}\normopera_{\omega}^{M}\ind_{Z}\right)\leq\sum_{Z\in\parti_{\omega,b}^{n}}2u^{l}\supnorm{f}F_{\theta^{M}\omega}(\normopera_{\omega}^{M}\ind_{\omega}).$$ Sum these terms together, we have
    \begin{equation*}
        \begin{aligned}
            F_{\theta^{M}\omega}(\normopera_{\omega}^{M}f)&\leq \frac{1}{2}F_{\omega}(f)F_{\theta^{M}\omega}(\normopera_{\omega}^{M}\ind_\omega)\\
            &\qquad+\sum_{Z\in\parti_{\omega,g}^{n}}\vari_{Z}(f)\left(2u^{l}+(1+v\tilde{a})CF_{\omega}(\ind_Z)\right)F_{\theta^{M}\omega}(\normopera_{\omega}^{M}\ind_\omega) \\
            &\qquad+ \sum_{Z\in\parti_{\omega,b}^{n}}2u^{l}\supnorm{f}F_{\theta^{M}\omega}(\normopera_{\omega}^{M}\ind_{\omega}),
        \end{aligned}
    \end{equation*} divide on both sides by $F_{\theta^{M}\omega}(\normopera_{\omega}^{M}\ind_\omega)$ and let $l\to\infty$, hence $M\to\infty$ since $R$ is fixed, then $u^l\to0$ and the third term also tends to $0$, we obtain
    \begin{equation*}
        \begin{aligned}
            F_{\omega}(f)&\leq \frac{1}{2}F_{\omega}(f)+\sum_{Z\in\parti_{\omega,g}^{n}}\vari_{Z}(f)(1+v\tilde{a})CF_{\omega}(\ind_Z) \\
            &\leq  \frac{1}{2}F_{\omega}(f)+\vari(f)(1+v\tilde{a})C\underset{Z\in\parti_{\omega,g}^{n}}{\sup}\,F_{\omega}(\ind_Z) \\
            &\leq \left(\frac{1}{2}+\tilde{a}(1+v\tilde{a})C\underset{Z\in\parti_{\omega,g}^{n}}{\sup}\,F_{\omega}(\ind_Z)\right)F_{\omega}(f).
        \end{aligned}
    \end{equation*}
    Since for $n\geq N_{\omega,\tilde{a}}$, we have $\sup_{Z\in\parti_{\omega,g}^{n}}F_{\omega}(\ind_Z)<\frac{1}{8\tilde{a}^3}$ by Lemma \ref{lem6.26}. By $B\geq1$ and $a_0:=v^{-1}B>1$, we have $\tilde{a}>a_{0}>1$, then $C=1+\tilde{a}<2\tilde{a}$, we see $\tilde{a}(1+v\tilde{a})C<4\tilde{a}^3$, hence $$\tilde{a}(1+v\tilde{a})C\underset{Z\in\parti_{\omega,g}^{n}}{\sup}\,F_{\omega}(\ind_Z)<\frac{1}{8\tilde{a}^3}4\tilde{a}^3=\frac{1}{2},$$ leading to $F_{\omega}(f)<F_{\omega}(f)$ for all $f\in\Lambda_{\omega,\tilde{a}}$, which is a contradiction. Therefore, there must exist an interval $Z_f$ such that $\inf_{Z_f}f\geq\frac{1}{2}F_{\omega}(f)$.
\end{proof}
\section{Finite diameter of cones}\label{sec5}
Recall that in the setting of Lemma \ref{lem6.24}, given an initial $\omega_0\in\Omega$ with $j(\omega_0)=0$, let $\omega_i:=\theta^{l(\omega_{i-1})R}\omega_{i-1}$ be defined for each $i\geq1$, for simplicity, let $l_i:=l(\omega_i)$, then the total length is given by $$\tilde{M}:=\sum_{i=0}^{k-1}l_{i}R$$ for $k\geq1$. We choose such $\tilde{M}\geq M=l(\omega_0)R$ large enough such that $$\underset{S(\tilde{M},\tilde{M})}{\inf}\,\frac{\opera_{\omega}^{\tilde{M}}\ind_{Z}}{\opera_{\omega}^{\tilde{M}}\ind_{\omega}}\geq\frac{1}{2}F_{\omega}(\ind_Z)$$ for every $Z\in\parti_{\omega,g}^{N_{\omega,\tilde{a}}}$. It is possible since we always have $F_{\omega}(\ind_Z)>0$. Obviously such an integer $\tilde{M}$ must be larger than $N_{\omega,\tilde{a}}$. We consider the following conditions, partially inspired by \cite{liverani4} and \cite{atnip1}.

(1) There exists $\tilde{C}\geq1$ such that $$\tilde{C}^{-1}\leq\underset{S(\tilde{M},\tilde{M})}{\inf}\,\normopera_{\omega}^{\tilde{M}}\ind_{\omega}\leq\supnorm{\normopera_{\omega}^{\tilde{M}}\ind_{\omega}}\leq\tilde{C}.$$

(2) For all $Z\in\parti_{\omega,g}^{N_{\omega,\tilde{a}}}$, there exists $0<\tilde{c}\leq\tilde{C}$ such that $$F_{\omega}(\ind_{Z})\geq\tilde{c}.$$
Hence, we can choose a set $\Omega_F$ of fibers with measure at least $1-\epsilon$ such that the above two conditions hold, $\theta^{-R}(\Omega_F)\subseteq\Omega_G$, and the cone $\Lambda_{\omega,\tilde{a}}$ has finite diameter after iterates of $\normopera_\omega$.
\begin{lemma}
    For all $\omega\in\Omega_{F}$ with $j(\omega)=0$, we have $$\normopera_{\omega}^{\tilde{M}}\Lambda_{\omega,\tilde{a}}\subseteq\Lambda_{\theta^{\tilde{M}}\omega,(u+v)\tilde{a}}\subseteq\Lambda_{\theta^{\tilde{M}}\omega,\tilde{a}}$$ and $$\diam_{\Lambda_{\theta^{\tilde{M}}\omega,\tilde{a}}}\left(\normopera_{\omega}^{\tilde{M}}\Lambda_{\omega,\tilde{a}}\right)\leq \Delta:=2\log\frac{4\tilde{C}^{2}C((u+v)C+1)}{\tilde{c}}<\infty.$$
\end{lemma}
\begin{proof}
    By Lemma \ref{lem6.18} and Proposition \ref{prop6.23}, for $f\in\Lambda_{\omega,\tilde{a}}$ we have $\normopera_{\omega}^{\tilde{M}}f\in\Lambda_{\theta^{\tilde{M}}\omega,(u+v)\tilde{a}}\subseteq\Lambda_{\theta^{\tilde{M}}\omega,\tilde{a}}$ and $$\Theta_{\Lambda_{\theta^{\tilde{M}}\omega,\tilde{a}}}\left(\ind_{\theta^{\tilde{M}}\omega},\normopera_{\omega}^{\tilde{M}}f\right)\leq\log\frac{\sup\normopera_{\omega}^{\tilde{M}}f+(u+v)F_{\theta^{\tilde{M}}\omega}(\normopera_{\omega}^{\tilde{M}}f)}{\min\{\inf_{S(\tilde{M},\tilde{M})}\normopera_{\omega}^{\tilde{M}}f,(1-(u+v))F_{\theta^{\tilde{M}}\omega}(\normopera_{\omega}^{\tilde{M}}f)\}}.$$ By Lemma \ref{lem6.25}, the numerator is
    \begin{equation*}
        \begin{aligned}
            \sup\,\normopera_{\omega}^{\tilde{M}}f+(u+v)F_{\theta^{\tilde{M}}\omega}(\normopera_{\omega}^{\tilde{M}}f)&\leq \vari(\normopera_{\omega}^{\tilde{M}}f)+(u+v+1)F_{\theta^{\tilde{M}}\omega}(\normopera_{\omega}^{\tilde{M}}f) \\
            &\leq ((u+v)\tilde{a}+(u+v+1))F_{\theta^{\tilde{M}}\omega}(\normopera_{\omega}^{\tilde{M}}f) \\
            &\leq C((u+v)C+1)F_{\theta^{\tilde{M}}\omega}(\normopera_{\omega}^{\tilde{M}}\ind_{\omega})F_{\omega}(f) \\
            &\leq \tilde{C}C((u+v)C+1)F_{\omega}(f).
        \end{aligned}
    \end{equation*} where we have applied $F_{\theta^{\tilde{M}}\omega}(\normopera_{\omega}^{\tilde{M}}\ind_\omega)\leq\supnorm{\normopera_{\omega}^{\tilde{M}}\ind_\omega}\leq\tilde{C}$. Since for each $f\in\Lambda_{\omega,\tilde{a}}$ there exists an interval $Z_f\in\parti_{\omega,g}^{N_{\omega,\tilde{a}}}$ such that $\inf_{Z_f}f\geq\frac{1}{2}F_{\omega}(f)$ according to Lemma \ref{lem6.27}, by $$\inf_{S(\tilde{M},\tilde{M})}\frac{\normopera_{\omega}^{\tilde{M}}\ind_{Z_f}}{\normopera_{\omega}^{\tilde{M}}\ind_\omega}\geq\frac{1}{2}F_{\omega}(\ind_{Z_f})\geq\frac{\tilde{c}}{2},$$ we have that
    \begin{equation*}
        \begin{aligned}
            \underset{S(\tilde{M},\tilde{M})}{\inf}\,\normopera_{\omega}^{\tilde{M}}f &\geq \underset{S(\tilde{M},\tilde{M})}{\inf}\,\normopera_{\omega}^{\tilde{M}}(f\ind_{Z_f}) \\
            &\geq \underset{Z_f}{\inf}\,f\cdot\underset{S(\tilde{M},\tilde{M})}{\inf}\,\normopera_{\omega}^{\tilde{M}}\ind_{Z_f} \\
            &\geq \frac{1}{2}F_{\omega}(f)\cdot\underset{S(\tilde{M},\tilde{M})}{\inf}\,\normopera_{\omega}^{\tilde{M}}\ind_{Z_f} \\
            &\geq \frac{1}{2}F_{\omega}(f)\cdot\underset{S(\tilde{M},\tilde{M})}{\inf}\,\frac{\normopera_{\omega}^{\tilde{M}}\ind_{Z_f}}{\normopera_{\omega}^{\tilde{M}}\ind_\omega}\cdot\underset{S(\tilde{M},\tilde{M})}{\inf}\,\normopera_{\omega}^{\tilde{M}}\ind_{\omega} \\
            &\geq \frac{\tilde{c}}{4\tilde{C}}F_{\omega}(f).
        \end{aligned}
    \end{equation*} Since $\tilde{c}\leq\tilde{C}$, we have $\frac{\tilde{c}}{4\tilde{C}}\leq\frac{1}{4}$. Also recall that $F_{\theta^{\tilde{M}}\omega}(\normopera_{\omega}^{\tilde{M}}f)\geq F_{\omega}(f)$ by Lemma \ref{lem6.6}, by the positivity of $F_{\omega}(f)$ and $1-(u+v)>\frac{1}{4}$, we deduce \[(1-(u+v))F_{\theta^{\tilde{M}}\omega}(\normopera_{\omega}^{\tilde{M}}f)>\frac{1}{4}F_{\omega}(f)\geq\frac{\tilde{c}}{4\tilde{C}}F_{\omega}(f),\] hence the denominator attains the minimum $\frac{\tilde{c}}{4\tilde{C}}F_{\omega}(f)$. Combine the two estimations together, we have $$\Theta_{\Lambda_{\theta^{\tilde{M}}\omega,\tilde{a}}}\left(\ind_{\theta^{\tilde{M}}\omega},\normopera_{\omega}^{\tilde{M}}f\right)\leq\log\frac{\tilde{C}C((u+v)C+1)F_{\omega}(f)}{\frac{\tilde{c}}{4\tilde{C}}F_{\omega}(f)}=\log\frac{4\tilde{C}^{2}C((u+v)C+1)}{\tilde{c}}.$$ The right side is finite and is denoted by $\frac{\Delta}{2}$, by triangle inequality the diameter of the cone $\Lambda_{\theta^{\tilde{M}}\omega,\tilde{a}}$ is at most $\Delta<\infty$.
\end{proof}
Due to the existence of set $\Omega_F$, we shall consider a new random variable $0\leq i:=i(\omega)\leq R-1$ to be the smallest integer such that the following conditions hold:

(1) $$\lim_{n\to\infty}\frac{1}{n}\#\{0\leq k\leq n-1:\theta^{\pm kR+i}\omega\in\Omega_{G}\}>1-\epsilon;$$

(2) $$\lim_{n\to\infty}\frac{1}{n}\#\{0\leq k\leq n-1:\theta^{\pm kR+i}\omega\in\Omega_{F}\}>1-\epsilon.$$ It immediately implies $i(\theta^{i(\omega)}\omega)=0$. When $i(\omega)=0$, we also have $j(\omega)=0$ either. We have a result of exponential decay of Hilbert metric below.
\begin{proposition}\label{prop6.29}
    Let $\epsilon>0$ be sufficiently small and let $V:\Omega\to(0,\infty)$ be a measurable function, there exists a constant $D\in(0,1)$ and a measurable function $n:\Omega\to\N$ such that for $\prob$-a.e. $\omega\in\Omega$, for all $n\geq n(\omega)$, all $m\geq0$ and all $|p|\leq n$, $$\Theta_{\Lambda_{\theta^{n+p}\omega,+}}\left(\normopera_{\theta^{p}\omega}^{n}f_{\theta^{p}\omega},\normopera_{\theta^{p-m}\omega}^{n+m}g_{\theta^{p-m}\omega}\right)\leq \Delta D^{n}$$ holds for all nonnegative $f,g\in BV_{\Omega}(I)$ with $\vari(f_{\theta^{p}\omega})\leq V(\omega)e^{n\epsilon}$ and $\vari(g_{\theta^{p-m}\omega})\leq V(\omega)e^{(n+m)\epsilon}$.
\end{proposition}
\begin{proof}
    We follow the scheme of J. Buzzi's proof in the case of random Lasota-Yorke maps \cite[Proposition 4.1]{buzzi2}. J. Atnip et al. provided a more detailed version \cite[Lemma 8.1]{atnip1}.
    
    Let $i:=i(\theta^{p}\omega)$ and set $d:=d(\theta^{p}\omega)\geq0$ to be the smallest integer satisfying

    (1) $i+dR\geq\frac{\epsilon n+\log V(\omega)}{\xi-\epsilon}$;

    (2) $C_{\epsilon}(\theta^{p+i+dR}\omega)\leq B$.

    The two conditions fail for at most $\frac{|p|+|p+n|}{R}O(\epsilon)=\frac{n}{R}O(\epsilon)$ integers $d$ in the interval $[-\frac{|p|}{R},\frac{|p+n|}{R}]$ by definition of $i$, thus $dR\leq O(\epsilon)n$ when $n$ is large enough. Now let $\hat{i}:=i+dR$ and $p':=p+\hat{i}$, then condition (2) turns to $C_{\epsilon}(\theta^{p'}\omega)\leq B$. For $\theta^{p'}\omega$, we observe $i(\theta^{p'}\omega)=0$ and hence $j(\theta^{p'}\omega)=0$. By Proposition \ref{prop6.13}, for $f_{\theta^{p}\omega}\in\Lambda_{\theta^{p}\omega,+}$,
    \begin{equation*}
        \begin{aligned}
            \vari(\normopera_{\theta^{p}\omega}^{\hat{i}}f_{\theta^{p}\omega})&\leq C_{\epsilon}(\theta^{\hat{i}}\theta^{p}\omega)e^{-(\xi-\epsilon)\hat{i}}\vari(f_{\theta^{p}\omega})+C_{\epsilon}(\theta^{\hat{i}}\theta^{p}\omega)F_{\theta^{\hat{i}}\theta^{p}\omega}(\normopera_{\theta^{p}\omega}^{\hat{i}}f_{\theta^{p}\omega}) \\
            &= C_{\epsilon}(\theta^{p'}\omega)e^{-(\xi-\epsilon)\hat{i}}\vari(f_{\theta^{p}\omega})+C_{\epsilon}(\theta^{p'}\omega)F_{\theta^{p'}\omega}(\normopera_{\theta^{p}\omega}^{\hat{i}}f_{\theta^{p}\omega}) \\
            &\leq Be^{-(\xi-\epsilon)\hat{i}}\vari(f_{\theta^{p}\omega})+BF_{\theta^{p'}\omega}(\normopera_{\theta^{p}\omega}^{\hat{i}}f_{\theta^{p}\omega}) \\
            &\leq B\frac{\vari(f_{\theta^{p}\omega})}{V(\omega)e^{n\epsilon}}+BF_{\theta^{p'}\omega}(\normopera_{\theta^{p}\omega}^{\hat{i}}f_{\theta^{p}\omega}) \\
            &\leq 2BF_{\theta^{p'}\omega}(\normopera_{\theta^{p}\omega}^{\hat{i}}f_{\theta^{p}\omega}) \\
            &\leq \tilde{a}F_{\theta^{p'}\omega}(\normopera_{\theta^{p}\omega}^{\hat{i}}f_{\theta^{p}\omega}).
        \end{aligned}
    \end{equation*} Therefore, $\normopera_{\theta^{p}\omega}^{\hat{i}_{*}}f_{\theta^{p}\omega}\in\Lambda_{\theta^{p'}\omega,\tilde{a}}$. Doing a similar calculation for $g_{\theta^{p-m}\omega}\in\Lambda_{\theta^{p-m}\omega,+}$ with condition $\vari(g_{\theta^{p-m}\omega})\leq V(\omega)e^{(n+m)\epsilon}$, we have $\normopera_{\theta^{p-m}\omega}^{m+\hat{i}_*}g_{\theta^{p-m}\omega}\in\Lambda_{\theta^{p'}\omega,\tilde{a}}$.

    The next step is to overcome the difficulty of iterate $n$ and to show the existence of such a measurable function $n(\omega)$. We observe that 
    $$\Theta_{\Lambda_{\theta^{n+p}\omega,+}}\left(\normopera_{\theta^{p}\omega}^{n}f_{\theta^{p}\omega},\normopera_{\theta^{p-m}\omega}^{n+m}g_{\theta^{p-m}\omega}\right)=\Theta_{\Lambda_{\theta^{n+p}\omega,+}}\left(\normopera_{\theta^{p'}\omega}^{n'}\circ\normopera_{\theta^{p}\omega}^{\hat{i}}f_{\theta^{p}\omega},\normopera_{\theta^{p'}\omega}^{n'}\circ\normopera_{\theta^{p-m}\omega}^{m+\hat{i}}g_{\theta^{p-m}\omega}\right)$$ where $n':=n-\hat{i}=n-(p'-p)\geq(1-O(\epsilon))n$. For convenience, set $\tilde{f}:=\normopera_{\theta^{p}\omega}^{\hat{i}}f_{\theta^{p}\omega}$ and $\tilde{g}:=\normopera_{\theta^{p-m}\omega}^{m+\hat{i}}g_{\theta^{p-m}\omega}$. For such $n'$, write $n'=tR+r$ with $0\leq r\leq R-1$. Define $\tau:=\theta^{p'}\omega$ and let $L$ be the total length of the intersections of coating intervals of the orbit starting at $\tau$ with $[0,t+1)$. By Lemma \ref{lem6.24}, for sufficiently large $t\geq t_{0}(\omega)$ for measurable $t_0:\Omega\to\N$, $L\leq O(\sqrt{\epsilon})t$. From measurability of $t$, the existence of measurable $n(\omega)$ with $n\geq n(\omega)$ is guaranteed as $n$ depending on $t$. If the first block $\tau$ is good, set $k=1$, otherwise set $k=l(\tau)+1$, then $k\leq L+1 \leq O(\sqrt{\epsilon})t$ for all large $t$. Since $\normopera_{\tau}^{n'}=\normopera_{\theta^{p+n'-r}\tau}^{r}\circ\normopera_{\tau}^{tR}$ and $\normopera_{\theta^{p+n'-r}\tau}^{r}$ preserves the cone $\Lambda_{\theta^{tR}\tau,+}$,
    \begin{equation*}
        \begin{aligned}
            \Theta_{\Lambda_{\theta^{n+p}\omega,+}}\left(\normopera_{\theta^{p'}\omega}^{n'}\tilde{f},\normopera_{\theta^{p'}\omega}^{n'}\tilde{g}\right)&=\Theta_{\Lambda_{\theta^{n+p}\omega,+}}\left(\normopera_{\tau}^{n'}\tilde{f},\normopera_{\tau}^{n'}\tilde{g}\right) \\
            &\leq \Theta_{\Lambda_{\theta^{tR}\tau,+}}\left(\normopera_{\tau}^{tR}\tilde{f},\normopera_{\tau}^{tR}\tilde{g}\right) \\
            &\leq \Theta_{\Lambda_{\theta^{tR}\tau,\tilde{a}}}\left(\normopera_{\tau}^{tR}\tilde{f},\normopera_{\tau}^{tR}\tilde{g}\right).
        \end{aligned}
    \end{equation*} Since $\normopera_{\theta^{jR}\tau}^{l(\theta^{jR}\tau)}$ is a weak contraction on $\Lambda_{\theta^{jR}\tau,\tilde{a}}$ if $\theta^{jR}\tau$ is bad, is a strong contraction by a factor $\tilde{D}\in(0,1)$ if $\theta^{jR}\tau$ is good, we have $$\Theta_{\Lambda_{\theta^{tR}\tau,\tilde{a}}}\left(\normopera_{\tau}^{tR}\tilde{f},\normopera_{\tau}^{tR}\tilde{g}\right)\leq\tilde{D}^{t-L-1}\Theta_{\Lambda_{\theta^{kR}\tau,\tilde{a}}}\left(\normopera_{\tau}^{kR}\tilde{f},\normopera_{\tau}^{kR}\tilde{g}\right),$$ here $LR$ is the total length of coating intervals and $L\leq O(\sqrt{\epsilon})t$, so $\tilde{D}^{t-L-1}\leq\tilde{D}^{(1-O(\sqrt{\epsilon}))t}$. By Proposition \ref{prop6.23}, we have $$\normopera_{\tau}^{kR}\Lambda_{\tau,\tilde{a}}\subseteq\normopera_{\theta^{(k-1)R}\tau}^{R}\Lambda_{\theta^{(k-1)R}\tau,\tilde{a}}$$ for all $1\leq k\leq L+1$. Note that the diameter is at most $\Delta<\infty$ for good fiber $\theta^{(k-1)R}\tau$, hence $$\Theta_{\Lambda_{\theta^{n+p}\omega,+}}\left(\normopera_{\theta^{p'}\omega}^{n'}\tilde{f},\normopera_{\theta^{p'}\omega}^{n'}\tilde{g}\right)\leq \Delta D^{n}$$ where $D:=\tilde{D}^{\frac{1-O(\sqrt{\epsilon})}{R}}\in(0,1)$. Moreover, $D$ does not depend on function $V(\omega)$.
\end{proof}
In conjunction with Lemma \ref{lem4.3}, we obtain the second inequality.
\begin{corollary}\label{coro6.30}
    Suppose the hypotheses hold in Proposition \ref{prop6.29}, there exists $\bar{D}\in(0,1)$ and a measurable function $n:\Omega\to\N$ such that for $\prob$-a.e. $\omega\in\Omega$, for all $n\geq n(\omega)$, all $m\geq0$ and all $|p|\leq n$,
    $$\supnorm{\normopera_{\theta^{p}\omega}^{n}f_{\theta^{p}\omega}-\normopera_{\theta^{p-m}\omega}^{n+m}g_{\theta^{p-m}\omega}}\leq\left(e^{\Delta\bar{D}^{n}}-1\right)\supnorm{\normopera_{\theta^{p}\omega}^{n}f_{\theta^{p}\omega}}.$$
\end{corollary}
\section{Conformal and invariant measures}\label{sec6}
\subsection{Construction of density functions}
For a function $f\in BV_{\Omega}(I)$ we say $f$ belongs to the class $\mathcal{T}$ provided for every $\epsilon>0$ and every $n\in\Z$ there exists a measurable function $V_{f,\epsilon}:\Omega\to(0,\infty)$ such that 

(1) $\vari(f_{\theta^{n}\omega})\leq V_{f,\epsilon}(\omega)e^{|n|\epsilon}$;

(2) $F_{\omega}(|f_{\theta^{n}\omega}|)\geq V_{f,\epsilon}(\omega)^{-1}e^{-|n|\epsilon}$.


We denote by $\mathcal{T}^+$ these functions $f\in\mathcal{T}$ whose $f_{\omega}\geq0$ for every $\omega\in\Omega$.
\begin{proposition}\label{prop6.31}
    For $\prob$-a.e. $\omega\in\Omega$, there exists a measurable function $\lambda_{\omega}>0$ and $q_{\omega}\in BV(X_\omega)$ such that $$\opera_{\omega}q_{\omega}=\lambda_{\omega}q_{\theta\omega}$$ and $$F_{\omega}(q_\omega)=1.$$ Moreover, we have $\log\lambda_{\omega}\in L^1(\prob)$ and $\rho_{\omega}\leq\lambda_{\omega}$ for $\prob$-a.e. $\omega\in\Omega$.
\end{proposition}
\begin{proof}
    For $f\in\mathcal{T}^+$, let $$f_{\omega,n}:=\frac{\rho_{\theta^{-n}\omega}^{n}}{F_{\omega}(\opera_{\theta^{-n}\omega}^{n}f_{\theta^{-n}\omega})}f_{\theta^{-n}\omega}\in\Lambda_{\theta^{-n}\omega,+}$$ be defined for all $n\geq1$. Since $\rho_{\theta^{-n}\omega}^{n}F_{\theta^{-n}\omega}(f_{\theta^{-n}\omega})\leq F_{\omega}(\opera_{\theta^{-n}\omega}^{n}f_{\theta^{-n}\omega})$ by the second inequality in Lemma \ref{lem6.6}, then there exists a measurable $V_{f,\epsilon}:\Omega\to(0,\infty)$ such that $$\vari(f_{\omega,n})=\frac{\rho_{\theta^{-n}\omega}^{n}}{F_{\omega}(\opera_{\theta^{-n}\omega}^{n}f_{\theta^{-n}\omega})}\vari(f_{\theta^{-n}\omega})\leq\frac{\vari(f_{\theta^{-n}\omega})}{F_{\theta^{-n}\omega}(f_{\theta^{-n}\omega})}\leq V_{f,\epsilon}(\omega)^{2}e^{2n\epsilon}=V_{f,\epsilon'}(\omega)e^{n\epsilon'}$$ holds for sufficiently large $n\geq1$, where $V_{f,\epsilon'}(\omega):=V_{f,\epsilon}(\omega)^{2}$ and $\epsilon'=2\epsilon$. Therefore, set function $V(\omega):=V_{f,\epsilon'}(\omega)^{2}$ in Corollary \ref{coro6.30}, as $n\geq n(\omega)$ tends to infinity, the sequence $$\{\tilde{f}_{\omega,n}\}_{n=1}^{\infty}=\{\normopera_{\theta^{-n}\omega}^{n}f_{\omega,n}\}_{n=1}^{\infty}=\left\{\frac{\opera_{\theta^{-n}\omega}^{n}f_{\theta^{-n}\omega}}{F_{\omega}(\opera_{\theta^{-n}\omega}^{n}f_{\theta^{-n}\omega})}\right\}_{n=1}^{\infty}$$ is a Cauchy sequence in cone $\Lambda_{\omega,+}$, hence the limit exists and is denoted by $q_{\omega,f}\in\Lambda_{\omega,+}$, $$q_{\omega,f}:=\lim_{n\to\infty}\frac{\opera_{\theta^{-n}\omega}^{n}f_{\theta^{-n}\omega}}{F_{\omega}(\opera_{\theta^{-n}\omega}^{n}f_{\theta^{-n}\omega})}\geq0.$$ We then have $F_{\omega}(q_{\omega,f})=1$. For $$\opera_{\omega}q_{\omega,f}=\lim_{n\to\infty}\frac{\opera_{\theta^{-n}\omega}^{n+1}f_{\theta^{-n}\omega}}{F_{\omega}(\opera_{\theta^{-n}\omega}^{n}f_{\theta^{-n}\omega})},$$ we observe that $$\frac{F_{\theta\omega}(\opera_{\theta^{-n}\omega}^{n+1}f_{\theta^{-n}\omega})}{F_{\omega}(\opera_{\theta^{-n}\omega}^{n}f_{\theta^{-n}\omega})}\geq\frac{\rho_{\omega}F_{\omega}(\opera_{\theta^{-n}\omega}^{n}f_{\theta^{-n}\omega})}{F_{\omega}(\opera_{\theta^{-n}\omega}^{n}f_{\theta^{-n}\omega})}=\rho_{\omega},$$ and $$\frac{F_{\theta\omega}(\opera_{\theta^{-n}\omega}^{n+1}f_{\theta^{-n}\omega})}{F_{\omega}(\opera_{\theta^{-n}\omega}^{n}f_{\theta^{-n}\omega})}\leq\frac{\supnorm{\opera_{\omega}\ind_{\omega}}F_{\omega}(\opera_{\theta^{-n}\omega}^{n}f_{\theta^{-n}\omega})}{F_{\omega}(\opera_{\theta^{-n}\omega}^{n}f_{\theta^{-n}\omega})}=\supnorm{\opera_{\omega}\ind_{\omega}}.$$ Recall that $\supnorm{\opera_{\omega}\ind_{\omega}}$ is bounded above since the potential $\varphi$ is summable and $S_{\omega}^{1}<\infty$. Therefore the sequence $$\left\{\frac{F_{\theta\omega}(\opera_{\theta^{-n}\omega}^{n+1}f_{\theta^{-n}\omega})}{F_{\omega}(\opera_{\theta^{-n}\omega}^{n}f_{\theta^{-n}\omega})}\right\}_{n=1}^{\infty}$$ is bounded and hence there exists a subsequence $\{n_k\}_{k=1}^{\infty}$ such that the sequence has a limit $\lambda_{\omega,f}$ along $\{n_k\}_{k=1}^{\infty}$ as $k\to\infty$: $$\lambda_{\omega,f}:=\lim_{k\to\infty}\frac{F_{\theta\omega}(\opera_{\theta^{-n_k}\omega}^{n_{k}+1}f_{\theta^{-n_k}\omega})}{F_{\omega}(\opera_{\theta^{-n_k}\omega}^{n_k}f_{\theta^{-n_k}\omega})}.$$ We have 
    \begin{equation*}
        \begin{aligned}
            \opera_{\omega}q_{\omega,f}&=\lim_{k\to\infty}\frac{\opera_{\theta^{-n_k}\omega}^{n_k+1}f_{\theta^{-n_k}\omega}}{F_{\omega}(\opera_{\theta^{-n_k}\omega}^{n_k}f_{\theta^{-n_k}\omega})} \\
            &=\lim_{k\to\infty}\frac{\opera_{\theta^{-n_k}\omega}^{n_k+1}f_{\theta^{-n_k}\omega}}{F_{\theta\omega}(\opera_{\theta^{-n_k}\omega}^{n_k+1}f_{\theta^{-n_k}\omega})}\cdot\lim_{k\to\infty}\frac{F_{\theta\omega}(\opera_{\theta^{-n_k}\omega}^{n_k+1}f_{\theta^{-n_k}\omega})}{F_{\omega}(\opera_{\theta^{-n_k}\omega}^{n_k}f_{\theta^{n_k}\omega})} \\
            &=\lambda_{\omega,f}q_{\theta\omega,f}.
        \end{aligned}
    \end{equation*}
    It is natural to replace $n_k$ here with $n$, and then we have $$\lambda_{\omega,f}=\lim_{n\to\infty}\frac{F_{\theta\omega}(\opera_{\theta^{-n}\omega}^{n+1}f_{\theta^{-n}\omega})}{F_{\omega}(\opera_{\theta^{-n}\omega}^{n}f_{\theta^{-n}\omega})}.$$

    We claim $q_{\omega,f}$ and $\lambda_{\omega,f}$ produced here do not depend on function $f\in\mathcal{T}^+$. For functions $f,g\in\mathcal{T}^+$, set $$V(\omega):=\max\left\{V_{f,\epsilon'}(\omega)^2,V_{g,\epsilon'}(\omega)^2\right\}$$ and apply Proposition \ref{prop6.29}, by triangle inequality, we have
    \begin{equation*}
        \begin{aligned}
            \Theta_{\Lambda_{\omega,+}}(q_{\omega,f},q_{\omega,g})&\leq\Theta_{\Lambda_{\omega,+}}(q_{\omega,f},\tilde{f}_{\omega,n})+\Theta_{\Lambda_{\omega,+}}(\tilde{f}_{\omega,n},\tilde{g}_{\omega,n})+\Theta_{\Lambda_{\omega,+}}(\tilde{g}_{\omega,n},q_{\omega,g}) \\
            &\leq 3\Delta D^{n}
        \end{aligned}
    \end{equation*} for sufficiently large $n\geq1$ and it hence goes to $0$. It yields $$\supnorm{q_{\omega,f}-q_{\omega,g}}\leq\supnorm{q_{\omega,f}}\left(\exp(\Theta_{\Lambda_{\omega,+}}(q_{\omega,f},q_{\omega,g}))-1\right)\to0$$ as $n\to\infty$, hence $q_{\omega,f}=q_{\omega,g}$ for all $f,g\in\mathcal{T}^+$, which also implies $\lambda_{\omega,f}=\lambda_{\omega,g}$. We denote the unique value by $q_\omega$, $\lambda_\omega$ respectively and hence $q_{\omega}\in BV(X_\omega)$, $F_{\omega}(q_\omega)=1$ and $\opera_{\omega}q_{\omega}=\lambda_{\omega}q_{\theta\omega}$ for $\prob$-a.e. $\omega\in\Omega$.

    From previous calculations we see $0<\rho_\omega\leq\lambda_{\omega}\leq\supnorm{\opera_{\omega}\ind_{\omega}}$, hence $\lambda_\omega>0$ for $\prob$-a.e. $\omega\in\Omega$. Since $\log\rho_{\omega}^{n},\log\supnorm{\opera_{\omega}\ind_\omega}\in L^1(\prob)$, we also deduce $\log\lambda_{\omega}\in L^1(\prob)$. The function $\omega\mapsto\lambda_\omega$ is measurable since the sequence where $f_{\theta^{-n}\omega}$ replaced with $\ind_{\theta^{-n}\omega}$ is measurable; then the limit is measurable. The measurability of maps $\omega\mapsto\supnorm{q_\omega}$ and $\omega\mapsto\inf q_\omega$ are also followed in the same way.
\end{proof}
We apply this result repeatedly to obtain the same equations for $n$ iterates.
\begin{corollary}
    For all $n\geq1$, we have $$\opera_{\omega}^{n}q_\omega=\lambda_{\omega}^{n}q_{\theta^{n}\omega}$$ where $\lambda_{\omega}^{n}:=\prod_{i=0}^{n-1}\lambda_{\theta^{i}\omega}\geq\rho_{\omega}^{n}$.
\end{corollary}
We deduce such functions $\{q_\omega\}_{\omega\in\Omega}$ are strictly positive on the support set $S(1,\infty)$.
\begin{proposition}\label{prop6.33}
    For $\prob$-a.e. $\omega\in\Omega$, we have $$\underset{S(1,\infty)}{\inf}\,q_\omega>0.$$
\end{proposition}
\begin{proof}
    Since $F_{\omega}(q_\omega)=1$, then for sufficiently large $n\geq1$ we have $$\underset{S(n,n)}{\inf}\,\opera_{\omega}^{n}q_\omega>0,$$ which implies $\inf_{K_{\omega,n-1}}q_\omega>0$. Therefore, for $\prob$-a.e. $\omega\in\Omega$ and all $n\geq1$,
    \begin{equation*}
        \begin{aligned}
            \underset{S(1,\infty)}{\inf}\,q_\omega&=(\lambda_{\theta^{-n}\omega}^{n})^{-1}\underset{S(1,\infty)}{\inf}\,\opera_{\theta^{-n}\omega}^{n}q_{\theta^{-n}\omega} \\
            &\geq (\lambda_{\theta^{-n}\omega}^{n})^{-1}\underset{S(1,\infty)}{\inf}\,\opera_{\theta^{-n}\omega}^{n}\ind_{\theta^{-n}\omega}\cdot\underset{K_{\theta^{-n}\omega,n-1}}{\inf}\,q_{\theta^{-n}\omega}>0
        \end{aligned}
    \end{equation*} completes the proof since each term is strictly positive.
\end{proof}
\subsection{Construction of conformal measures}
\begin{proposition}\label{prop6.34}
    For $\prob$-a.e. $\omega\in\Omega$, the random functional $F_\omega$ is positive, linear, and $$F_{\theta\omega}(\opera_{\omega}f)=\lambda_{\omega}F_{\omega}(f)$$ for all $f\in BV(X_\omega)$. Moreover, we have $\lambda_\omega=\rho_\omega=F_{\theta\omega}(\opera_\omega \ind_\omega)$ for $\prob$-a.e. $\omega\in\Omega$.
\end{proposition}
\begin{proof}
    The random functional $F_\omega$ is obviously positive. For arbitrary two sequences $\{x_n\}_{n=0}^{\infty}$ and $\{y_n\}_{n=0}^{\infty}$ with $x_n,y_n\in S(n,n)$, we have
    \begin{equation*}
        \begin{aligned}
            \lim_{n\to\infty}\bigg|\frac{\opera_{\omega}^{n}f}{\opera_{\omega}^{n}\ind_{\omega}}(x_n)-\frac{\opera_{\omega}^{n}f}{\opera_{\omega}^{n}\ind_{\omega}}(y_n)\bigg|&=\lim_{n\to\infty}\bigg|\frac{\opera_{\omega}^{n}f}{\opera_{\omega}^{n}\ind_{\omega}}(y_n)\bigg|\cdot\bigg|\frac{\opera_{\omega}^{n}f}{\opera_{\omega}^{n}\ind_{\omega}}(x_n)\cdot\frac{\opera_{\omega}^{n}\ind_\omega}{\opera_{\omega}^{n}f}(y_n)-1\bigg| \\
            &\leq \supnorm{f}\limsup_{n\to\infty}\,\exp\left(\Theta_{\Lambda_{\theta^{n}\omega,+}}(\normopera_{\omega}^{n}f,\normopera_{\omega}^{n}\ind_\omega)-1\right),
        \end{aligned}
    \end{equation*} the right side tends to $0$ as $n\to\infty$. Therefore, it is sufficient to rewrite $$F_{\omega}(f):=\lim_{n\to\infty}\frac{\opera_{\omega}^{n}f}{\opera_{\omega}^{n}\ind_\omega}(x_n)$$ for all $f\in\Lambda_{\omega,+}$ and all $x_n\in S(n,n)$. Hence the linearity follows. It is natural to extend it to general function $f\in BV(X_\omega)$ by writing $f=f^{+}-f^{-}$ with $f^+:=\max\{0,f\},f^-:=\max\{0,-f\}\in\Lambda_{\omega,+}$, then we have $F_{\omega}(f)=F_{\omega}(f^+)-F_{\omega}(f^-)$. Furthermore,
    \begin{equation*}
        \begin{aligned}
            F_{\theta\omega}(\opera_{\omega}f)&=\lim_{n\to\infty}\frac{\opera_{\omega}^{n+1}f}{\opera_{\theta\omega}^{n}\ind_{\theta\omega}}(x_{n+1}) \\
            &=\lim_{n\to\infty}\frac{\opera_{\omega}^{n+1}f}{\opera_{\omega}^{n+1}\ind_\omega}(x_{n+1})\cdot\lim_{n\to\infty}\frac{\opera_{\omega}^{n+1}\ind_\omega}{\opera_{\theta\omega}^{n}\ind_{\theta\omega}}(x_{n+1}) \\
            &=F_{\omega}(f)\cdot F_{\theta\omega}(\opera_{\omega}\ind_\omega).
        \end{aligned}
    \end{equation*} Let $f=q_\omega\in BV(X_\omega)$, by $F_\omega(q_\omega)=1$ and $\opera_{\omega}q_{\omega}=\lambda_{\omega}q_{\theta\omega}$, we obtain $F_{\theta\omega}(\opera_{\omega}q_{\omega})=\lambda_{\omega}F_{\theta\omega}(q_{\theta\omega})=\lambda_\omega$, while on the other hand, $F_{\omega}(q_\omega)F_{\theta\omega}(\opera_{\omega}\ind_\omega)=F_{\theta\omega}(\opera_{\omega}\ind_\omega)$. Recall that $\rho_\omega=F_{\theta\omega}(\opera_{\omega}\ind_\omega)$, we have $\lambda_\omega=\rho_\omega=F_{\theta\omega}(\opera_{\omega}\ind_\omega)$ for $\prob$-a.e. $\omega\in\Omega$.
\end{proof}
\begin{proposition}\label{prop6.35}
    For $\prob$-a.e. $\omega\in\Omega$, there exists an atomless Borel probability measure $\nu_{\omega}\in\M(I)$ such that $$F_{\omega}(f)=\int_{X_\omega}f\,\dd\nu_\omega$$ for all $f\in BV(X_\omega)$ and $\nu_{\omega}(X_\omega)=1$. We have $$\nu_{\theta\omega}(\opera_{\omega}f)=\lambda_{\omega}\nu_{\omega}(f)$$ for all $f\in BV(X_\omega)$ and $\supp(\nu_\omega)\subseteq K_{\omega,\infty}$ for $\prob$-a.e. $\omega\in\Omega$.
\end{proposition}
\begin{proof}
    We adapt from \cite[Lemma 4.10]{liverani3} in covering weighted deterministic systems. From Proposition \ref{prop6.34} we see $F_\omega$ is positive and linear on $BV(X_\omega)$ for every $\omega\in\Omega$. Since $X_\omega$ is compact, every continuous function $f\in C(X_\omega)$ can be approximated by a sequence of functions $\{f_n\}_{n\geq1}$ in $BV(X_\omega)$ in the supremum norm and so are the random continuous functions. For each $\omega\in\Omega$, we can extend $F_\omega$ to $C(X_\omega)$. By Riesz representation theorem, we can identify it with a Borel probability measure $\nu_\omega$ on $X_\omega$ for each $\omega\in\Omega$. We want to show $\nu_{\omega}(Z)=F_{\omega}(\ind_Z)$ for all non-degenerate intervals $Z\subseteq X_\omega$, we assume $Z$ is an interval with $\overline{Z}=[a_1,a_2]$. Let $\epsilon>0$ and let $V_1,V_2\subseteq I$ be open interval containing $a_1,a_2$ respectively, we assume $F_{\omega}(\ind_{V_1})<\frac{\epsilon}{2}$ and $F_{\omega}(\ind_{V_2})<\frac{\epsilon}{2}$. Such open intervals can be constructed by the following way.

    Since the potential is contracting, fo all $n\geq1$ and all intervals $U\in\parti_{\omega}^{n}$, we have that \[F_{\omega}(\ind_U)\leq(\rho_{\omega}^{n})^{-1}F_{\theta^{n}\omega}(\opera_{\omega}^{n}\ind_U)\leq\frac{\supnorm{g_{\omega}^{n}}}{\rho_{\omega}^{n}}<\frac{\epsilon}{4}\] for given $\epsilon>0$ since $T_{\omega}^{n}$ is injective on every element of $\parti_{\omega}^{n}$. Notice that $\parti_{\omega}^{n}$ is a partition of $X_\omega$, we have $\sum_{U\in\parti_{\omega}^{n}}F_{\omega}(\ind_U)=1$. It allows us to find a subcollection $\tilde{\parti}_{1}$ of $\parti_{\omega}^{n}$ such that $\sum_{U\in\tilde{\parti}_{1}}F_{\omega}(\ind_U)<\frac{\epsilon}{2}$ and numbers $b_1,c_1\in\mathrm{int}(\bigcup_{U\in\tilde{\parti}_1}U)$ with $b_1\leq a_1\leq c_1$. Choose $V_1:=\mathrm{int}(\bigcup_{U\in\tilde{\parti}_1}U)$ as the first open interval containing $a_1$. In the same steps, we choose the second open interval $V_2$ containing $a_2$.

    Let $f\in BV(X_\omega)\cap C(X_\omega)$ be a function such that \[\ind_{Z}\leq f|_{V_1\cup V_2\cup Z}\leq\ind\] and $f\equiv0$ outside the union. We calculate \[\nu_{\omega}(\ind_Z)\leq F_{\omega}(f)\leq F_{\omega}(\ind_{V_1})+F_{\omega}(\ind_{V_2})+F_{\omega}(\ind_Z)\leq\epsilon+F_{\omega}(\ind_Z),\] as $\epsilon$ is arbitrary, we have $\nu_{\omega}(Z)\leq F_{\omega}(\ind_Z)$ for all $Z\subseteq X_\omega$. Now for any two disjoint intervals $Z_{1},Z_{2}\subseteq X_{\omega}$ with $Z_{1}\cup Z_{2}=X_\omega$, we have \[1=\nu_{\omega}(Z_1)+\nu_{\omega}(Z_2)\leq F_{\omega}(\ind_{Z_1})+F_{\omega}(\ind_{Z_2})=1.\] Then set $Z_1=Z$ and $Z_2=X_{\omega}\setminus Z$ we obtain $F_{\omega}(\ind_Z)=1-\nu_{\omega}(X_{\omega}\setminus Z)=\nu_{\omega}(Z)$ for all $Z\subseteq X_\omega$. Since $Z$ is non-degenerate, we deduce $\nu_\omega$ is atomless.
    
    Recall that $\ind_{\omega}:=\ind_{J_\omega}$ where $J_\omega:=X_{\omega}\setminus H_\omega$, we then have $F_{\omega}(\ind_{\omega})=1=\nu_{\omega}(J_\omega)$ by above. Therefore, for every $\omega\in\Omega$, $\nu_\omega$ is Borel probability measure on $X_\omega$ with full measure on $J_\omega$. Since the singular set $S_\omega$ is countable and disjoint from $X_\omega$, then $\nu_{\omega}(S_\omega)=0$. By this we can extend it to a Borel probability measure on $I$ with full measure on $J_\omega$ and we use the same notation $\nu_\omega$.
    
    We shall extend $F_\omega$ to $BV(I)$ by defining $F_{\omega}(f):=F_{\omega}(f|_{X_\omega})$ for every $f\in BV(I)$. We need to show $\nu_{\omega}(f)=F_{\omega}(f)$ for all $f\in BV(I)$. Let $\epsilon>0$, since the jump discontinuities of $f$ can be larger than $\epsilon$ on only a finite set, choose a piecewise continuous function $f_\epsilon:I\to\R$ such that $\supnorm{f-f_\epsilon}<\epsilon$. Since $\nu_\omega(f_\epsilon)=F_{\omega}(f_\epsilon)$, we have $$|\nu_\omega(f)-F_{\omega}(f)|\leq|\nu_\omega(f)-\nu_\omega(f_\epsilon)|+|\nu_\omega(f_\epsilon)-F_{\omega}(f_\epsilon)|+|F_{\omega}(f_\epsilon)-F_{\omega}(f)|\leq 2\supnorm{f-f_\epsilon}<2\epsilon.$$ We then have $\nu_\omega(f)=F_\omega(f)$ for all $f\in BV(I)$. Next we show the measure $\nu_\omega$ is supported in the survivor set.

    Fix $n\geq1$, suppose $f\in L^1(\nu_{\omega,c})$ on $I$ and $f\equiv0$ on survivor set $K_{\omega,n-1}$, the operators are also applicable to integrable functions, then $$\int_{I}f\,\dd\nu_\omega=(\lambda_{\omega}^{n})^{-1}\int_{I}\opera_{\omega}^{n}f\,\dd\nu_{\theta^{n}\omega}=(\lambda_{\omega}^{n})^{-1}\int_{I}\opera_{\omega,c}^{n}(f\ind_{K_{\omega,n-1}})\,\dd\nu_{\theta^{n}\omega}=0,$$ since $\lambda_{\omega}^{n}$ is finite, let $n\to\infty$, we have $\supp(\nu_\omega)\subseteq K_{\omega,\infty}$ for $\prob$-a.e. $\omega\in\Omega$.
\end{proof}
By \cite{atnip4,denker2}, the measures are conformal.
\begin{corollary}
    The family of Borel probability measures $\{\nu_\omega\}_{\omega\in\Omega}$ is conformal, for each Borel subset $A$ on which $T_{\omega}^{n}|_A$ is one-to-one, we have $$\nu_{\theta^{n}\omega}(T_{\omega}^{n}(A))=\lambda_{\omega}^{n}\int_{A}e^{-S_{n}\varphi_\omega}\,\dd\nu_\omega.$$
\end{corollary}
As a consequence of Proposition \ref{prop6.31}, Proposition \ref{prop6.34} and Proposition \ref{prop6.38}, we have the following results on the semi-normalized operator.
\begin{corollary}
    For the semi-normalized open operator $\normopera_{\omega}:=\rho_{\omega}^{-1}\opera_{\omega}=\lambda_{\omega}^{-1}\opera_{\omega}$, we have $$\normopera_{\omega}q_{\omega}=q_{\theta\omega},\quad \nu_{\theta\omega}(\normopera_{\omega}f)=\nu_{\omega}(f).$$ Moreover, for $n\geq1$ we have $$\normopera_{\omega}^{n}q_{\omega}=q_{\theta^{n}\omega},\quad \nu_{\theta^{n}\omega}(\normopera_{\omega}^{n}f)=\nu_{\omega}(f).$$
\end{corollary}
We define the fully normalized open operator $\fullnormopera_{\omega}:BV(X_\omega)\to BV(X_{\theta\omega})$ by $$\fullnormopera_{\omega}f:=\frac{1}{q_{\theta\omega}}\normopera_{\omega}(fq_\omega)=\frac{1}{\lambda_{\omega}q_{\theta\omega}}\opera_{\omega}(fq_\omega),$$ then we have $\fullnormopera_{\omega}\ind_\omega=\ind_{\theta\omega}$. For $\omega\in\Omega$, define the Borel probability measure $\mu_{\omega}\in\M(I)$ by $$\mu_{\omega}(f):=\int_{K_{\omega,\infty}}fq_{\omega}\,\dd\nu_\omega$$ for all $f\in L^1(\nu_\omega)$, for simplicity, we write it as $\mu_{\omega}:=q_{\omega}\nu_{\omega}$. We observe $\mu_\omega$ is also atomless and it is supported in $K_{\omega,\infty}$. Furthermore, $\mu_\omega$ is absolutely continuous with respect to $\nu_\omega$ for every $\omega\in\Omega$. We close this section with the following result by a simple calculation as that in \cite[Proposition 8.11]{atnip1}.
\begin{proposition}\label{prop6.38}
    The family of Borel probability measures $\{\mu_\omega\}_{\omega\in\Omega}$ is $T$-invariant, that is,
    $$\int_{K_{\omega,\infty}}f\circ T_\omega\,\dd\mu_\omega=\int_{K_{\theta\omega,\infty}}f\,\dd\mu_{\theta\omega}$$ for all $f\in L^1(\mu_{\theta\omega})$.
\end{proposition}
\begin{proof}
    For all $f\in L^1(\mu_{\theta\omega})$, we have
    \begin{equation*}
        \begin{aligned}
            \int_{K_{\theta\omega,\infty}}f\,\dd\mu_{\theta\omega}&=\int_{K_{\theta\omega,\infty}}fq_{\theta\omega}\,\dd\nu_{\theta\omega}=\int_{K_{\theta\omega,\infty}}f\cdot\normopera_{\omega}q_{\omega}\,\dd\nu_{\theta\omega} \\
            &=\int_{K_{\theta\omega,\infty}}\normopera_{\omega}(f\circ T_{\omega}\cdot q_{\omega})\,\dd\nu_{\theta\omega}=\int_{K_{\omega,\infty}}f\circ T_{\omega}\cdot q_{\omega}\,\dd\nu_{\omega} \\
            &=\int_{K_{\omega,\infty}}f\circ T_\omega\,\dd\mu_{\omega},
        \end{aligned}
    \end{equation*} which exactly means $\mu_{\omega}\circ T_{\omega}^{-1}=\mu_{\theta\omega}$ for $\prob$-a.e. $\omega\in\Omega$, hence $\{\mu_\omega\}_{\omega\in\Omega}$ is $T$-invariant.
\end{proof}
\section{Decay of correlations}\label{sec7}
We have obtained conformal measures and density function for piecewise monotone random interval map with countably many branches with hole assigned. Now we establish the results of exponential decay of correlations of two observables in this open system. From \cite[Section 2]{gundlach}, we say a function $f:\Omega\to\R$ is \emph{tempered} if $$\lim_{|n|\to\infty}\frac{1}{|n|}\log|f(\theta^{n}\omega)|=0$$ for $\prob$-a.e. $\omega\in\Omega$ and all $n\in\Z$. It is equivalent to the fact that for every $\epsilon>0$, there exists a measurable finite $C(\omega,\epsilon)>0$ such that $f(\theta^{n}\omega)\leq C(\omega,\epsilon)e^{|n|\epsilon}$ holds for all $n\in\Z$. We cite a lemma to state that the $BV$-norm of density $q_\omega$ does not grow too much along the $\theta$-orbit of fibers, and it is bounded above by an exponential term with a measurable random coefficient; briefly, the $BV$-norm is tempered.
\begin{lemma}[{\cite[Lemma 8.5]{atnip1}}]\label{lem6.39}
    For every $\epsilon>0$ there exists a measurable random constant $C(\omega,\epsilon)>0$ such that for $\prob$-a.e. $\omega\in\Omega$ and $n\in\Z$, we have $$\bvnorm{q_{\theta^{n}\omega}}\leq C(\omega,\epsilon)e^{|n|\epsilon}.$$
\end{lemma}
We have a useful lemma from \cite[Corollary 8.9]{atnip1}.
\begin{lemma}\label{lem6.40}
    For $\prob$-a.e. $\omega\in\Omega$ and for nonnegative $f\in BV_{\Omega}(I)$ with $\nu_{\omega}(f_\omega)=1$, there exists $\iota\in(0,1)$ and a finite measurable function $C_{f}:\Omega\to(0,\infty)$ such that for all $n\geq1$ and all $|p|\leq n$, if $\vari(f_{\theta^{n+p}\omega})\leq V_{f,\epsilon}(\omega)e^{n\epsilon}$ for sufficiently small $\epsilon>0$, then we have \[\supnorm{\normopera_{\theta^{p}\omega}^{n}f_{\theta^{p}\omega}-q_{\theta^{n+p}\omega}}\leq C_{f}(\omega)\bvnorm{f_{\theta^{p}\omega}}\iota^n.\]
\end{lemma}
\begin{proof}
    By $\nu_{\omega}(f_\omega)=1$ we have $\supnorm{f_\omega}\geq1$ and hence $\bvnorm{f_\omega}\geq1$, by $\nu_{\omega}(q_\omega)=1$ we also have $\supnorm{q_\omega}\geq1$. For $\prob$-a.e. $\omega\in\Omega$ and all $n\geq1$, we have
    \begin{equation*}
        \begin{aligned}
            \supnorm{\normopera_{\theta^{p}\omega}^{n}f_{\theta^{p}\omega}-q_{\theta^{n+p}\omega}}&\leq\supnorm{\normopera_{\theta^{p}\omega}^{n}f_{\theta^{p}\omega}}+\supnorm{q_{\theta^{n+p}\omega}} \\
            &\leq 2L_{\theta^{p}\omega}^{n}\bvnorm{f_{\theta^{p}\omega}}\supnorm{q_{\theta^{n+p}\omega}} \\
            &\leq 2L_{\theta^{-n}\omega}^{3n}\bvnorm{f_{\theta^{p}\omega}}\supnorm{q_{\theta^{n+p}\omega}}
        \end{aligned}
    \end{equation*} by Lemma \ref{lem6.12}. According to Lemma \ref{lem6.39} we have $\vari(q_{\theta^{n+p}\omega})\leq C(\omega,\epsilon)e^{|n+p|\epsilon}$ for sufficiently small $\epsilon>0$. Take $V(\omega):=\max\{V_{f,\epsilon}(\omega),C(\omega,\epsilon)\}$ in Proposition \ref{prop6.29}, we have \[\Theta_{\Lambda_{\theta^{n+p}\omega,+}}\left(\normopera_{\theta^{p}\omega}^{n}f_{\theta^{p}\omega},q_{\theta^{n+p}\omega}\right)=\Theta_{\Lambda_{\theta^{n+p}\omega,+}}\left(\normopera_{\theta^{p}\omega}^{n}f_{\theta^{p}\omega},\normopera_{\theta^{p}\omega}^{n}q_{\theta^{p}\omega}\right)\leq\Delta D^n\] for all $n\geq n(\omega)$. Applying Lemma \ref{lem4.3} with $\rho(\cdot)=\nu_{\theta^{p+n}\omega}(\cdot)$ and $\Vert\cdot\Vert=\supnorm{\cdot}$, we get
    \begin{equation*}
        \begin{aligned}
            \supnorm{\normopera_{\theta^{p}\omega}^{n}f_{\theta^{p}\omega}-q_{\theta^{n+p}\omega}}&\leq\supnorm{q_{\theta^{n+p}\omega}}\left(\exp\left(\Theta_{\Lambda_{\theta^{n+p}\omega,+}}\left(\normopera_{\theta^{p}\omega}^{n}f_{\theta^{p}\omega},q_{\theta^{n+p}\omega}\right)\right)-1\right) \\
            &\leq \supnorm{q_{\theta^{n+p}\omega}}(e^{\Delta D^n}-1) \\
            &\leq \supnorm{q_{\theta^{n+p}\omega}}\tilde{\iota}^n
        \end{aligned}
    \end{equation*} for some $\tilde{\iota}\in(0,1)$. Combine the two inequalities together, we have
    \begin{equation*}
        \begin{aligned}
            \supnorm{\normopera_{\theta^{p}\omega}^{n}f_{\theta^{p}\omega}-q_{\theta^{n+p}\omega}}&\leq 2\underset{0\leq i\leq n(\omega)}{\max}L_{\theta^{p}\omega}^{i}\tilde{\iota}^{-i}\bvnorm{f_{\theta^{p}\omega}}\supnorm{q_{\theta^{n+p}\omega}}\tilde{\iota}^n \\
            &= 2L_{\theta^{p}\omega}^{n(\omega)}\tilde{\iota}^{-n(\omega)}\bvnorm{f_{\theta^{p}\omega}}\supnorm{q_{\theta^{n+p}\omega}}\tilde{\iota}^n.
        \end{aligned}
    \end{equation*} Since \[\supnorm{q_{\theta^{n+p}\omega}}\leq\bvnorm{q_{\theta^{n+p}\omega}}\leq C(\omega,\epsilon)e^{|n+p|\epsilon}\leq C(\omega,\epsilon)e^{2n\epsilon},\] take \[C_{f}(\omega):=2L_{\theta^{n(\omega)}\omega}^{3n(\omega)}\tilde{\iota}^{-n(\omega)}C(\omega,\epsilon)>0\] and let $\iota:=e^{2\epsilon}\tilde{\iota}\in(0,1)$, we obtain the final inequality. The measurability of $C_{f}(\omega)$ is clear since it is a multiplication of measurable functions.
\end{proof}
We use the above lemma to prove the following theorem, excluding the condition of $\nu_{\omega}(f_\omega)=1$ for $\prob$-a.e. $\omega\in\Omega$.
\begin{theorem}\label{thm6.41}
    There exists a measurable finite function $\hat{C}_{f}:\Omega\to(0,\infty)$ and $\iota\in(0,1)$ from Lemma \ref{lem6.40} such that for all $f\in BV_{\Omega}(I)$ with $\supnorm{f_\omega},\log\vari(f_\omega),\log\inf|f_\omega|\in L^1(\prob)$ for every $\omega\in\Omega$, for all $n\geq1$ and all $|p|\leq n$, we have
    $$\supnorm{\normopera_{\theta^{p}\omega}^{n}f_{\theta^{p}\omega}-\nu_{\theta^{p}\omega}(f_{\theta^{p}\omega})q_{\theta^{p+n}\omega}}\leq\hat{C}_{f}(\omega)\bvnorm{f_{\theta^{p}\omega}}\iota^{n}.$$
\end{theorem}
\begin{proof}
    We refer to \cite[Theorem 10.3]{atnip1} for more details. One can observe such $f\in BV_{\Omega}(I)$ with conditions afterward belongs to the class $\mathcal{T}$. Suppose $f\geq0$, then \[\tilde{f}:=\{\tilde{f}_{\omega}\}_{\omega\in\Omega}=\left\{\frac{f_\omega}{\nu_{\omega}(f_\omega)}\right\}_{\omega\in\Omega}\in\mathcal{T}^+\] and $\nu_{\omega}(\tilde{f}_\omega)=1$. Apply Lemma \ref{lem6.40}, we obtain
    \[\bigg\Vert\normopera_{\theta^{p}\omega}^{n}\left(\frac{f_{\theta^{p}\omega}}{\nu_{\theta^{p}\omega}(f_{\theta^{p}\omega})}\right)-q_{\theta^{p+n}\omega}\bigg\Vert_{\infty}\leq C_{\tilde{f}}(\omega)\bigg\Vert\frac{f_{\theta^{p}\omega}}{\nu_{\theta^{p}\omega}(f_{\theta^{p}\omega})}\bigg\Vert_{BV}\iota^{n}\] for all $n\geq1$. Multiply both sides by $\nu_{\theta^{p}\omega}(f_{\theta^{p}\omega})$, it yields
    \[\supnorm{\normopera_{\theta^{p}\omega}^{n}(f_{\theta^{p}\omega})-\nu_{\theta^{p}\omega}(f_{\theta^{p}\omega})q_{\theta^{p+n}\omega}}\leq C_{\tilde{f}}(\omega)\bvnorm{f_{\theta^{p}\omega}}\iota^{n}\] with $C_{\tilde{f}}(\omega)$ measurable. For $f\in\mathcal{T}$, write it as $f=f^{+}-f^{-}$ with $f^+,f^-\geq0$, apply triangle inequality, we obtain \[\supnorm{\normopera_{\theta^{p}\omega}^{n}(f_{\theta^{p}\omega})-\nu_{\theta^{p}\omega}(f_{\theta^{p}\omega})q_{\theta^{p+n}\omega}}\leq \hat{C}_{f}(\omega)\bvnorm{f_{\theta^{p}\omega}}\iota^{n}\] where \[\hat{C}_{\tilde{f}}(\omega):=C_{\tilde{f^+}}(\omega)+C_{\tilde{f^-}}(\omega)\] is finite and measurable.
\end{proof}
It directly leads to the corollary below with respect to the fully normalized operator $\fullnormopera_\omega$ and probability measure $\mu_\omega$; for instance, see \cite[Theorem 10.3]{atnip1} and \cite[Theorem 1.11.2]{atnip3}.
\begin{corollary}\label{coro6.42}
    Suppose the hypotheses hold in Theorem \ref{thm6.41}, then there exists a measurable finite function $D_f:\Omega\to(0,\infty)$ and $\kappa\in(\iota,1)$ such that for $\prob$-a.e. $\omega\in\Omega$, for all $n\geq1$ and all $|p|\leq n$, we have \[\supnorm{\fullnormopera_{\theta^{p}\omega}^{n}f_{\theta^{p}\omega}-\mu_{\theta^{p}\omega}(f_{\theta^{p}\omega})\ind_{\theta^{p+n}\omega}}\leq D_{f}(\omega)\bvnorm{f_{\theta^{p}\omega}}\kappa^{n}.\]
\end{corollary}
\begin{theorem}\label{thm6.43}
    For $\prob$-a.e. $\omega\in\Omega$, for all $n\geq1$, all $|p|\leq n$, all $f\in BV_{\Omega}(I)$ with \[\supnorm{f_\omega},\log\vari(f_\omega),\log\inf|f_\omega|\in L^1(\prob)\] for every $\omega\in\Omega$ and all $g\in L^1(\mu)$, we have 
    \begin{equation*}
        \begin{aligned}
            &|\mu_{\theta^{p}\omega}\left(f_{\theta^{p}\omega}(g_{\theta^{p+n}\omega}\circ T_{\theta^{p}\omega}^{n})\right)-\mu_{\theta^{p}\omega}(f_{\theta^{p}\omega})\mu_{\theta^{p+n}\omega}(g_{\theta^{p+n}\omega})| \\
            &\leq D_f(\omega)\bvnorm{f_{\theta^{p}\omega}}\Vert g_{\theta^{p+n}\omega}\Vert_{L^1(\mu_{\theta^{p+n}\omega})}\kappa^n
        \end{aligned}
    \end{equation*} where $D_f(\omega)$ and $\kappa$ are from Corollary \ref{coro6.42}.
\end{theorem}
\begin{proof}
    Let $\tilde{f}_{\omega}:=f_{\omega}-\mu_{\omega}(f_\omega)$ for all $\omega\in\Omega$. We rewrite
    \begin{equation*}
        \begin{aligned}
            &\,\,\mu_{\theta^{p}\omega}\left(f_{\theta^{p}\omega}(g_{\theta^{p+n}\omega}\circ T_{\theta^{p}\omega}^{n})\right)-\mu_{\theta^{p}\omega}(f_{\theta^{p}\omega})\mu_{\theta^{p+n}\omega}(g_{\theta^{p+n}\omega}) \\
            &=\mu_{\theta^{p+n}\omega}(\fullnormopera_{\theta^{p}\omega}^{n}f_{\theta^{p}\omega}\cdot g_{\theta^{p+n}\omega})-\mu_{\theta^{p}\omega}(f_{\theta^{p}\omega})\mu_{\theta^{p+n}\omega}(g_{\theta^{p+n}\omega}) \\
            &=\mu_{\theta^{p+n}\omega}(\fullnormopera_{\theta^{p}\omega}^{n}\tilde{f}_{\theta^{p}\omega}\cdot g_{\theta^{p+n}\omega}).
        \end{aligned}
    \end{equation*} By Corollary \ref{coro6.42} we have
    \[\supnorm{\fullnormopera_{\theta^{p}\omega}^{n}\tilde{f}_{\theta^{p}\omega}}=\supnorm{\fullnormopera_{\theta^{p}\omega}^{n}f_{\theta^{p}\omega}-\mu_{\theta^{p}\omega}(f_{\theta^{p}\omega})\ind_{\theta^{p+n}\omega}}\leq D_{f}(\omega)\bvnorm{f_{\theta^{p}\omega}}\kappa^n\] for $n\geq1$. Hence,
    \begin{equation*}
        \begin{aligned}
            &\,\,|\mu_{\theta^{p}\omega}\left(f_{\theta^{p}\omega}(g_{\theta^{p+n}\omega}\circ T_{\theta^{p}\omega}^{n})\right)-\mu_{\theta^{p}\omega}(f_{\theta^{p}\omega})\mu_{\theta^{p+n}\omega}(g_{\theta^{p+n}\omega})| \\
            &\leq\mu_{\theta^{p+n}\omega}(|\fullnormopera_{\theta^{p}\omega}^{n}\tilde{f}_{\theta^{p}\omega}\cdot g_{\theta^{p+n}\omega}|) \\
            &\leq\mu_{\theta^{n+p}\omega}(|g_{\theta^{p+n}\omega}|)\supnorm{\fullnormopera_{\theta^{p}\omega}^{n}\tilde{f}_{\theta^{p}\omega}} \\
            &\leq D_f(\omega)\bvnorm{f_{\theta^{p}\omega}}\mu_{\theta^{p+n}\omega}(|g_{\theta^{n+p}\omega}|)\kappa^n \\
            &=D_f(\omega)\bvnorm{f_{\theta^{p}\omega}}\Vert g_{\theta^{p+n}\omega}\Vert_{L^1(\mu_{\theta^{p+n}\omega})}\kappa^n
        \end{aligned}
    \end{equation*} holds for every $n\geq1$ and every $|p|\leq n$.
\end{proof}
\begin{proof}[Proof of Theorem \ref{main2}]
    For the two results, by $\normopera_{\omega}^{n}=(\lambda_{\omega}^{n})^{-1}\opera_{\omega}^{n}$, letting $p=0$ in Theorem \ref{thm6.41} and Theorem \ref{thm6.43} finishes the proof.
\end{proof}
\section{Random probability measure}\label{sec8}
In this section, we review the conformal measures $\{\nu_\omega\}_{\omega\in\Omega}$ produced in Proposition \ref{prop6.35} and clarify they are indeed random probability measures.
\begin{proposition}[{\cite[Proposition 10.4]{atnip1}}]
    The family of Borel probability measures $\nu=\{\nu_\omega\}_{\omega\in\Omega}$ is unique and can be defined by \[\nu_{\omega}(f_\omega):=\lim_{n\to\infty}\frac{\supnorm{\opera_{\omega}^{n}f_\omega}}{\supnorm{\opera_{\omega}^{n}\ind_\omega}}\] for all $f\in BV_{\Omega}(I)$ with $\supnorm{f_\omega},\log\vari(f_\omega),\log\inf|f_\omega|\in L^1(\prob)$ for every $\omega\in\Omega$. As a consequence, the family of Borel probability measures $\mu=\{\mu_\omega\}_{\omega\in\Omega}$ is unique.
\end{proposition}
\begin{proof}
    Suppose $\tilde{\nu}=\{\tilde{\nu}_{\omega}\}_{\omega\in\Omega}$ is another family of Borel probability measures satisfying $\tilde{\nu}_{\theta\omega}(\opera_{\omega}f_\omega)=\tilde{\lambda}_{\omega}\tilde{\nu}_{\omega}(f_\omega)$ where $\tilde{\lambda}_{\omega}:=\tilde{\nu}_{\theta\omega}(\opera_\omega \ind_\omega)$. Since $f$ and $q$ are both tempered, their quotient $f/g$ belongs to class $\mathcal{T}$. Corollary \ref{coro6.42} gives \[\lim_{n\to\infty}\bigg\Vert\fullnormopera_{\omega}^{n}\left(\frac{f_\omega}{q_\omega}\right)\bigg\Vert_{\infty}=\tilde{\nu}_{\omega}(f_\omega).\] By changing expressions, we have \[\lim_{n\to\infty}\frac{\supnorm{\opera_{\omega}^{n}f_\omega}}{\supnorm{\opera_{\omega}^{n}\ind_\omega}}=\frac{\tilde{\nu}_{\omega}(f_\omega)}{\tilde{\nu}_{\omega}(\ind_\omega)}=\tilde{\nu}_{\omega}(f_\omega).\] The limit does not depend on $\tilde{\nu}_\omega$ and $q_\omega$, meaning the family $\{\nu_\omega\}_{\omega\in\Omega}$ is uniquely determined. The uniqueness of $\{\mu_\omega\}_{\omega\in\Omega}$ is also guaranteed since the density function $\{q_\omega\}_{\omega\in\Omega}$ is unique from \cite[Proposition 9.4]{atnip1}.
\end{proof}
The following proposition is following \cite[Proposition 9.3]{atnip1}. The family $\{\nu_\omega\}_{\omega\in\Omega}$ can be raised to a random probability measure with marginal $\prob$.
\begin{proposition}\label{prop6.45}
    The family of Borel probability measures $\{\nu_\omega\}_{\omega\in\Omega}$ gives rise to a random probability measure $\nu\in\probmeas_{\Omega}(\Omega\times I)$ defined by \[\nu(f)=\int_{\Omega}\int_{I}f_{\omega}\,\dd\nu_{\omega}\,\dd\prob(\omega)\] for all $f\in L^1(\nu)$ on $\Omega\times I$. The family of Borel probability measures $\{\mu_\omega\}_{\omega\in\Omega}$ gives rise to a random $T$-invariant probability measure $\mu\in\probmeas_{\Omega}(\Omega\times I)$ defined by \[\mu(f)=\int_{\Omega}\int_{I}f_{\omega}\,\dd\mu_{\omega}\,\dd\prob(\omega)\] for all $f\in L^1(\mu)$ on $\Omega\times I$. Moreover, $\mu$ is absolutely continuous with respect to $\nu$.
\end{proposition}
\begin{proof}
    For $\prob$-a.e. $\omega\in\Omega$, $\nu_{\omega}$ is a Borel probability measure on $I$. For every interval $Z\subseteq I$, the map $\omega\mapsto\nu_{\omega}(Z)$ is measurable since \[\nu_{\omega}(Z):=\nu_{\omega}(\ind_Z)=\lim_{n\to\infty}\frac{\supnorm{\opera_{\omega}^{n}\ind_Z}}{\supnorm{\opera_{\omega}^{n}\ind_\omega}}\] is a limit of measurable functions. For the $\sigma$-algebra $\B$ on $I$, since it is generated by intervals, the map $\omega\mapsto\nu_{\omega}(B)$ is measurable for every $B\in\B$. According to \cite[Proposition 3.3]{crauel}, the disintegration $\{\nu_\omega\}_{\omega\in\Omega}$ defines the random probability measure $\nu$ on $\Omega\times I$ with marginal $\prob$ on $\Omega$. Clearly this result also holds for $\{\mu_\omega\}_{\omega\in\Omega}$. Since $\mu_{\omega}\circ T_{\omega}^{-1}=\mu_{\theta\omega}$ by Proposition \ref{prop6.38}, then $\mu$ is $T$-invariant. The relation $\mu\prec\nu$ is obvious from $\mu_\omega:=q_\omega\nu_\omega$ for every $\omega\in\Omega$.
\end{proof}
The results of decay of correlation implicitly meaning ergodicity of invariant probability measure.
\begin{proposition}\label{prop6.46}
    The random probability measure $\mu$ is ergodic.
\end{proposition}
\begin{proof}
    We refer to the proof of \cite[Proposition 4.7]{mayer} on distance expanding random maps. Let $A\subseteq\Omega\times I$ be a measurable set with $T^{-1}(A)=A$ where $T:\Omega\times I\to\Omega\times I$ is the skew-product map. For $\omega\in\Omega$, let $A_\omega:=\{x\in X_\omega:(\omega,x)\in A\}$, then $T_{\omega}^{-1}(A_{\theta\omega})=A_\omega$. Let $\tilde{\Omega}:=\{\omega\in\Omega:\mu_{\omega}(A_\omega)>0\}$ be a $\theta$-invariant subset of $\Omega$, we assume $\prob(\tilde{\Omega})>0$.

    We define a function $f$ on $\Omega\times I$ by $f_{\omega}:=\ind_{A_\omega}$ for every $\omega\in\Omega$, then $f_\omega\in L^1(\mu_\omega)$ and $f_{\theta^{n}\omega}\circ T_{\omega}^{n}=f_\omega$ for $\prob$-a.e. $\omega\in\Omega$. Let $\hat{\Omega}\subseteq\Omega$ be a subset with full measure such that Theorem \ref{thm6.43} holds. For $\omega\in\tilde{\Omega}\cap\hat{\Omega}$, let $g_\omega\in L^1(\mu_\omega)$ with $\mu_{\omega}(g_\omega)=0$, apply Theorem \ref{thm6.43} with $p=0$ and let $n\to\infty$, we see \[\lim_{n\to\infty}|\mu_{\omega}((f_{\theta^{n}\omega}\circ T_{\omega}^{n})g_{\omega})|=0.\] Therefore, $\int_{A_\omega}g_{\omega}\,\dd\mu_\omega=0$ holds for every $g_{\omega}\in L^1(\mu_\omega)$, meaning $\mu_{\omega}(A_\omega)=1$ for all $\omega\in\tilde{\Omega}\cap\hat{\Omega}$, which leads to $\mu(A)=1$. According to the definition of ergodic map, we see the skew-product map $T$ is ergodic and as a consequence, $\mu$ is ergodic measure.
\end{proof}
\begin{proof}[Proof of Theorem \ref{main1}]
    The existence of probability measures $\{\nu_\omega\}_{\omega\in\Omega}$ comes from Proposition 
    \ref{prop6.34} and Proposition \ref{prop6.38}, while the density functions $\{q_\omega\}_{\omega\in\Omega}$ and their corresponding positive eigenvalues $\{\lambda_\omega\}_{\omega\in\Omega}$ are produced from Proposition \ref{prop6.31}. The positivity of density function is proved in Proposition \ref{prop6.33}. The invariance of $\{\mu_\omega\}_{\omega\in\Omega}$ is from Proposition \ref{prop6.38}.

    Proposition \ref{prop6.45} says the families $\{\nu_\omega\}_{\omega\in\Omega}$ and $\{\mu_\omega\}_{\omega\in\Omega}$ are unique and can be raised to random probability measures $\nu,\mu$ respectively. Proposition \ref{prop6.46} shows $\mu$ is ergodic.
\end{proof}
\section{Expected pressure, escape rate, and conditionally invariant measure}\label{sec9}
During this section, we investigate the characteristics of random open systems.
\subsection{Expected pressure}\label{sec9.1}
Given a summable contracting (with respect to the open system) potential $\varphi_c:\Omega\times I\to\R\cup\{-\infty\}$, it generates the open potential $\varphi:=\varphi_{c}|_{J}$ in the random open system. However, for the closed and open operators, we associate them with the same potential $\varphi_c$, so the open potential $\varphi$ here does not relate to the open operator directly. We define and compare the expected pressures of the closed and open potentials, by studying two different eigenvalues produced in the closed and open systems; see \cite[Definition 1.12.1]{atnip3}. The expected pressures of potentials $\varphi_c,\varphi$ are defined respectively by
$$\pres{\varphi_c}:=\int_{\Omega}\log\lambda_{\omega,c}\,\dd\prob(\omega),\quad\pres{\varphi}:=\int_{\Omega}\log\lambda_{\omega}\,\dd\prob(\omega).$$ They are well defined and finite since $\log\lambda_{\omega,c},\log\lambda_{\omega}\in L^1(\prob)$. It also implies that $$\pres{\varphi_c}=\lim_{n\to\infty}\frac{1}{n}\log\lambda_{\omega,c}^{n},\quad\pres{\varphi}=\lim_{n\to\infty}\frac{1}{n}\log\lambda_{\omega}^{n}$$ by Birkhoff's ergodic theorem.

\begin{proposition}\label{prop6.47}
    For $\prob$-a.e. $\omega\in\Omega$, we have \[\lim_{n\to\infty}\bigg\Vert\frac{1}{n}\log\opera_{\omega}^{n}\ind_{\omega}-\frac{1}{n}\log\lambda_{\omega}^{n}\bigg\Vert_{\infty}=0.\]
\end{proposition}
\begin{proof}
    We follow \cite[Lemma 10.1]{atnip1} and \cite[Lemma 1.12.2]{atnip3} to take the same procedures. Since $\opera_{\omega}^{n}=\lambda_{\omega}^{n}\normopera_{\omega}^{n}$ for every $n\geq1$, it only needs to show \[\lim_{n\to\infty}\frac{1}{n}\supnorm{\log\normopera_{\omega}^{n}\ind_\omega}=0.\] We claim that $\inf_{S(n,\infty)}q_{\theta^{n}\omega}$ and $\supnorm{q_{\theta^{n}\omega}}$ are tempered for all $n\geq1$. Lemma \ref{lem6.40} implies \[\underset{S(n,\infty)}{\inf}\,q_{\theta^{n}\omega}\geq\underset{S(n,\infty)}{\inf}\,\normopera_{\omega}^{n}\ind_{\omega}-C_{\ind}(\omega)\iota^n\] for all sufficiently large $n\geq1$, further implying \[\frac{1}{n}\log\underset{S(n,\infty)}{\inf}\,q_{\theta^{n}\omega}\geq\frac{1}{n}\log(\underset{S(n,\infty)}{\inf}\,\normopera_{\omega}^{n}\ind_{\omega}-C_{\ind}(\omega)\iota^n).\] Since $\log\inf_{S(n,\infty)}\normopera_{\omega}^{n}\ind_{\omega}\in L^1(\prob)$, then it is tempered. The term $C_{\ind}(\omega)\iota^n$ is finite and it will be eliminated in the limit. Combine them together, we can see the limit satisfies \[\lim_{n\to\infty}\log\underset{S(n,\infty)}{\inf}\,q_{\theta^{n}\omega}\geq-\epsilon\] for sufficiently small $\epsilon>0$, meaning \[\lim_{n\to\infty}\log\underset{S(n,\infty)}{\inf}\,q_{\theta^{n}\omega}=0.\]  Regarding the supremum norm, we also obtain the same limit, then the claim is finished. We observe \[\frac{1}{n}\log(\underset{S(n,\infty)}{\inf}\,q_{\theta^{n}\omega}-C_{\ind}(\omega)\iota^n)\leq\frac{1}{n}\log\normopera_{\omega}^{n}\ind_{\omega}\leq\frac{1}{n}\log(\supnorm{q_{\theta^{n}\omega}}+C_{\ind}(\omega)\iota^n),\] both sides are tending to zero as $n\to\infty$ and we have the desired limit.
\end{proof}
Since $\opera_{\omega}\ind_{\omega}\leq\opera_{\omega,c}\ind_{X_\omega}$ for each $\omega\in\Omega$, by Proposition \ref{prop6.47} we have $\pres{\varphi}\leq\pres{\varphi_c}$. That is, the expected pressure of open potential cannot exceed that of the original one.
\subsection{Escape rate}
We import the concept of escape rate from \cite[Section 3.1]{poll} into random interval map with holes, for more details, see \cite[Section 12]{atnip3}.
\begin{defn}\label{def9.1}
    Given a hole $H\subseteq\Omega\times I$ and a random probability measure $\nu\in\probmeas_{\Omega}(\Omega\times I)$, for each $\omega\in\Omega$, the \emph{lower fiberwise escape rate} for $H$ is defined by \[\underline{R}_{\nu_{\omega}}(H):=-\limsup_{n\to\infty}\frac{1}{n}\log\nu_{\omega}(K_{\omega,n}),\] the \emph{upper fiberwise escape rate} for $H$ is defined by \[\overline{R}_{\nu_{\omega}}(H):=-\liminf_{n\to\infty}\frac{1}{n}\log\nu_{\omega}(K_{\omega,n}).\] If $\underline{R}_{\nu_{\omega}}(H)=\overline{R}_{\nu_{\omega}}(H)$, we say the \emph{escape rate} for $H$ exists and it is denoted by $R_{\nu_{\omega}}(H)$.
\end{defn}
\begin{proposition}[{\cite[Proposition 1.12.4]{atnip3}}]\label{prop6.49}
    Given a hole $H\subseteq X$, for $\prob$-a.e. $\omega\in\Omega$ we have \[R_{\nu_{\omega,c}}(H)=\pres{\varphi_c}-\pres{\varphi}.\]
\end{proposition}
\begin{proof}
    We add more details into the proof. Let $\nu_c$ be the determined random probability measure of the closed system. For every $n\geq1$, we have
    \begin{equation*}
        \begin{aligned}
            \nu_{\omega,c}(K_{\omega,n-1})&=\frac{1}{\lambda_{\omega,c}^{n}}\nu_{\theta^{n}\omega,c}(\opera_{\omega,c}^{n}\ind_{K_{\omega,n-1}}) \\
            &=\frac{\lambda_{\omega}^{n}}{\lambda_{\omega,c}^{n}}\nu_{\theta^{n}\omega,c}(\normopera_{\omega}^{n}\ind_\omega) \\
            &=\frac{\lambda_{\omega}^{n}}{\lambda_{\omega,c}^{n}}\left(\nu_{\theta^{n}\omega,c}(q_{\theta^{n}\omega})-\nu_{\theta^{n}\omega,c}(\normopera_{\omega}^{n}\ind_{\omega}-q_{\theta^{n}\omega})\right).
        \end{aligned}
    \end{equation*} By Theorem \ref{thm6.41}, as \[\supnorm{\normopera_{\omega}^{n}\ind_\omega-\nu_{\omega}(\ind_\omega)q_{\theta^{n}\omega}}\leq\hat{C}_{\ind}(\omega)\bvnorm{\ind_\omega}\iota^{n}\] for measurable finite $\hat{C}_{\ind}(\omega)$ and $\iota\in(0,1)$, let $n\to\infty$ then $\nu_{\theta^{n}\omega,c}(\normopera_{\omega}^{n}\ind_{\omega}-q_{\theta^{n}\omega})\to0$ exponentially fast. Therefore, \[\limsup_{n\to\infty}\frac{1}{n}\log\nu_{\omega,c}(K_{\omega,n-1})=\liminf_{n\to\infty}\frac{1}{n}\log\nu_{\omega,c}(K_{\omega,n-1})=\lim_{n\to\infty}\frac{1}{n}\log\frac{\lambda_{\omega}^{n}}{\lambda_{\omega,c}^{n}}+\lim_{n\to\infty}\frac{1}{n}\log\nu_{\theta^{n}\omega,c}(q_{\theta^{n}\omega}).\] Since $\inf_{S(n,\infty)}q_{\omega}$ and $\supnorm{q_{\theta^{n}\omega}}$ are both tempered, it implies \[0=\lim_{n\to\infty}\frac{1}{n}\log\underset{S(n,\infty)}{\inf}\,q_{\theta^{n}\omega}\leq\lim_{n\to\infty}\frac{1}{n}\log\nu_{\theta^{n}\omega,c}(q_{\theta^{n}\omega})\leq\lim_{n\to\infty}\frac{1}{n}\log\supnorm{q_{\theta^{n}\omega}}=0,\] hence \[\lim_{n\to\infty}\frac{1}{n}\log\nu_{\theta^{n}\omega,c}(q_{\theta^{n}\omega})=0.\] By the definition of expected pressure we obtain \[R_{\nu_{\omega,c}}(H)=-\lim_{n\to\infty}\frac{1}{n}\log\frac{\lambda_{\omega}^{n}}{\lambda_{\omega,c}^{n}}=\lim_{n\to\infty}\frac{1}{n}\log\lambda_{\omega,c}^{n}-\lim_{n\to\infty}\frac{1}{n}\log\lambda_{\omega}^{n}=\pres{\varphi_c}-\pres{\varphi}.\]
\end{proof}
\subsection{Conditionally invariant measure}
We introduce conditionally invariant measure to random open dynamical systems. Recall $K_{\omega,n}$ is $(\omega,n)$-survivor set that \[K_{\omega,n}:=\bigcap_{i=0}^{n}T_{\omega}^{-i}(J_{\theta^{i}\omega})\] for $J_{\theta^{i}\omega}=I_{\theta^{i}\omega}\setminus H_{\theta^{i}\omega}$.
\begin{defn}[{\cite[Definition 1.2.1]{atnip3}}]
    A random probability measure $\tau\in\probmeas_{\Omega}(\Omega\times I)$ is a \emph{random conditionally invariant probability measure} if \[\tau_{\omega}(K_{\omega,n}\cap T_{\omega}^{-n}(A))=\tau_{\omega}(K_{\omega,n})\tau_{\theta^{n}\omega}(A)\] for all $n\geq0$, all $\omega\in\Omega$ and all Borel subsets $A\subseteq I$.
\end{defn}
If $\tau$ is random conditionally invariant measure, set $A=K_{\theta^{n}\omega,m}$ here, then $K_{\omega,n}\cap T_{\omega}^{-n}(A)=K_{\omega,n+m}$ for all $m,n\geq1$. Therefore, we have \[\tau_{\omega}(K_{\omega,n+m})=\tau_{\omega}(K_{\omega,n})\tau_{\theta^{n}\omega}(K_{\theta^{n}\omega,m}).\] By this property, we obtain \[\tau_{\omega}(K_{\omega,n})=\prod_{i=0}^{n-1}\tau_{\theta^{i}\omega}(K_{\theta^{i}\omega,1}).\] Let $c_{\omega}:=\tau_{\omega}(K_{\omega,1})$, thus $\tau_{\omega}(K_{\omega,n})=c_{\omega}^{n}:=\prod_{i=0}^{n-1}c_{\theta^{i}\omega}$. Furthermore, let $A=J_{\omega}=K_{\omega,0}$ and $n=0$, we observe \[\tau_{\omega}(J_\omega)=\tau_{\omega}(J_\omega)^{2},\] implying $\tau_{\omega}(J_\omega)$ is $0$ or $1$.
\begin{lemma}{\cite[Lemma 1.2.4]{atnip3}}\label{lem6.51}
    For $f,g\in BV(I)$ and all $\omega\in\Omega$, for all $n\geq1$ we have \[\int_{J_{\theta^{n}\omega}}f\cdot\opera_{\omega}^{n}g\,\dd\nu_{\theta^{n}\omega,c}=\lambda_{\omega,c}^{n}\int_{K_{\omega,n}}(f\circ T_{\omega}^{n})\cdot g\,\dd\nu_{\omega,c}.\]
\end{lemma}
We say a random conditionally invariant probability measure $\tau$ is absolutely continuous with respect to a random probability measure $\tau'$ if $\tau\prec\tau'$. We modify \cite[Lemma 1.1]{liverani4} into the random version and obtain the following lemma, giving an alternative definition.
\begin{lemma}\label{lem6.52}
    Given a function $f\in BV_{\Omega}(I)$, let $\tau:=(\ind_{J}f)\nu_c$ be a random probability measure on $J$ and assume $\tau\prec\nu_c$, then $\tau$ is a random conditionally invariant probability measure absolutely continuous with respect to $\nu_c$ if and only if there exists $c_\omega\in(0,1]$ such that \[\opera_{\omega}f_{\omega}=c_{\omega}\lambda_{\omega,c}f_{\theta\omega}\] holds for every $\omega\in\Omega$.
\end{lemma}
\begin{proof}
    Changing the corresponding symbols in \cite[Lemma 1.1]{liverani4} and adding the state $\omega$ to these subscripts, together with Lemma \ref{lem6.51}, completes the proof.
\end{proof}
\begin{proposition}[{\cite[Lemma 1.2.5]{atnip3}}]\label{prop6.53}
    The random probability measure $\tau\in\probmeas_{\Omega}(\Omega\times I)$ whose disintegration $\tau_{\omega}:=\ind_{\omega}q_{\omega}\nu_{\omega,c}$ is the unique random conditionally invariant probability measure absolutely continuous with respect to $\nu_c$ and $\supp(\tau)\subseteq J$.
\end{proposition}
\begin{proof}
    For all $\omega\in\Omega$ and all subsets $A_\omega\subseteq I$ with $\nu_{\omega,c}(A_\omega)=0$, we immediately obtain $\tau_{\omega}(\ind_{A_\omega})=0$ and hence $\tau$ is absolutely continuous with respect to $\nu_c$. Since $\nu_c$ is supported in $J$, we also have $\supp(\tau)\subseteq J$. The uniqueness follows from that of the random probability measure $\nu_c$.

    Now we verify $\{\tau_\omega\}_{\omega\in\Omega}$ is random conditionally invariant probability measure. By Lemma \ref{lem6.52}, for $n\geq1$ we have
    \begin{equation*}
        \begin{aligned}
            \tau_{\omega}(K_{\omega,n}\cap T_{\omega}^{-n}(A))&=\int_{I}\ind_{T_{\omega}^{-n}(A)}\ind_{K_{\omega,n}}\,\dd\tau_{\omega} \\
            &=\int_{K_{\omega,n}}(\ind_{A}\circ T_{\omega}^{n})\cdot q_{\omega}\,\dd\nu_{\omega,c} \\
            &=(\lambda_{\omega,c}^{n})^{-1}\int_{J_{\theta^{n}\omega}}\ind_{A}\cdot\opera_{\omega}^{n}q_{\omega}\,\dd\nu_{\omega,c} \\
            &=\frac{\lambda_{\omega}^{n}}{\lambda_{\omega,c}^{n}}\int_{J_{\theta^{n}\omega}\cap A}q_{\theta^{n}\omega}\,\dd\nu_{\theta^{n}\omega} \\
            &=\frac{\lambda_{\omega}^{n}}{\lambda_{\omega,c}^{n}}\int_{J_{\theta^{n}\omega}\cap A}\,\dd\tau_{\theta^{n}\omega} \\
            &=\frac{\lambda_{\omega}^{n}}{\lambda_{\omega,c}^{n}}\tau_{\theta^{n}\omega}(J_{\theta^{n}\omega}\cap A).
        \end{aligned}
    \end{equation*} Let $A=J_{\theta^{n}\omega}$ on both sides of the equation. Then we see \[\tau_{\omega}(K_{\omega,n}\cap T_{\omega}^{-n}(J_{\theta^{n}\omega}))=\frac{\lambda_{\omega}^{n}}{\lambda_{\omega,c}^{n}}\tau_{\theta^{n}\omega}(J_{\theta^{n}\omega}).\] Since $K_{\omega,n}$ is included in $T_{\omega}^{-n}(J_{\theta^{n}\omega})$, and in addition $\tau_{\theta^{n}\omega}(J_{\theta^{n}\omega})=1$ because the measure $\tau_\omega$ is supported in $J_\omega$, we obtain \[\tau_{\omega}(K_{\omega,n})=\frac{\lambda_{\omega}^{n}}{\lambda_{\omega,c}^{n}}.\] The right side can be denoted by \[c_{\omega}^{n}=\prod_{i=0}^{n-1}\frac{\lambda_{\theta^{i}\omega}}{\lambda_{\theta^{i}\omega,c}}\in(0,1],\] since $0<\lambda_{\theta^{i}\omega}\leq\lambda_{\theta^{i}\omega,c}$ for $0\leq i\leq n-1$. Therefore, by Lemma \ref{lem6.52} we deduce that $\tau:=\{\tau_\omega\}_{\omega\in\Omega}$ is a random conditionally invariant measure.
\end{proof}
\begin{remark}
    Indeed, we recall $\opera_{\omega}q_{\omega}=\lambda_{\omega}q_{\theta\omega}$ from Proposition \ref{prop6.31}; rewriting it as $\opera_{\omega}q_{\omega}=(\lambda_{\omega,c}^{-1}\lambda_{\omega})\lambda_{\omega,c}q_{\theta\omega}$ and using Lemma \ref{lem6.52} finishes the proof immediately.
\end{remark}
Summarizing all results in this section, we finish the proof of the third main theorem.
\begin{proof}[Proof of Theorem \ref{main3}]
    The calculation of the escape rate $R_{\nu_{\omega,c}}(H)$ is directly from Proposition \ref{prop6.49}. The statement of the random conditionally invariant probability measure is collected from Proposition \ref{prop6.53}.
\end{proof}
\addtocontents{toc}{\protect\setcounter{tocdepth}{-1}}
\subsection*{Acknowledgements}
\addtocontents{toc}{\protect\setcounter{tocdepth}{1}}
The author would like to thank Associate Professor Johannes Jaerisch, for his guidance and encouragement throughout the draft of the paper. The author would also like to thank all other members at the `Ergodic Theory and Fractal Geometry for Conformal Dynamical Systems' seminar for many enthusiastic discussions and for providing a supportive research environment.
\bibliography{reference}
\bibliographystyle{alpha}
\end{document}